\documentclass[12pt]{amsart}
\usepackage{a4wide,graphicx,amsfonts,amssymb,stmaryrd,amscd,amsmath,latexsym,amsbsy,bm}
\numberwithin{equation}{section}
\theoremstyle{plain}

\newtheorem{thm}{Theorem}[section]

\newtheorem{prop}[thm]{Proposition}

\newtheorem{defi}[thm]{Definition}
\newtheorem{lem}[thm]{Lemma}
\newtheorem{cor}[thm]{Corollary}

\theoremstyle{remark}
\newtheorem{rema}[thm]{Remark}

\title[Connection problems]{Connection problems for quantum affine KZ equations and integrable lattice models}
\author{Jasper V. Stokman}
\address{KdV Institute for Mathematics, University of Amsterdam,
Science Park 904, 1098 XH Amsterdam, The Netherlands \& IMAPP, 
Radboud University, Heyendaalseweg 135, 6525 AJ Nijmegen, The Netherlands.}
\email{j.v.stokman@uva.nl}
\subjclass[2000]{}
\begin{document}
\begin{abstract}
Cherednik attached to an affine Hecke algebra module a compatible system of difference equations, called quantum affine Knizhnik-Zamolodchikov (KZ) equations.
In case of a principal series module we construct a basis of power series solutions of the quantum affine KZ equations.
Relating the bases for different asymptotic sectors gives rise to a Weyl group cocycle, which we compute explicitly
in terms of theta functions. 

For the spin representation of the affine Hecke algebra of type $C$
the quantum affine KZ equations become the boundary qKZ equations associated to the Heisenberg spin-$\frac{1}{2}$
XXZ chain.
We show that in this special case the
results lead to an explicit $4$-parameter family of elliptic solutions of the dynamical reflection equation associated to Baxter's $8$-vertex face dynamical $R$-matrix. We use these solutions to define an explicit $9$-parameter elliptic family of boundary quantum Knizhnik-Zamolodchikov-Bernard (KZB) equations.
\end{abstract}
\maketitle
\setcounter{tocdepth}{2}
\tableofcontents

\section{Introduction}
In \cite{CQKZ} Cherednik associates to an abstract affine $R$-matrix $\{\mathcal{R}_\alpha\}_\alpha$, labelled by the roots $\alpha$ of an affine root system,
a compatible system of equations called quantum affine KZ equations.
For type A affine $R$-matrices can be constructed using the braiding of quantum affine algebras. The resulting quantum affine KZ equations become the 
Frenkel-Reshetikhin-Smirnov qKZ equations \cite{Sm,FR}. For type C affine $R$-matrices arise naturally in the context of integrable lattice models with boundaries.
The corresponding 
quantum affine KZ equations are called
boundary qKZ equations. In this case quantum affine symmetric pairs
\cite{Kolb} produce examples, see, e.g., \cite{DM,BK}. 
Besides the quantum group approach, which always leads to quantum affine KZ equations of classical type, one can attach quantum affine KZ equations to affine Hecke algebra modules.
The associated affine $R$-matrices are constructed
using the affine intertwiners of the 
double affine Hecke algebra \cite{CQInd}.

In special cases the affine $R$-matrices can be obtained from both the quantum group and the Hecke algebra construction. The underlying actions of the quantum group and Hecke algebra are related by quantum Schur-Weyl type dualities. 

In this paper we construct bases of power series solutions of quantum affine KZ equations associated to principal series modules of the affine Hecke algebra.
The bases depend on asymptotic sectors, which are given in terms of Weyl chambers of the underlying root system. 
The connection matrices relating the bases for different asymptotic sectors give rise to a Weyl group cocycle.
We solve the connection problem by deriving an explicit expression of the Weyl group cocycle in terms of theta functions.
In addition we study the applications to integrable lattice models with boundaries in detail.

The connection to
quantum integrable models is threefold. Firstly, for arbitrary root systems the difference Cherednik-Matsuo correspondence \cite{CQInd,S,S2} gives 
a bijective correspondence between the solutions of the quantum affine KZ equations associated to principal series modules
and solutions of the spectral problem of Ruij\-se\-naars-Macdonald-Koornwinder-Che\-red\-nik (RMKC) difference operators.
The RMKC difference operators are the Hamiltonians of the quantum relativistic trigonometric Calogero-Moser system, first introduced by Ruij\-se\-naars \cite{Ru} for type A. These quantum relativistic integrable models are intimately related to Macdonald-Koornwinder polynomials \cite{ChBook}. This point of view was emphasized in \cite{S2}. 

Secondly, special cases of the quantum affine KZ equations for principal series modules arise as compatibility conditions for correlation functions
and form factors of (semi-)infinite XXZ spin chains \cite{JM,JKKMW,W}. In these cases vertex operator and algebraic Bethe ansatz techniques have been employed to
construct solutions of the quantum affine KZ equations (see, e.g.,  \cite{FR,TV,Re,JM} and references therein for type A, 
and \cite{JKKMW,W,DF,DFZJ,Ka,dGP,Koj,RSV,SV,BK} for type C). These methods have the drawback that they require additional assumptions on the associated
XXZ spin chains, such as the existence of pseudo-vacuum vectors. Our construction of power series solutions of the quantum affine KZ equations is applicable without such
restrictions.

Thirdly, the solution of the connection problem
is an explicit Weyl group cocycle expressed in terms of theta functions. For classical type and for
special classes of principal series modules the cocycle
gives rise to elliptic solutions of dynamical quantum Yang-Baxter and reflection equations. 
  
In Section 3 we treat the general theory, but we start in Section \ref{section2} with a detailed discussion of 
 the applications 
  to integrable lattice models with boundaries.
 Applications to other classes of integrable lattice models, such as quasi-periodic lattice models with quantum supersymmetry or higher spin models, 
will be addressed in future work.
The crucial step in applying the general theory to integrable lattice models
is to decompose the affine Hecke algebra module underlying the pertinent integrable lattice model 
as a direct sum of principal series modules. For each principal series component the general theory from Section 3 can be applied
to obtain an associated explicit connection cocycle of elliptic type.

We treat in full detail the spin representation of the affine Hecke algebra of type C \cite{dGN,SV} 
encoding the affine Hecke algebra symmetries of the Heisenberg spin-$\frac{1}{2}$ XXZ chain with arbitrary reflecting boundary conditions on both ends. 
The associated quantum affine KZ equations are boundary qKZ equations
depending on $6$ parameters: a representation parameter, a bulk parameter, and two free parameters for each end of the XXZ spin chain with reflecting ends. The associated $K$-matrices, encoding the integrability at the boundary, are those first obtained in \cite{dVG} by direct computations. We explicitly apply the general theory from Section 3 to this context, resulting in the construction of bases of power series solutions of the boundary qKZ equations (Theorem \ref{asymptoticspin}) and the description of the associated connection matrices in terms of explicit elliptic solutions of dynamical quantum Yang-Baxter and reflection equations (Theorem \ref{Mwc}). 
It leads to a $4$-parameter family of solutions of the dynamical reflection equation with respect to Baxter's \cite{Ba} dynamical elliptic $R$-matrix, with the
representation parameter playing the role of the dynamical parameter. It can be reinterpreted as a $4$-parameter family of solutions of the 
boundary Yang-Baxter equation for Baxter's eight vertex face model. 
Solutions to the boundary Yang-Baxter equations for the eight vertex face model, or equivalently, by vertex-face correspondences, to reflection equations for the eight vertex
model itself, have been computed by direct computations before; see \cite{dVG, IK, HSFY} from the vertex perspective, \cite{FHS,BP,KH} from the face perspective, and \cite{HK} for the explicit link by vertex-face correspondences. 

The current paper thus provides a conceptual understanding of such solutions and of their free parameters. In fact, 
by the difference Cherednik-Matsuo correspondence \cite{CQInd,S,SV} the solutions of the boundary qKZ equations associated to the spin representation
correspond to solutions of the spectral problem of the (higher order) Koornwinder operators \cite{Kopap} for a special family of spectral points which are
naturally parametrized by the representation parameter of the spin representation. It results in the interpretation of the remaining 5 free parameters of the boundary qKZ equations as the free parameters in the theory of
Koornwinder polynomials (see \cite{SV} for a detailed discussion). The four free parameters in the elliptic solutions of the dynamical
reflection equation obtained from the computation of the associated connection problem
then correspond to the Askey-Wilson parameters \cite{S2} in the theory of Koornwinder polynomials \cite{Kopap}.

In \cite{Fe2,FTV}  solutions of dynamical quantum Yang-Baxter equations are used to define quantum analogs of the Knizhnik-Zamolodchikov-Bernard (KZB) equations. 
In the final part of Section \ref{section2} we construct quantum analogs of the KZB equations in the presence of boundaries. 
These boundary quantum KZB equations are defined in terms of solutions of the dynamical quantum Yang-Baxter equation and associated solutions of the left and right dynamical reflection equations.
Applied to Baxter's elliptic dynamical R-matrix
and the associated 4-parameter family of solutions of the dynamical reflection equation, we obtain
an explicit $9$-parameter elliptic family of boundary quantum Knizhnik-Zamolodchikov-Bernard (KZB) equations. 
We expect that a generalisation of the difference Cherednik-Matsuo correspondence relates their solutions to the solutions of the spectral
problem of the $9$-parameter family of the elliptic Ruijsenaars' systems of type $C$ as introduced by van Diejen \cite{vD} and Komori \& Hikami \cite{KH0}.

In Section 3 we construct the bases of power series solutions of the quantum affine KZ equations associated to principal series modules.
The explicit expressions of the associated connection matrices
in terms of theta functions are derived
using results of \cite{S2} (which dealt with the case of minimal principal series modules).
In the last part of Section 3 we explain how these results, when applied to the spin representation of the affine Hecke algebra of type $C$, 
produces the results on integrable lattice models with boundaries as described in Section 2.
 
\vspace{.3cm}
\noindent
{\bf Acknowledgments:} I thank Ivan Cherednik, Masatoshi Noumi, Nicolai Reshetikhin and Bart Vlaar for discussions and Yasushi Komori for providing me with
a copy of \cite{HK}.

\section{The connection problem for the spin-$\frac{1}{2}$ XXZ boundary qKZ equations}\label{section2}
\subsection{The spin representation}\label{spinsec}
Let $n\geq 2$ and $0<q<1$. We fix a basis $\{v_+,v_-\}$ of $\mathbb{C}^2$ and represent linear operators on $\mathbb{C}^2$ and 
$\mathbb{C}^2\otimes\mathbb{C}^2$ as matrices with respect to the ordered basis $(v_+,v_-)$ and $(v_+\otimes v_+,v_+\otimes v_-,v_-\otimes v_+,v_-\otimes v_-)$ respectively.

Let $S_n$ be the symmetric group in $n$ letters and write $W_0=S_n\ltimes (\pm 1)^n$ for the hyperoctahedral group.
Let $s_1,\ldots,s_n$ be simple reflections of $W_0$
satisfying $s_i^2=e$ and the braid relations $s_is_{i+1}s_i=s_{i+1}s_is_{i+1}$ ($1\leq i<n$),
$s_{n-1}s_ns_{n-1}s_n=s_ns_{n-1}s_ns_{n-1}$ and $s_is_j=s_js_i$ if $|i-j|>1$, with $e$ the neutral element of $W_0$.
The hyperoctahedral group $W_0$ acts on $\mathbb{C}^n$ by
\begin{equation*}
\begin{split}
s_i\mathbf{z}&=(z_1,\ldots,z_{i-1},z_{i+1},z_i,z_{i+2},\ldots,z_n),\qquad 1\leq i<n,\\
s_n\mathbf{z}&=(z_1,\ldots,z_{n-1},-z_n),
\end{split}
\end{equation*}
where $\mathbf{z}=(z_1,\ldots,z_n)$. Note that $\mathbb{Z}^n$ is $W_0$-stable.

The affine Weyl group of type $C_n$
is $W:=W_0\ltimes\mathbb{Z}^n$.
We write $\lambda\mapsto\tau(\lambda)$ for the canonical group embedding $\mathbb{Z}^n\hookrightarrow W$. The action of $W_0$ on $\mathbb{Z}^n$
(resp. $\mathbb{C}^n$) extends to an action of $W$ by
\[
\tau(\lambda)\mathbf{z}:=\mathbf{z}+\lambda=(z_1+\lambda_1,\ldots,z_n+\lambda_n),\qquad \lambda=(\lambda_1,\ldots,\lambda_n)\in\mathbb{Z}^n.
\]
Let $\{e_i\}_{i=1}^n$ be the standard orthonormal basis of $\mathbb{R}^n$ and set
\[
s_0:=\tau(e_1)s_1\cdots s_{n-1}s_ns_{n-1}\cdots s_1.
\]
Then
\[
s_0\mathbf{z}:=(1-z_1,z_2,\ldots,z_n)
\]
and $W$ is a Coxeter group with simple reflections $s_0,s_1,\ldots,s_n$. 
The associated Coxeter graph is
\begin{figure}[htbp]
\begin{center}
\includegraphics[width=7cm, height=1.5cm]{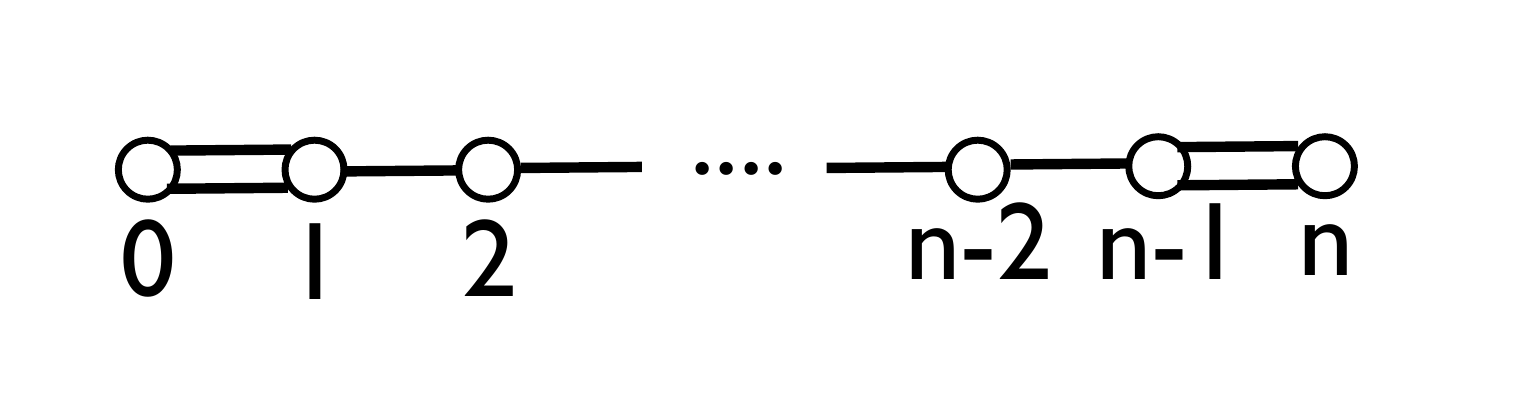}
\end{center}
\end{figure}

Note that
\begin{equation}\label{taui}
\tau(e_i)=s_{i-1}\cdots s_1s_0s_1\cdots s_{n-1}s_ns_{n-1}\cdots s_i
\end{equation}
for $i=1,\ldots,n$.

Fix parameters $\zeta^\prime,\kappa,\zeta\in\mathbb{C}$. The affine Hecke algebra $H=H(q^{\zeta^\prime},q^{\kappa},q^{\zeta}) $ of type $C_n$ is the unital associative algebra over 
$\mathbb{C}$ with generators $T_0,\ldots,T_n$ that satisfy the braid relations associated to the above Coxeter graph
and satisfy the Hecke relations
\begin{equation*}
\begin{split}
(T_0-q^{-\zeta^\prime})(T_0+q^{\zeta^\prime})&=0,\\
(T_i-q^{-\kappa})(T_i+q^{\kappa})&=0,\qquad 1\leq i<n,\\
(T_n-q^{-\zeta})(T_n+q^{\zeta})&=0.
\end{split}
\end{equation*}
The spin representation \cite{dGN,SV} of $H$ is defined as follows.
\begin{prop}\label{spinprop}
There exists a unique representation 
$\pi^{sp}_{\xi,\xi^\prime}: H\rightarrow\textup{End}_{\mathbb{C}}\bigl(\bigl(\mathbb{C}^2\bigr)^{\otimes n}\bigr)$
such that
\begin{equation*}
\pi^{sp}_{\xi,\xi^\prime}(T_i)=\left(\begin{matrix} q^{-\kappa} & 0 & 0 & 0\\
0 & 0 & 1 & 0\\
0 & 1 & q^{-\kappa}-q^{\kappa} & 0\\
0 & 0 & 0 & q^{-\kappa}\end{matrix}\right)_{i,i+1}
\end{equation*}
for $1\leq i<n$ and 
\begin{equation*}
\pi^{sp}_{\xi,\xi^\prime}(T_0)=\left(\begin{matrix} q^{-\zeta^\prime}-q^{\zeta^\prime} & q^{-\xi^\prime}\\
q^{\xi^\prime} & 0\end{matrix}\right)_1,\qquad
\pi^{sp}_{\xi,\xi^\prime}(T_n)=\left(\begin{matrix} 0 & q^{\xi}\\ q^{-\xi} & q^{-\zeta}-q^{\zeta}\end{matrix}\right)_n.
\end{equation*}
Here we have used the standard tensor leg notation to indicate on which tensor legs of the $n$-fold
tensor product space $\bigl(\mathbb{C}^2\bigr)^{\otimes n}$ the matrices act.
\end{prop}
The isomorphism class of $\pi^{sp}_{\xi,\xi^\prime}$ only depends on $\xi+\xi^\prime$, see \cite[Prop. 3.5]{SV}. We write $\pi^{sp}_{(\xi^\prime)}$ for the 
representation $\pi^{sp}_{0,\xi^\prime}$, so that
$\pi^{sp}_{\xi,\xi^\prime}\simeq\pi^{sp}_{(\xi+\xi^\prime)}$.
We will sometimes suppress the representation parameter $\xi$ and write $\pi^{sp}=\pi^{sp}_{(\xi)}$.
\begin{rema}\label{remarkrho}
The spin representation $\pi^{sp}$ factorizes through a representation of the two-boundary Temperley-Lieb algebra,
see \cite{dGN, SV}.
\end{rema}
For generic parameters there are two natural complex linear bases of the representation space $\bigl(\mathbb{C}^2\bigr)^{\otimes n}$ which we
denote by $\{v_{\underline{\epsilon}}\}_{\underline{\epsilon}}$ and $\{b_{\underline{\epsilon}}\}_{\underline{\epsilon}}$, where the indices
$\underline{\epsilon}=(\epsilon_1,\epsilon_2,\ldots,\epsilon_n)$ are running over $\{\pm 1\}^n$. The first basis $\{v_{\underline{\epsilon}}\}_{\underline{\epsilon}}$
is simply defined as 
\[
\bigl(\mathbb{C}^2\bigr)^{\otimes n}=\bigoplus_{\underline{\epsilon}}\mathbb{C}v_{\underline{\epsilon}},
\qquad v_{\underline{\epsilon}}:=v_{\epsilon_1}\otimes v_{\epsilon_2}\otimes\cdots\otimes v_{\epsilon_n},
\]
where $v_{\pm 1}:=v_{\pm}$. In particular, $v_{\underline{1}}=v_+^{\otimes n}$ for $\underline{1}:=(1,\ldots,1)\in\{\pm 1\}^n$.

The basis elements $v_{\underline{\epsilon}}$ can be expressed in terms of the $\pi^{sp}$-action of $H$ on the vector
$v_+^{\otimes n}$
as follows.
For any $w\in W$ and corresponding reduced expression $w=s_{i_1}s_{i_2}\cdots s_{i_r}$ ($0\leq i_j\leq n$) set
\[
T_w:=T_{i_1}T_{i_2}\cdots T_{i_r}\in H.
\]
The $T_w\in H$ are well defined (independent of the choice of reduced expression of $w\in W$). Define 
$w_{\underline{\epsilon}}\in W_0$  for a $n$-tuple $\underline{\epsilon}=(\epsilon_1,\ldots,\epsilon_n)\in\{\pm 1\}^n$ by
\begin{equation}\label{wepsilon}
w_{\underline{\epsilon}}:=(s_{i_k}s_{i_k+1}\cdots s_n)\cdots (s_{i_2}s_{i_2+1}\cdots s_n)(s_{i_1}s_{i_1+1}\cdots s_n)
\end{equation}
where $\{i_1,i_2,\ldots,i_k\}:=\{i\,\, | \,\, \epsilon_i=-1\}$ and $1\leq i_1<i_2<\cdots<i_k\leq n$. Note that
$w_{\underline{\epsilon}}(\underline{1})=\underline{\epsilon}$ with respect to the action of $W_0$ on $\{\pm 1\}^n\subset\mathbb{Z}^n$.
The following lemma is now easy to verify.
\begin{lem}\label{wepsilonlemma}
The elements
$\{w_{\underline{\epsilon}}\}_{\underline{\epsilon}}$ are the minimal coset representatives of $W_0/S_n$, and
\[
v_{\underline{\epsilon}}=\pi^{sp}(T_{w_{\underline{\epsilon}}})v_+^{\otimes n}.
\]
\end{lem}

The second basis $\{b_{\underline{\epsilon}}\}_{\underline{\epsilon}}$ of $\bigl(\mathbb{C}^2\bigr)^{\otimes n}$ is defined 
as follows. Set for $i\in\{1,\ldots,n\}$,
\[
Y_i:=T_{i-1}^{-1}\cdots T_1^{-1}T_0T_1\cdots T_{n-1}T_nT_{n-1}\cdots T_i\in H
\]
(cf. \eqref{taui}). Then $\lbrack Y_i,Y_j\rbrack=0$ and the elements $T_1,\ldots,T_n,Y_1^{\pm 1},\ldots, Y_n^{\pm 1}$ algebraically generate $H$.
The affine Hecke algebra can be completely characterised in terms of these generators, leading to the Bernstein-Zelevinsky presentation of $H$ (see \cite{Lu}).
 We write $Y^\lambda=Y_1^{\lambda_1}Y_2^{\lambda_2}\cdots Y_n^{\lambda_n}$ for $\lambda=(\lambda_1,\ldots,\lambda_n)\in\mathbb{Z}^n$.
 
 For $\eta=(\eta_1,\eta_2,\ldots,\eta_n)\in\mathbb{C}^n$ and a left $H$-module $V$ 
 set
 \begin{equation*}
 \begin{split}
 V_\eta:=&\{v\in V \,\, | \,\, Y_iv=q^{-\eta_i}v\quad \forall\, i \}\\
 =&\{v\in V \,\, | \,\, Y^\lambda v=q^{-(\lambda,\eta)}v\quad \forall\, \lambda\in\mathbb{Z}^n \}
 \end{split}
 \end{equation*}
 for the common eigenspace of the commuting operators $Y_i$ on $V$ with eigenvalues $q^{-\eta_i}$, where we have
used the standard bilinear form $\bigl(\mathbf{z},\mathbf{w}\bigr)=\sum_{i=1}^nz_iw_i$ on $\mathbb{C}^n$.
 
 Set
 \begin{equation}\label{xi}
 \gamma:=(\xi+(n-1)\kappa,\xi+(n-3)\kappa,\ldots,\xi+(1-n)\kappa)\in\mathbb{C}^n.
 \end{equation}
It follows from the fact that $\bigl(\mathbb{C}^2\bigr)^{\otimes n}$ is a principal series module of $H$ with central character $q^{-\gamma}$ (see \cite[Prop. 3.5]{SV}
and Subsection \ref{link}) that the spin representation $\bigl(\pi^{sp}_{(\xi)},\bigl(\mathbb{C}^2\bigr)^{\otimes n}\bigr)$ decomposes for generic parameters as
 \[
\bigl(\mathbb{C}^{2}\bigr)^{\otimes n}=\bigoplus_{\underline{\epsilon}\in\{\pm 1\}^n}\bigl(\bigl(\mathbb{C}^2\bigr)^{\otimes n}\bigr)_{w_{\underline{\epsilon}}\gamma}
 \]
with the common eigenspaces $\bigl(\bigl(\mathbb{C}^2\bigr)^{\otimes n}\bigr)_{w_{\underline{\epsilon}}\gamma}$ being one-dimensional. 
 The generic conditions on the parameters can be made precise, see \cite[Prop. 2.12]{S} and Subsection \ref{psr}. 
 
 The basis element $b_{\underline{\epsilon}}$ is 
a particular choice of nonzero element from $\bigl(\bigl(\mathbb{C}^2\bigr)^{\otimes n}\bigr)_{w_{\underline{\epsilon}}\gamma}$.
To define it we need to introduce a bit more notation.
Let $R_0=\{\pm e_i\pm e_j\}_{1\leq i\not=j\leq n}\cup\{\pm e_i\}_{i=1}^n$ be the root system of type $B_n$. We fix
\[
\{\alpha_1,\ldots,\alpha_{n-1},\alpha_n\}=\{e_1-e_2,\ldots,e_{n-1}-e_n,e_n\}
\]
as basis of $R_0$ and write $R_0^+$ (respectively $R_0^-$) for the corresponding set of positive (respectively negative) roots in $R_0$. The simple
reflections $s_1,\ldots,s_n$ of $W_0$ correspond with the reflections in the simple roots $\alpha_1,\ldots,\alpha_n$.
Let $\leq$ be the associated Bruhat order on $W_0$.
Set for roots $\alpha\in R_0$,
\begin{equation*}
N_\alpha(\mathbf{z}):=
\begin{cases} 
\frac{(1-q^{(\alpha,\mathbf{z})})}{q^{\kappa}(1-q^{-2\kappa+(\alpha,\mathbf{z})})}\qquad
&\hbox{ if }\quad \|\alpha\|^2=2,\\
\frac{(1-q^{2(\alpha,\mathbf{z})})}
{q^\zeta(1-q^{-\zeta-\zeta^\prime+(\alpha,\mathbf{z})})(1+q^{-\zeta+\zeta^\prime+(\alpha,\mathbf{z})})}\qquad
&\hbox{ if } \quad \|\alpha\|^2=1.
\end{cases}
\end{equation*}
It is easy to check that if $\alpha\in R_0^+\cap w_{\underline{\epsilon}}^{-1}R_0^-$ ($\underline{\epsilon}\in\{\pm 1\}^n$) then 
$\mathbf{z}\mapsto N_\alpha(\mathbf{z})$ is regular at $\mathbf{z}=\gamma$
(see also Subsection \ref{psr}). The following lemma follows from standard techniques involving affine Hecke algebra intertwiners (see \cite[\S 2.5]{S} and Subsection \ref{psr}). 
\begin{lem}\label{bepsilon} 
Fix generic parameters and fix $\underline{\epsilon}\in\{\pm 1\}^n$. There exists a unique $0\not=b_{\underline{\epsilon}}\in\bigl(\bigl(\mathbb{C}^2\bigr)^{\otimes n}\bigr)_{w_{\underline{\epsilon}}\gamma}$
satisfying 
\[
b_{\underline{\epsilon}}=N_{\underline{\epsilon}}v_{\underline{\epsilon}}+
\sum_{\underline{\epsilon^\prime}\in\{\pm 1\}^n: w_{\underline{\epsilon^\prime}}<w_{\underline{\epsilon}}}
L_{\underline{\epsilon^\prime}}v_{\underline{\epsilon^\prime}}
\]
for certain coefficients $L_{\underline{\epsilon^\prime}}\in\mathbb{C}$, where
\[
N_{\underline{\epsilon}}:=\prod_{\alpha\in R_0^+\cap w_{\underline{\epsilon}}^{-1}R_0^-}N_\alpha(\gamma).
\]
The vectors $b_{\underline{\epsilon}}$ ($\underline{\epsilon}\in\{\pm 1\}^n$) forms a basis of $\bigl(\mathbb{C}^2\bigr)^{\otimes n}$.
\end{lem}
Note that $b_{\underline{1}}=v_+^{\otimes n}=v_{\underline{1}}$.

\subsection{The boundary qKZ equations}\label{rqKZ}

The Baxterization \cite[\S 4]{SV} 
of the representation $\pi^{sp}$ gives rise to a $W$-cocycle 
$\{C^u(\mathbf{z})\}_{u\in W}$ depending on two additional parameters $\upsilon$ and $\upsilon^\prime$,
which are the two additional degrees of freedom in the $C^\vee C_n$ type double affine Hecke algebra $\mathbb{H}$ containing
$H$ as a subalgebra. We do not recall here this construction of the cocycle $\{C^u(\mathbf{z})\}_{u\in W}$, which uses the intertwiners of the double
affine Hecke algebra $\mathbb{H}$, but instead will give the resulting explicit formulas for $C^u(\mathbf{z})$ directly. Details on this Baxterization
procedure can be found in \cite{ChBook, S, SV}.

The values $C^u(\mathbf{z})$ are 
$\textup{End}_{\mathbb{C}}\bigl(\bigl(\mathbb{C}^2\bigr)^{\otimes n}\bigr)$-valued
meromorphic functions in $\mathbf{z}\in\mathbb{C}^n$ satisfying
\[
C^e(\mathbf{z})=\textup{Id}_{(\mathbb{C}^2)^{\otimes n}},
\qquad C^{uv}(\mathbf{z})=C^u(\mathbf{z})C^v(u^{-1}\mathbf{z})
\]
for all $u,v\in W$ and satisfying
\begin{equation}\label{localnondyn}
\begin{split}
C^{s_0}(\mathbf{z})&=\underline{\mathcal{K}}_1\bigl(\frac{1}{2}-z_1\bigr),\\
C^{s_i}(\mathbf{z})&=P_{i,i+1}\mathcal{R}_{i,i+1}(z_i-z_{i+1}),\qquad 1\leq i<n,\\
C^{s_n}(\mathbf{z})&=\mathcal{K}_n(z_n)
\end{split}
\end{equation}
with 
\begin{equation*}
\mathcal{R}(z):=\frac{1}{1-q^{-2\kappa+z}}\left(\begin{matrix}
1-q^{-2\kappa+z} & 0 & 0 & 0\\
0 & q^{-\kappa}(1-q^z) & 1-q^{-2\kappa} & 0\\
0 & (1-q^{-2\kappa})q^z & q^{-\kappa}(1-q^z) & 0\\
0 & 0 & 0 & 1-q^{-2\kappa+z}\end{matrix}\right)
\end{equation*}
and 
\begin{equation*}
\begin{split}
\underline{\mathcal{K}}(z)&=\underline{k}(z)\left(\begin{matrix}
(q^{\zeta^\prime}-q^{-\zeta^\prime})q^{2z}+(q^{\upsilon^\prime}-q^{-\upsilon^\prime})q^z & q^{-\xi}(1-q^{2z})\\
q^{\xi}(1-q^{2z}) & q^{\zeta^\prime}-q^{-\zeta^\prime}+(q^{\upsilon^\prime}-q^{-\upsilon^\prime})q^z\end{matrix}
\right),\\
\mathcal{K}(z)&=k(z)\left(\begin{matrix} q^{\zeta}-q^{-\zeta}+(q^{\upsilon}-q^{-\upsilon})q^z & 1-q^{2z}\\
1-q^{2z} & (q^{\zeta}-q^{-\zeta})q^{2z}+(q^{\upsilon}-q^{-\upsilon})q^z\end{matrix}\right)
\end{split}
\end{equation*}
and scalar functions $\underline{k}(z)$ and $k(z)$ given by
\[
\underline{k}(z):=\frac{q^{-\zeta^\prime}}{(1-q^{-\zeta^\prime-\upsilon^\prime+z})
(1+q^{-\zeta^\prime+\upsilon^\prime+z})},\quad
k(z):=\frac{q^{-\zeta}}{(1-q^{-\zeta-\upsilon+z})(1+q^{-\zeta+\upsilon+z})},
\]
see \cite[\S 4]{SV}. 
The fact that $\{C^u(\mathbf{z})\}_{u\in W}$
is a $W$-cocycle of the form \eqref{localnondyn} is equivalent to $\mathcal{R}(z)$ being a unitary $R$-matrix:
\[
\mathcal{R}_{12}(z_1-z_2)\mathcal{R}_{13}(z_1-z_3)\mathcal{R}_{23}(z_2-z_3)=\mathcal{R}_{23}(z_2-z_3)\mathcal{R}_{13}(z_1-z_3)
\mathcal{R}_{12}(z_1-z_2),
\]
$\mathcal{R}_{21}(z)\mathcal{R}(-z)=\textup{Id}_{\mathbb{C}^2\otimes\mathbb{C}^2}$, and 
$\underline{\mathcal{K}}(z)$ respectively $\mathcal{K}(z)$ being a left respectively right unitary $K$-matrix with respect to $\mathcal{R}(z)$:
\begin{equation}\label{reflnondyn}
\begin{split}
\mathcal{R}(z_1-z_2)\underline{\mathcal{K}}_1(z_1)\mathcal{R}_{21}(z_1+z_2)\underline{\mathcal{K}}_2(z_2)&=
\underline{\mathcal{K}}_2(z_2)\mathcal{R}(z_1+z_2)\underline{\mathcal{K}}_1(z_1)\mathcal{R}_{21}(z_1-z_2),\\
\mathcal{R}_{21}(z_1-z_2)\mathcal{K}_1(z_1)\mathcal{R}(z_1+z_2)\mathcal{K}_2(z_2)&=
\mathcal{K}_2(z_2)\mathcal{R}_{21}(z_1+z_2)\mathcal{K}_1(z_1)\mathcal{R}(z_1-z_2),
\end{split}
\end{equation}
$\underline{\mathcal{K}}(z)\underline{\mathcal{K}}(-z)=\textup{Id}_{\mathbb{C}^2}=\mathcal{K}(z)\mathcal{K}(-z)$
(compare with \cite[\S 2.3]{dGN}). The equations \eqref{reflnondyn} are called {\it reflection equations}, or {\it boundary Yang-Baxter equations}.

A transfer operator can be associated to the data $\underline{\mathcal{K}}(z)$, $\mathcal{R}(z)$ and $\mathcal{K}(z)$ (see \cite{Sk}). It is the transfer operator of the finite inhomogeneous XXZ spin-$\frac{1}{2}$ chain with arbitrary reflecting boundaries at both ends, see \cite{FK,dGN,SV} and references
therein. 

In this paper we are interested in the quantum affine KZ equations related to the integrability data $\underline{\mathcal{K}}(z)$,
$\mathcal{R}(z)$ and $\mathcal{K}(z)$.
They are the following {\it boundary quantum KZ equations}
(cf. \cite{CQKZ,SV} and references therein).
\begin{defi}\label{bqKZform}
Let $f: \mathbb{C}^n\rightarrow \bigl(\mathbb{C}^2\bigr)^{\otimes n}$ be a $\bigl(\mathbb{C}^2\bigr)^{\otimes n}$-valued meromorphic function on $\mathbb{C}^n$. We say that 
$f$ is a solution of the boundary quantum KZ equations associated to the spin representation $\pi^{sp}=\pi^{sp}_{(\xi)}$ if $f$ satisfies the difference
equations
\begin{equation}\label{REFLqKZXXZ}
C^{\tau(\lambda)}(\mathbf{z})f(\mathbf{z}-\lambda)=f(\mathbf{z})\qquad \forall\, \lambda\in\mathbb{Z}^n.
\end{equation}
We write $\textup{Sol}_{sp}$ for the space of meromorphic solutions $f:  \mathbb{C}^n\rightarrow \bigl(\mathbb{C}^2\bigr)^{\otimes n}$ of \eqref{REFLqKZXXZ}.
\end{defi}
\begin{rema}
{\bf (i)} By the cocycle property of $\{C^v(\mathbf{z})\}_{v\in W}$ we obtain a $W$-action on the space of $\bigl(\mathbb{C}^2\bigr)^{\otimes n}$-valued meromorphic functions 
on $\mathbb{C}^n$ by
\[
\bigl(\nabla(v)f\bigr)(\mathbf{z}):=C^v(\mathbf{z})f(v^{-1}\mathbf{z}),\qquad v\in W.
\]
The boundary quantum KZ equations \eqref{REFLqKZXXZ} are equivalent to the equations
$\nabla(\tau(\lambda))f=f$ for all $\lambda\in\mathbb{Z}^n$. The solution space $\textup{Sol}_{sp}$
becomes a $W_0$-module by restricting the $\nabla$-action of $W_0$ to $\textup{Sol}_{sp}$.\\
{\bf (ii)} The boundary quantum KZ equations \eqref{REFLqKZXXZ} are equivalent to the equations
$\nabla(\tau(e_i))f=f$ ($i=1,\ldots,n$) which, by the cocycle property of
$\{C^u(\mathbf{z})\}_{u\in W}$, \eqref{localnondyn} and \eqref{taui}, take on the explicit form
\[
C^{\tau(-e_i)}(\mathbf{z})f(\mathbf{z}+e_i)=f(\mathbf{z}),\qquad \, i=1,\ldots,n\]
with
\begin{equation*}
\begin{split}
C^{\tau(-e_i)}(\mathbf{z})&=\mathcal{R}_{i+1,i}(z_i-z_{i+1})\mathcal{R}_{i+2,i}(z_i-z_{i+2})\cdots\mathcal{R}_{ni}(z_i-z_n)\mathcal{K}_i(z_i)\\
&\times\mathcal{R}_{in}(z_i+z_n)\cdots\mathcal{R}_{i,i+1}(z_i+z_{i+1})\mathcal{R}_{i,i-1}(z_{i-1}+z_i)\cdots\mathcal{R}_{i1}(z_1+z_i)\\
&\times\underline{\mathcal{K}}_i\bigl(\frac{1}{2}+z_i\bigr)\mathcal{R}_{1i}(1-z_1+z_i)\cdots \mathcal{R}_{i-2,i}(1-z_{i-2}+z_i)\mathcal{R}_{i-1,i}(1-z_{i-1}+z_i).
\end{split}
\end{equation*}
It is in this form that the boundary quantum KZ equations often appear in the literature, see, e.g., \cite{CQKZ,JKKMW,W,DF,DFZJ,dGP,Koj,RSV}.\\
{\bf (iii)} For special choices of the $K$-matrices $\underline{\mathcal{K}}(z)$ and $\mathcal{K}(z)$ a vertex operator approach leads to solutions
of the associated boundary quantum KZ equations \cite{JKKMW}. These solutions give rise to correlation functions of the semi-infinite XXZ spin-$\frac{1}{2}$ chain. 
 See \cite{DF,DFZJ,Ka,SV,BK} for other constructions of solutions for special classes of $K$-matrices $\underline{\mathcal{K}}(z)$ and $\mathcal{K}(z)$.
\end{rema}
\subsection{Basis consisting of power series solutions}\label{asymptXXZsec}
Note that $\textup{Sol}_{sp}$ is a vector space over the field $F$ of scalar valued $\mathbb{Z}^n$-translation invariant
meromorphic functions on $\mathbb{C}^n$. We now give, for generic
parameters $\xi,\kappa,\zeta,\zeta^\prime,\upsilon,\upsilon^\prime$, the construction of a 
$F$-basis of $\textup{Sol}_{sp}$.  It is a special case of the construction of a basis of solutions
of the quantum affine KZ equations associated to principal series modules in Subsection \ref{pss} using the asymptotic
techniques from \cite{vMS, vM, SAnn, S2}. 

Let
$w_0\in W_0$ for the longest Weyl group element ($w_0=-1$ in the natural action of $W_0$ on $\mathbb{C}^n$ by permutations
and sign changes) and define $\rho\in\mathbb{C}^n$ by
\[
\rho=(\zeta+\zeta^\prime+(n-1)\kappa,\zeta+\zeta^\prime+(n-3)\kappa,\ldots,\zeta+\zeta^\prime+(1-n)\kappa).
\]
Let $\widetilde{\rho}$ be the vector $\rho$ with $\upsilon$ and $\zeta^\prime$ interchanged. The role of the plane wave in the asymptotic expansion is played
by
\[
\mathcal{W}(\mathbf{z},\mathbf{w}):=q^{(\rho-\mathbf{w},\widetilde{\rho}+w_0\mathbf{z})}.
\]

Let $\zeta,\zeta^\prime,\upsilon,\upsilon^\prime\in\mathbb{C}$.
Define the associated Askey-Wilson parameters by
\begin{equation}\label{abcd}
\{a,b,c,d\}:=\{q^{\zeta+\upsilon}, -q^{\zeta-\upsilon}, q^{\frac{1}{2}+\zeta^\prime+\upsilon^\prime},
-q^{\frac{1}{2}+\zeta^\prime-\upsilon^\prime}\}
\end{equation}
and the dual Askey-Wilson parameters by
\[
\{\widetilde{a},\widetilde{b},\widetilde{c},\widetilde{d}\}:=
\{q^{\zeta+\zeta^\prime},-q^{\zeta-\zeta^\prime},q^{\frac{1}{2}+\upsilon+\upsilon^\prime},
-q^{\frac{1}{2}+\upsilon-\upsilon^\prime}\},
\]
cf. \cite{S2}. Note that the dual Askey-Wilson parameters are obtained from the Askey-Wilson parameters
by interchanging $\upsilon$ and $\zeta^\prime$. This defines an involution on the parameters, which we call the duality involution. 

Write
\[
\bigl(x_1,\ldots,x_m;q\bigr)_{\infty}:=\prod_{i=1}^r\bigl(x_i;q\bigr)_{\infty},\qquad
\bigl(x;q\bigr)_{\infty}:=\prod_{j=0}^{\infty}(1-q^jx)
\]
for products of $q$-shifted factorials and set
\begin{equation*}
\begin{split}
\mathcal{S}_{sp}(\mathbf{z}):=&\prod_{i=1}^n\bigl(q^{1-z_i}/a,q^{1-z_i}/b,q^{1-z_i}/c,q^{1-z_i}/d;q\bigr)_{\infty}\\
\times&\prod_{1\leq r<s\leq n}\bigl(q^{1-2\kappa-z_r+z_s}, q^{1-2\kappa-z_r-z_s}, -q^{1-z_r+z_s},-q^{1-z_r-z_s};q\bigr)_{\infty}
\end{split}
\end{equation*}
and
\begin{equation*}
\begin{split}
\mathcal{U}(\mathbf{z}):=&\frac{\mathcal{S}_{sp}(\mathbf{z})}{\prod_{i=1}^n\bigl(q^{1-2z_i};q\bigr)_{\infty}\prod_{1\leq r<s\leq n}
\bigl(q^{2-2z_r+2z_s},q^{2-2z_r-2z_s};q^2\bigr)_{\infty}}\\
=&\prod_{i=1}^n\frac{\bigl(q^{1-z_i}/a,q^{1-z_i}/b,q^{1-z_i}/c,q^{1-z_i}/d;q\bigr)_{\infty}}
{\bigl(q^{1-2z_i};q\bigr)_{\infty}}\prod_{1\leq r<s\leq n}\frac{\bigl(q^{1-2\kappa-z_r+z_{s}},q^{1-2\kappa-z_r-z_s};q\bigr)_{\infty}}
{\bigl(q^{1-z_r+z_s},q^{1-z_r-z_s};q\bigr)_{\infty}}.
\end{split}
\end{equation*}
Write $\widetilde{\mathcal{U}}(\mathbf{z})$ for $\mathcal{U}(\mathbf{z})$ with
$\upsilon$ and $\zeta^\prime$ interchanged (i.e. the Askey-Wilson parameters are replaced by the dual Askey-Wilson parameters).
Write $Q_+$ for the $\mathbb{Z}_{\geq 0}$-linear combinations of the simple roots $\alpha_j$ ($1\leq j\leq n$). 
\begin{thm}\label{asymptoticspin}
For generic parameter values we have
\[
\textup{Sol}_{sp}=\bigoplus_{\underline{\epsilon}\in\{\pm 1\}^n}F\Phi_{\underline{\epsilon}}
\]
with $\Phi_{\underline{\epsilon}}\in\textup{Sol}_{sp}$ characterised by the expansion formula
\begin{equation}\label{expansionbq}
\Phi_{\underline{\epsilon}}(\mathbf{z}):=\frac{\mathcal{W}(\mathbf{z},w_{\underline{\epsilon}}\gamma)}
{\mathcal{S}_{sp}(\mathbf{z})\widetilde{\mathcal{U}}(w_{\underline{\epsilon}}\gamma)}
\sum_{\alpha\in Q_+}\Gamma_{\alpha,\underline{\epsilon}}q^{-(\alpha,\mathbf{z})},\qquad
\Gamma_{\alpha,\underline{\epsilon}}\in\bigl(\mathbb{C}^2\bigr)^{\otimes n}
\end{equation}
with the power series converging normally for $\mathbf{z}$ in compacta of $\mathbb{C}^n$,
with $\gamma$ given by \eqref{xi}, and with leading coefficient given by
\[
\Gamma_{0,\underline{\epsilon}}:=\pi^{sp}(T_{w_0})b_{\underline{\epsilon}}.
\]
\end{thm}
The proof of the theorem is given in Subsection \ref{link}. It is obtained as a special case of Proposition \ref{specI}, which deals with power series
solutions of quantum affine KZ equations associated to principal series modules of affine Hecke algebras. Note that $\mathcal{S}_{sp}(\mathbf{z})\Phi_{\underline{\epsilon}}(\mathbf{z})$
is holomorphic in $\mathbf{z}\in\mathbb{C}^n$, i.e. the factor $\mathcal{S}_{sp}(\mathbf{z})$ in \eqref{expansionbq} is singling out the singularities of $\Phi_{\underline{\epsilon}}(\mathbf{z})$.

Observe that $\Phi_{\underline{\epsilon}}(\mathbf{z})$ is the solution of the boundary qKZ equations which behaves as the "plane wave"
$\bigl(\widetilde{\mathcal{U}}(w_{\underline{\epsilon}}\gamma)^{-1}\pi^{sp}(T_{w_0})b_{\underline{\epsilon}}\bigr)\mathcal{W}(\mathbf{z},w_{\underline{\epsilon}}\gamma)$ 
when $\Re\bigl((\alpha_i,\mathbf{z})\bigr)\rightarrow -\infty$ for $i=1,\ldots,n$. This power series basis thus picks the Weyl chamber $C_-:=\{\mathbf{x}\in\mathbb{R}^n \,\, | \,\, 
\bigl(\alpha_i,\mathbf{x}\bigr)<0\,\,\,\forall\, i\}$ as the asymptotic region in which the basis elements behave as plane waves. 

\begin{rema}
{\bf (i)} The generic conditions on the parameters
can be made precise, see Subsections \ref{pss} \& \ref{link}.\\
{\bf (ii)} The choice of normalisation factor $\widetilde{\mathcal{U}}(w_{\underline{\epsilon}}\gamma)^{-1}$
in \eqref{expansionbq} is motivated by duality properties of the asymptotic series solutions of quantum affine KZ equations
associated to minimal principal series representations (see \cite{vM, vMS, SAnn, S2} and Subsection \ref{pss}). 
\end{rema}
\subsection{The connection problem}\label{connXXZsec}
In this subsection we assume that the parameters are generic.
Let $v\in W_0$. Then $\{(\nabla(v)\Phi_{\underline{\epsilon}^\prime})(\mathbf{z})\}_{\underline{\epsilon}^\prime}$ is a $F$-linear basis of
$\textup{Sol}_{sp}$ with the basis elements behaving asymptotically as plane waves for $\Re\bigl((\alpha_i,v^{-1}\mathbf{z})\bigr)\rightarrow -\infty$ for $i=1,\ldots,n$. 
Let $M_{cm;\underline{\epsilon},\underline{\epsilon}^\prime}^v(\cdot,\xi)\in F$ be the unique elements satisfying
\[
\bigl(\nabla(v)\Phi_{\underline{\epsilon}^\prime}\bigr)(\mathbf{z})=
\sum_{\underline{\epsilon}\in\{\pm 1\}^n}M_{cm;\underline{\epsilon},\underline{\epsilon}^\prime}^v(\mathbf{z},\xi)\Phi_{\underline{\epsilon}}(\mathbf{z})
\]
(the subindex "cm" stands for "connection matrix"). 
We single out the dependence on the representation parameter $\xi$, because it will play the role of the dynamical parameter when relating the
$M_{cm;\underline{\epsilon},\underline{\epsilon}^\prime}^v(\mathbf{z};\xi)$ to solutions of dynamical quantum Yang-Baxter and reflection equations.

Consider the corresponding $2^n\times 2^n$-matrix 
\[
M_{cm}^v(\cdot,\xi)=\bigl(M^v_{cm;\underline{\epsilon},\underline{\epsilon}^\prime}(\cdot,\xi)\bigr)_{\underline{\epsilon},\underline{\epsilon}^\prime\in\{\pm 1\}^n}
\]
as an $\textup{End}_{\mathbb{C}}\bigl((\mathbb{C}^2)^{\otimes n}\bigr)$-valued meromorphic function on $\mathbb{C}^n$ by
\[
M_{cm}^v(\mathbf{z},\xi)v_{\underline{\epsilon}^\prime}:=\sum_{\underline{\epsilon}\in\{\pm 1\}^n}
M_{cm;\underline{\epsilon},\underline{\epsilon}^\prime}^v(\mathbf{z},\xi)
v_{\underline{\epsilon}}.
\]
\begin{cor}\label{cocycleM}
The set $\{M_{cm}^v(\mathbf{z},\xi)\}_{v\in W_0}$ of $\textup{End}_{\mathbb{C}}\bigl((\mathbb{C}^2)^{\otimes n}\bigr)$-valued
meromorphic functions in $\mathbf{z}\in\mathbb{C}^n$ is a $W_0$-cocycle:
\[
M_{cm}^e(\mathbf{z},\xi)=\textup{Id}_{(\mathbb{C}^2)^{\otimes n}},\qquad
M_{cm}^{uv}(\mathbf{z},\xi)=M_{cm}^u(\mathbf{z},\xi)M_{cm}^v(u^{-1}\mathbf{z},\xi)
\]
for all $u,v\in W_0$.
\end{cor}
We call $\{M_{cm}^v(\mathbf{z},\xi)\}_{v\in W_0}$ the {\it connection cocycle} of the boundary quantum KZ equations associated
to the spin representation $\pi^{sp}_{(\xi)}$. 

Let $h$ be the linear operator on $\mathbb{C}^2$ defined by $hv_\epsilon=\epsilon v_\epsilon$ ($\epsilon\in\{\pm 1\}$). Write $h_i$ for the linear operator on 
$\bigl(\mathbb{C}^2\bigr)^{\otimes n}$ given by
\[
h_i:=\textup{id}_{(\mathbb{C}^2)^{\otimes (i-1)}}\otimes h\otimes \textup{id}_{(\mathbb{C}^2)^{\otimes (n-i)}}.
\]
For a family $S(\xi)$ of linear operators on $\bigl(\mathbb{C}^2\bigr)^{\otimes n}$ depending meromorphically on $\xi$ we write 
\[
S(\xi+\alpha h_i)v_{\underline{\epsilon}}:=S(\xi+\alpha\epsilon_i)v_{\underline{\epsilon}}.
\]
for the associated linear operator on $\bigl(\mathbb{C}^2\bigr)^{\otimes n}$ in which the representation parameter $\xi$ is shifted according to the "weight" of the 
$i$th tensor component of $v_{\underline{\epsilon}}$.

Write
\[
\theta(x_1,\ldots,x_r;q):=\prod_{i=1}^r\theta(x_i;q),\qquad \theta(x;q):=\prod_{j=0}^{\infty}(1-q^jx)(1-q^{j+1}/x)
\]
for products of renormalised Jacobi theta functions. Set
\begin{equation}\label{calC}
\mathcal{C}(z,\xi):=
\frac{\theta\bigl(\widetilde{a}q^\xi,\widetilde{b}q^\xi,\widetilde{c}q^\xi,dq^{\xi-z}/\widetilde{a};q\bigr)}
{\theta\bigl(q^{2\xi},dq^{-z};q\bigr)}q^{-(\zeta+\upsilon-z)(\zeta+\zeta^\prime-\xi)}
\end{equation}
and write $\widetilde{\mathcal{C}}(z,\xi)$ for $\mathcal{C}(z,\xi)$ with $\upsilon$
and $\zeta^\prime$ interchanged,
\[
\widetilde{\mathcal{C}}(z,\xi):=
\frac{\theta\bigl(aq^\xi,bq^\xi,cq^\xi,\widetilde{d}q^{\xi-z}/a;q\bigr)}
{\theta\bigl(q^{2\xi},\widetilde{d}q^{-z};q\bigr)}q^{-(\zeta+\zeta^\prime-z)(\zeta+\upsilon-\xi)}.
\]
Note that $\mathcal{C}(z,\xi)$ is one-periodic in both $z$ and $\xi$. 
We show in Subsection \ref{link}
that the connection cocycle $\{M_{cm}^u(\mathbf{z},\xi)\}_{u\in W_0}$ is characterised by the following formulas.
\begin{thm}\label{Mwc}
We have
\begin{equation}\label{MrelBa}
\begin{split}
M^{s_i}_{cm}(\mathbf{z},\xi)&=P_{i,i+1}R_{cm}(z_i-z_{i+1},2\xi-2\kappa (h_1+h_2+\cdots+h_{i-1}))_{i,i+1},\quad 1\leq i<n,\\
M^{s_n}_{cm}(\mathbf{z},\xi)&=K_{cm}(z_n,\xi-\kappa(h_1+h_2+\cdots+h_{n-1}))_n
\end{split}
\end{equation}
where $P: \mathbb{C}^2\otimes\mathbb{C}^2\rightarrow \mathbb{C}^2\otimes\mathbb{C}^2$ is the permutation operator and
\begin{equation*}
R_{cm}(z,\xi):=\left(\begin{matrix} 1 & 0 & 0 & 0\\
0 & A_{cm}(z,\xi) & B_{cm}(z,\xi) & 0\\
0 & B_{cm}(z,-\xi) & A_{cm}(z,-\xi) & 0\\
0 & 0 & 0 & 1\end{matrix}\right),
\end{equation*}
\begin{equation*}
K_{cm}(z,\xi):=\left(\begin{matrix}\alpha_{cm}(z,\xi) & \beta_{cm}(z,\xi)\\
\beta_{cm}(z,-\xi) & \alpha_{cm}(z,-\xi)\end{matrix}\right)
\end{equation*}
with the matrix coefficients given explicitly by
\[
A_{cm}(z,\xi):=\frac{\theta(q^{2\kappa-\xi},q^{-z};q)}{\theta(q^{2\kappa-z},q^{-\xi};q)}q^{2\kappa(z-\xi)},\qquad 
B_{cm}(z,\xi):=\frac{\theta(q^{2\kappa},q^{-z-\xi};q)}{\theta(q^{-\xi},q^{2\kappa-z};q)}q^{(2\kappa+\xi)z},
\]
and 
\[
\alpha_{cm}(z,\xi):=\frac{\mathcal{C}(z,\xi)-\widetilde{\mathcal{C}}(\xi,z)}{\widetilde{\mathcal{C}}(\xi,-z)},
\qquad 
\beta_{cm}(z,\xi):=\frac{\mathcal{C}(z,\xi)}{\widetilde{\mathcal{C}}(-\xi,-z)}.
\]
\end{thm}
Note that $R_{cm}(z,\xi)$ and $K_{cm}(z,\xi)$ are one-periodic in both $z$ and $\xi$. 

\subsection{Dynamical $R$- and $K$-matrices}\label{DynRK}
The explicit form \eqref{MrelBa} of the connection cocycle implies that $R_{cm}(z,\xi)$ is a unitary solution of the dynamical
quantum Yang-Baxter equation and $K_{cm}(z,\xi)$ a unitary solution of the dynamical reflection equation associated to $R_{cm}(z,\xi)$. Before stating the exact result 
we introduce the dynamical quantum Yang-Baxter and reflection equations in a slightly more general setting, replacing the spin space $\mathbb{C}^2$ by an arbitrary
finite dimensional complex vector space $V$. Before doing so we need to extend some of the notations introduced in the previous paragraph to this more general context.

Let $h: V\rightarrow V$ be a semisimple complex linear operator.
We write $V=\bigoplus_{\mu\in\mathbb{C}}V_\mu$ for the corresponding eigenspace decomposition, with $V_\mu$ the eigenspace of $h$ with eigenvalue $\mu$. 
Write $h_i: V^{\otimes n}\rightarrow V^{\otimes n}$ for the linear operator
\[
h_i:=\textup{Id}_{V}^{\otimes (i-1)}\otimes h\otimes \textup{Id}_{V}^{\otimes (n-i)}.
\]
We use the shorthand notation $\underline{\mu}:=(\mu_1,\ldots,\mu_n)$ for a $n$-tuple of complex numbers $\mu_i$. Set
\[
(V^{\otimes n})_{\underline{\mu}}:=V_{\mu_1}\otimes\cdots\otimes V_{\mu_n}=\{v\in V^{\otimes n} \,\, | \,\, h_iv=\mu_iv,\quad i=1,\ldots,n\},
\]
so that $V^{\otimes n}=\bigoplus_{\underline{\mu}}(V^{\otimes n})_{\underline{\mu}}$. 
Let $S(\xi)$ be a family of linear operator on $V^{\otimes n}$ depending meromorphically on a 
complex parameter $\xi\in\mathbb{C}$.  
For a given 
complex number $\alpha\in\mathbb{C}$
we write $\xi\mapsto S(\xi+\alpha h_i)$ for the family of linear operator on $V^{\otimes n}$ defined by
\[
S(\xi+\alpha h_i)v:=S(\xi+\alpha \mu_i)v,\qquad  v\in (V^{\otimes n})_{\underline{\mu}}.
\]
We also occasionally use the backward weight shift. To define it, let $P_{\underline{\mu}}: V^{\otimes n}\rightarrow (V^{\otimes n})_{\underline{\mu}}$
be the projection along the direct sum decomposition $V^{\otimes n}=\bigoplus_{\underline{\nu}}(V^{\otimes n})_{\underline{\nu}}$. Then we define
\[
S(\xi+\alpha\underline{h}_i):=\sum_{\underline{\nu}}P_{\underline{\nu}}S(\xi+\alpha \nu_i).
\]
It is a linear operator on $V^{\otimes n}$ depending meromorphically on $\xi$.

A family
$(z,\xi)\mapsto R(z,\xi)$ of linear operators on $V\otimes V$ depending meromorphically on two complex parameters $z$ and $\xi$ is called a dynamical $R$-matrix if $R(z,\xi)$ satisfies 
the dynamical quantum Yang-Baxter equation with spectral parameter \cite{Fe}:
\begin{equation}\label{qdYBintro}
\begin{split}
R_{12}(z_1-z_2, \xi-&2\kappa h_3)R_{13}(z_1-z_3,\xi)R_{23}(z_2-z_3,\xi-2\kappa h_1)=\\
=&R_{23}(z_2-z_3,\xi)R_{13}(z_1-z_3,\xi-2\kappa h_2)R_{12}(z_1-z_2,\xi)
\end{split}
\end{equation}
as a meromorphic family of linear operators on $V^{\otimes 3}$
(for later purposes it is convenient to add the factor two in the step size $2\kappa$). The complex parameter $z$ is called the spectral parameter and $\xi$ the
dynamical parameter.
The dynamical quantum Yang-Baxter equation \eqref{qdYBintro}, also known as the Gervais-Neveu-Felder equation,
first appeared in \cite{GN}. It is closely related to Baxter's star-triangle equation, see \cite[Thm. 3]{Fe} and Subsection \ref{8section}.

We say that a dynamical $R$-matrix $R(z,\xi)$ satisfies the ice rule if
\begin{equation}\label{ice}
\lbrack R(z,\xi), \Delta(h)\rbrack=0,
\end{equation}
where $\Delta(h):=h_1+h_2$. It is said to be unitarity if
\begin{equation}\label{unitary}
R_{21}(z,\xi)R(-z,\xi)=\textup{Id}_V^{\otimes 2}.
\end{equation}
We call $R(z,\xi)$ dynamically $P$-symmetric if
\begin{equation}\label{Psymmetry}
R_{21}(z,\xi)=R(z,-\xi+2\kappa\Delta(h))
\end{equation}
as linear operators on $V\otimes V$.

\begin{rema}
In \cite{Fe2} a family $\widetilde{R}(z,\xi)$ of linear operator on $V\otimes V$ is said to satisfy the dynamical quantum Yang-Baxter equation
if 
\begin{equation}\label{dynamicalYBFelder}
\begin{split}
\widetilde{R}_{12}(z_1-z_2,\xi&-\kappa h_3)\widetilde{R}_{13}(z_1-z_3;\xi+\kappa h_2)\widetilde{R}_{23}(z_2-z_3;\xi-\kappa h_1)=\\
&=\widetilde{R}_{23}(z_2-z_3;\xi+\kappa h_1)\widetilde{R}_{13}(z_1-z_3;\xi-\kappa h_2)\widetilde{R}_{12}(z_1-z_2;\xi+\kappa h_3).
\end{split}
\end{equation}
The two versions \eqref{qdYBintro} and \eqref{dynamicalYBFelder} of the dynamical quantum
Yang-Baxter equations are equivalent with the identification given by
\[
R(z,\xi)=\widetilde{R}(z,\xi-\kappa\Delta(h)),
\]
provided that $R(z,\xi)$ (hence also $\widetilde{R}(z,\xi)$) satisfies the ice-rule.
\end{rema}

Dynamical $R$-matrices encode the integrable structures underlying
$2$-dimensional face models from statistical mechanics with periodic boundary conditions (see, e.g., \cite{Ba,Fe2,Fe}). 
A well-known example 
is Baxter's $8$-vertex face dynamical R-matrix. We come back to this example in Subsection \ref{8section}.

Solutions of dynamical versions of reflection equations are related to integrable reflecting boundary conditions for face models, see, e.g., \cite{AR,BP,BPO,F}
and Subsection \ref{8section}.
In this paper we will use the following left and right version of 
the dynamical reflection equation.
\begin{defi}
Let $(z,\xi)\mapsto R(z,\xi)$ be a meromorphic family of linear operators on $V^{\otimes 2}$ of two complex parameters $z$ and $\xi$.
\begin{enumerate}
\item A meromorphic family $(z,\xi)\mapsto K(z,\xi)$ of linear operators on $V$ is called a 
right dynamical $K$-matrix with respect to $R(z,\xi)$ if $K(z,\xi)$ satisfies the right dynamical reflection equation 
\begin{equation}\label{reflectioneqn}
\begin{split}
R_{21}(z_1-z_2,&2\xi)K_1(z_1,\xi-\kappa h_2)R_{12}(z_1+z_2,2\xi)K_2(z_2,\xi-\kappa h_1)=\\
=&K_2(z_2,\xi-\kappa h_1)R_{21}(z_1+z_2,2\xi)K_1(z_1,\xi-\kappa h_2)R_{12}(z_1-z_2,2\xi)
\end{split}
\end{equation}
as a family of linear operators on $V\otimes V$.
The right dynamical $K$-matrix $K(z,\xi)$ is called unitary if 
\begin{equation}\label{unitaryK}
K(z,\xi)K(-z,\xi)=\textup{Id}_{V}.
\end{equation}
\item A meromorphic family $(z,\xi)\mapsto \underline{K}(z,\xi)$ of linear operators on $V$ is called a left dynamical $K$-matrix with respect to $R(z,\xi)$ if 
$\underline{K}(z,\xi)$ satisfies the left dynamical reflection equation
\begin{equation}\label{reflectioneqnleft}
\begin{split}
&R(z_1-z_2,2\xi+2\kappa\Delta(h))\underline{K}_1(z_1,\xi+\kappa h_2)R_{21}(z_1+z_2,2\xi+2\kappa\Delta(h))\underline{K}_2(z_2,\xi+\kappa h_1)\\
&=\underline{K}_2(z_2,\xi+\kappa h_1)R(z_1+z_2,2\xi+2\kappa\Delta(h))\underline{K}_1(z_1,\xi+\kappa h_2)R_{21}(z_1-z_2,2\xi+2\kappa\Delta(h))
\end{split}
\end{equation}
as a family of linear operators on $V\otimes V$.
The left dynamical $K$-matrix $\underline{K}(z,\xi)$ is called unitary if 
$\underline{K}(z,\xi)\underline{K}(-z,\xi)=\textup{Id}_{V}$.
\end{enumerate}
\end{defi}
\begin{rema}\label{lrchange}
If $R(z,\xi)$ is dynamically $P$-symmetric and if $K(z,\xi)$ is a solution of the right dynamical reflection equation \eqref{reflectioneqn} then
\[
\underline{K}(z,\xi):=K(z,-\xi)
\]
is a solution of the left dynamical reflection equation \eqref{reflectioneqnleft} (and vice versa).
\end{rema}

Let $R(z,\xi)$ be a linear operator on $V\otimes V$ and
$K(z,\xi)$ a
linear operator on $V$, both depending meromorphically on $(z,\xi)\in\mathbb{C}^2$. Let 
$P: V\otimes V\rightarrow V\otimes V$
be the permutation operator.
Consider the linear operators   
\begin{equation}\label{Mrelnew}
\begin{split}
M^{s_i}(\mathbf{z},\xi)&:=P_{i,i+1}R_{i,i+1}(z_i-z_{i+1},2\xi-2\kappa (h_1+h_2+\cdots+h_{i-1})),\qquad 1\leq i<n,\\
M^{s_n}(\mathbf{z},\xi)&:=K_n(z_n,\xi-\kappa(h_1+h_2+\cdots+h_{n-1})).
\end{split}
\end{equation}
on $V^{\otimes n}$, depending meromorphically on $(\mathbf{z},\xi)\in\mathbb{C}^n\times\mathbb{C}$.
\begin{prop}\label{propCocycle0}
Suppose that $R(z,\xi): V\otimes V\rightarrow V\otimes V$ satisfies the ice-rule \eqref{ice}.
The following two statements are equivalent.
\begin{enumerate}
\item The linear operators \eqref{Mrelnew} are part of a (necessarily unique) set $\{M^u(\mathbf{z},\xi)\}_{u\in W_0}$ of 
$\textup{End}_{\mathbb{C}}\bigl(V^{\otimes n}\bigr)$-valued meromorphic functions $M^u(\mathbf{z},\xi)$
in $(\mathbf{z},\xi)\in\mathbb{C}^n\times\mathbb{C}$
satisfying the cocycle properties
\[
M^e(\mathbf{z},\xi)=\textup{Id}_{V^{\otimes n}},\qquad M^{uv}(\mathbf{z},\xi)=
M^u(\mathbf{z},\xi)M^v(u^{-1}\mathbf{z},\xi)
\]
for all $u,v\in W_0$;
\item $R(z,\xi)$ is a unitary dynamical $R$-matrix and $K(z,\xi)$ is an associated right unitary dynamical $K$-matrix.
\end{enumerate}
\end{prop}
\begin{proof}
This follows from a direct computation.
\end{proof}
Consider now the special case that $V=\mathbb{C}^2$ with basis $\{v_+,v_-\}$ and with $h$ the linear operator on $\mathbb{C}^2$ satisfying
$hv_\epsilon=\epsilon v_\epsilon$.
Since $R_{cm}(z,\xi)$ satisfies the ice-rule \eqref{ice} we obtain from Corollary \ref{cocycleM}, Theorem \ref{Mwc} and Proposition \ref{propCocycle0} the following
result.
\begin{cor}\label{AlmostThm}
$R_{cm}(z,\xi)$ is a 
unitary dynamical $R$-matrix and $K_{cm}(z,\xi)$ is a right unitary dynamical $K$-matrix associated to $R_{cm}(z,\xi)$.
\end{cor}
In the following subsection we will constructed gauges that transform $R_{cm}(z,\xi)$ into Baxter's \cite{Ba} dynamical $R$-matrix associated to the eight vertex face model and $K_{cm}(z,\xi)$
into associated right dynamical $K$-matrices.

\subsection{Gauges of dynamical $R$- and $K$-matrices}\label{gaugesec}
We give two types of gauge transformations for a special class of dynamical $R$- and $K$-matrices. We only consider the
case that $V=\mathbb{C}^2$ and $h$ is the linear operator on $\mathbb{C}^2$ defined by $hv_\epsilon=\epsilon v_\epsilon$.

\begin{prop}\label{propgauge}
Suppose that 
\begin{equation*}
R(z,\xi):=\left(\begin{matrix} 1 & 0 & 0 & 0\\
0 & A(z,\xi) & B(z,\xi) & 0\\
0 & B(z,-\xi) & A(z,-\xi) & 0\\
0 & 0 & 0 & 1\end{matrix}\right)
\end{equation*}
is a unitary dynamical $R$-matrix and that
\begin{equation*}
K(z,\xi):=\left(\begin{matrix}\alpha(z,\xi) & \beta(z,\xi)\\
\beta(z,-\xi) & \alpha(z,-\xi)\end{matrix}\right)
\end{equation*}
is a right unitary dynamical $K$-matrix associated to $R(z,\xi)$. Then
\begin{equation*}
R^g(z,\xi):=\left(\begin{matrix} 1 & 0 & 0 & 0\\
0 & A^g(z,\xi) & B^g(z,\xi) & 0\\
0 & B^g(z,-\xi) & A^g(z,-\xi) & 0\\
0 & 0 & 0 & 1\end{matrix}\right)
\end{equation*}
is a unitary dynamical $R$-matrix and 
\begin{equation*}
K^g(z,\xi):=\left(\begin{matrix}\alpha^g(z,\xi) & \beta^g(z,\xi)\\
\beta^g(z,-\xi) & \alpha^g(z,-\xi)\end{matrix}\right)
\end{equation*}
a right unitary dynamical $K$-matrix associated to $R^g(z,\xi)$ if
\begin{enumerate}
\item[{\bf (i)}] $A^g(z,\xi)=u(\xi)A(z,\xi)$, $B^g(z,\xi)=B(z,\xi)$ and $K^g(z,\xi)=K(z,\xi)$ with $u(\xi)$ a meromorphic function
satisfying $u(\xi)u(-\xi)=1$,
\end{enumerate}
and if
\begin{enumerate}
\item[{\bf (ii)}] $A^g(z,\xi)=q^{2\kappa(\xi-z)}A(z,\xi)$, $B^g(z,\xi)=q^{-z(2\kappa+\xi)}B(z,\xi)$,
$\alpha^g(z,\xi)=q^{2z\xi}\alpha(z,\xi)$ and $\beta^g(z,\xi)=\beta(z,\xi)$.
\end{enumerate}
\end{prop}
\begin{proof}
This follows by a tedious but direct computation.
\end{proof}
\begin{rema}
The type {\bf (i)} gauge transformations for dynamical $R$-matrices were considered before in \cite{EV}.
\end{rema}

Baxter's dynamical $R$-matrix $R_{Ba}(z,\xi)$ corresponding to the eight vertex face model \cite{Ba,ABF} 
is defined as 
\begin{equation}\label{BaSol}
R_{Ba}(z,\xi):=\left(\begin{matrix}1 & 0 & 0 & 0\\ 0 & A_{Ba}(z,\xi) & B_{Ba}(z,\xi) & 0\\
0 & B_{Ba}(z,-\xi) & A_{Ba}(z,-\xi) & 0\\
0 & 0 & 0 & 1\end{matrix}
\right)
\end{equation}
with the matrix coefficients given by
\[
A_{Ba}(z,\xi):=\frac{\theta(q^{-z},q^{2\kappa+\xi};q)}{\theta(q^{2\kappa-z},q^\xi;q)},\qquad B_{Ba}(z,\xi):=
\frac{\theta(q^{-z-\xi},q^{2\kappa};q)}{\theta(q^{2\kappa-z},q^{-\xi};q)}.
\]
Note that $R_{Ba}(z,w)$ is unitary, satisfies the ice-rule and is dynamically $P$-symmetric.
Write 
\begin{equation}\label{KBa}
K_{Ba}(z,\xi):=\left(\begin{matrix}\alpha_{Ba}(z,\xi) & \beta_{Ba}(z,\xi)\\ \beta_{Ba}(z,-\xi) & \alpha_{Ba}(z,-\xi)\end{matrix}\right)
\end{equation}
with
\begin{equation}\label{alphabeta}
\begin{split}
\alpha_{Ba}(z,\xi)&:=\alpha_{cm}(z,\xi)q^{2z\xi}=\left(\frac{\mathcal{C}(z,\xi)-\widetilde{\mathcal{C}}(\xi,z)}{\widetilde{\mathcal{C}}(\xi,-z)}\right)q^{2z\xi},\\
\beta_{Ba}(z,\xi)&:=\beta_{cm}(z,\xi)=\frac{\mathcal{C}(z,\xi)}{\widetilde{\mathcal{C}}(-\xi,-z)}
\end{split}
\end{equation}
(recall the definition of $\mathcal{C}(z,\xi)$ from \eqref{calC}).
If we need to specify the dependence of $K_{Ba}(z,\xi)$ on the parameters $\zeta,\zeta^\prime,\upsilon,\upsilon^\prime$ then we write
$K_{Ba}(z,\xi;\zeta,\zeta^\prime,\upsilon,\upsilon^\prime)$.

\begin{cor}\label{KmatrixTHM}
$R_{Ba}(z,\xi)$ is a unitary dynamical $R$-matrix and $K_{Ba}(z,\xi)$ is a unitary right dynamical $K$-matrix associated to $R_{Ba}(z,\xi)$.
\end{cor}
\begin{proof}
Apply to the unitary dynamical $R$-matrix $R_{cm}(z,\xi)$ and the associated unitary right dynamical $K$-matrix $K_{cm}(z,\xi)$
gauge {\bf (i)}
 with
 \[
u(\xi)=\frac{\theta(q^{2\kappa+\xi},q^{-\xi};q)}{\theta(q^{2\kappa-\xi},q^\xi;q\bigr)}
\]
and apply subsequently gauge {\bf (ii)}. Then we obtain $R_{Ba}(z,\xi)$ and $K_{Ba}(z,\xi)$. Hence $R_{Ba}(z,\xi)$ is a unitary dynamical $R$-matrix
and $K_{Ba}(z,\xi)$ an associated unitary right dynamical $K$-matrix as a consequence 
of Corollary \ref{AlmostThm}.
\end{proof}
The fact that $R_{Ba}(z,\xi)$ is a unitary solution of the dynamical quantum Yang-Baxter equation goes back to Baxter \cite{Ba} (the dynamical quantum Yang-Baxter equation in \cite{Ba} takes the form of star-triangle equations, see \cite{Fe2} and the next subsection for more details on this viewpoint). 
The interpretation of Baxter's solution $R_{Ba}(z,\xi)$ in terms of connection matrices for the Frenkel-Reshetikhin-Smirnov qKZ equations goes back to \cite{FR}
(see also \cite{Ko,M}). Other examples of unitary dynamical $R$-matrices have been
constructed in, e.g., \cite{DJMO,JKMO,JMO,Ko}. Some of these examples are constructed using fusion techniques, others by computing connection matrices. 
 A quantum group context for dynamical $R$-matrices
was developed by Felder, Etingof and Varchenko, see, e.g.,
 \cite{Fe,EV,EV2}.
 
 Baxter's unitary dynamical $R$-matrix $R_{Ba}(z,\xi)$ satisfies dynamical crossing symmetry. In coordinates it reads
\[
\delta_2\epsilon_2R_{\epsilon_1,-\delta_2}^{\delta_1,-\epsilon_2}(z,\xi)=q^{-\kappa(1+\epsilon_2)}
\frac{\theta(q^{\xi+2\kappa\epsilon_2},q^z;q)}{\theta(q^\xi,q^{z-2\kappa};q)}R_{\delta_1\delta_2}^{\epsilon_1\epsilon_2}(2\kappa-z,\xi+2\kappa\epsilon_2),
\]
where we have written $R_{Ba}(z,\xi)v_{\delta_1}\otimes v_{\delta_2}=\sum_{\epsilon_1,\epsilon_2}R_{\delta_1\delta_2}^{\epsilon_1\epsilon_2}(z,\xi)v_{\epsilon_1}\otimes v_{\epsilon_2}$.
In coordinate free dynamical notations it reads
 \begin{equation}\label{crossingsymmetry}
 \sigma_2^yR_{Ba}(z,\xi)^{T_1}\sigma_2^y=q^{-\kappa(1+h_2)}\frac{\theta(q^{\xi+2\kappa h_2},q^z;q)}{\theta(q^\xi,q^{z-2\kappa};q)}
 R_{Ba}(2\kappa-z,\xi+2\kappa\underline{h}_2)
 \end{equation}
where $\sigma^y=\left(\begin{matrix} 0 & -\sqrt{-1}\\ \sqrt{-1} & 0\end{matrix}\right)$ and $T_{1}$ is transposition in the first tensor component. 

In Corollary  \ref{KmatrixTHM} we have constructed a $4$-parameter family of unitary right dynamical $K$-matrices associated to $R_{Ba}(z,\xi)$ by relating them to connection matrices of boundary qKZ equations. Dynamical $K$-matrices associated to $R_{Ba}(z,\xi)$ have been constructed before by direct computations in the study of the eight vertex solid-on-solid model with integrable reflecting boundaries  (see \cite{dVG, IK, HSFY} from the vertex perspective, \cite{FHS,BP,KH} from the face perspective, and \cite{HK} for the explicit link by vertex-face correspondences). In particular, see \cite[\S 4.3]{BP} for alternative parametrisations of unitary right dynamical $K$-matrices associated
to $R_{Ba}(z,\xi)$ depending on four parameters.
We will explain the link to the eight vertex face model with integrable reflecting boundaries in more detail in the following subsection.

\begin{rema}\label{decoupleremark}
{\bf (i)} Let $S$ be the spin-reversal operator on $\mathbb{C}^2$, defined as the linear map $S: \mathbb{C}^2\rightarrow\mathbb{C}^2$ satisfying
$S(v_+)=v_-$ and $S(v_-)=v_+$. Then $K_{Ba}(z,w)$ has the spin-reversal symmetry
\begin{equation}\label{spinreversalsymmetry}
SK_{Ba}(z,\xi)S^{-1}=K_{Ba}(z,-\xi).
\end{equation}
{\bf (ii)}
The matrix coefficient $\beta_{Ba}(z,\xi)$ of $K_{Ba}(z,\xi)$ decouples,
\[
\beta_{Ba}(z,\xi)=\frac{\widetilde{\mu}(\xi)}{\mu(-z)}
\]
with
\[
\mu(z):=\frac{\theta(aq^z,bq^z,cq^z,dq^z;q)}{\theta(q^{2z};q)}q^{2(\zeta+\zeta^\prime)z}
\]
and with $\widetilde{\mu}(z)$ obtained from $\mu(z)$ by interchanging $\nu$ and $\zeta^\prime$.
The function $\mu(z)$ is a natural elliptic analog of the $c$-function associated to the Askey-Wilson polynomials \cite{AW}. It naturally appears 
in the study of the Askey-Wilson function transform (see \cite{KS} and \cite[Lem. 4.4]{SAnn}), as well as in
Rains'  \cite{Ra} difference operator acting on $\textup{BC}$-symmetric theta functions, which plays an important role
in the study of Rains' interpolation theta functions.\\
{\bf (iii)} In \cite{F} a reflection equation different from \eqref{reflectioneqn} is considered in relation to Baxter's dynamical $R$-matrix $R_{Ba}(z,\xi)$.
It does not involve shifts in the dynamical parameter.
The one-parameter family of diagonal solutions from \cite{F} appears to be unrelated to the four-parameter family of dynamical $K$-matrices
$K_{Ba}(z,\xi)$.\\
{\bf (iv)} The dynamical $K$-matrix $K_{Ba}(z,\xi)$ is anti-diagonal iff $\mathcal{C}(z,\xi)=\widetilde{\mathcal{C}}(\xi,z)$. An explicit discrete set
of parameter values for which this is the case has been derived in \cite[\S 3]{StSIGMA}, see also \cite{S2}. In \cite{StSIGMA, S2} the 
condition $\mathcal{C}(z,\xi)=\widetilde{\mathcal{C}}(\xi,z)$ is related to
the theory of reflectionless Askey-Wilson second-order $q$-difference operators and to the theory of Baker-Akhiezer functions \cite{CE}.
\end{rema}

\subsection{Face reformulation}\label{8section}
Felder \cite{Fe} observed that solutions of the dynamical quantum Yang-Baxter equation satisfying the ice-rule are in one-to-one correspondence
to solutions of the star-triangle equation \cite{Ba}, see also \cite[\S 3]{Ro}. We recall it here, and extend the correspondence to associated dynamical $K$-matrices.
We will match it with the notations from \cite{BP} in order to facilitate the comparison of our $4$-parameter family of dynamical $K$-matrices from Corollary \ref{KmatrixTHM}
with the $4$-parameter family of solutions of the boundary Yang-Baxter equation from \cite[\S 4.3]{BP}.

In this subsection we fix a finite dimensional complex vector space $V$ and a semisimple linear operator $h$ with simple
spectrum $I$ (in many examples $V$ is the vector space underlying a finite dimensional irreducible 
$\mathfrak{sl}_2(\mathbb{C})$-representation and $h$ is the Cartan generator of $\mathfrak{sl}_2(\mathbb{C})$, in which case
$h$ acts semisimply on $V$ with simple spectrum given by
$\{k-2l\}_{l=0}^k$ for some $k\in\mathbb{N}$). We fix a linear basis $\{v_\mu\}_{\mu\in I}$ of $V$ satisfying $hv_\mu=\mu v_\mu$. 

Suppose $R(z,\xi)$ is a linear operator on $V\otimes V$ depending meromorphically on $(z,\xi)\in\mathbb{C}^2$ and write
\[
R(z,\xi)v_{\mu_1}\otimes v_{\mu_2}=\sum_{\nu_1,\nu_2\in I}R^{\nu_1\nu_2}_{\mu_1\mu_2}(z,\xi)v_{\nu_1}\otimes v_{\nu_2}
\]
for $\mu_1,\mu_2\in I$. For $\mu_i,\nu_i\in\mathbb{C}$ we set $R_{\mu_1\mu_2}^{\nu_1\nu_2}(z,\xi)\equiv 0$ if $(\mu_1,\mu_2,\nu_1,\nu_2)\not\in I^{\times 4}$.
Note that the ice rule \eqref{ice} for $R(z,\xi)$ is equivalent to the property that
$R_{\mu_1\mu_2}^{\nu_1\nu_2}(z,\xi)\equiv 0$ unless $\mu_1+\mu_2=\nu_1+\nu_2$. Set for $a,b\in\mathbb{C}$,
\begin{equation*}
A_{ab}:=\begin{cases}
1\qquad &\hbox{ if } b-a\in I,\\
0\qquad &\hbox{ if } b-a\not\in I.
\end{cases}
\end{equation*}

Define for $a,b,c,d\in\mathbb{C}$ the Boltzmann weight
\begin{equation}\label{WeightST}
W\left(\begin{matrix} a & b\\ c & d\end{matrix}\,\,\vline\,\, z,\xi\right):= R^{d-c,c-a}_{b-a,d-b}(z,\xi-2\kappa a).
\end{equation}
Note that $W\left(\begin{matrix} a & b\\ c & d\end{matrix}\,\,\vline\,\, z,\xi\right)\equiv 0$ unless $A_{ab}A_{ac}A_{bd}A_{cd}=1$.
By \cite[\S 5]{Fe2} (see also \cite[\S 3]{Ro}), we have the following result.
\begin{prop}
Suppose that $R(z,\xi): V\otimes V\rightarrow V\otimes V$ satisfies the ice rule \eqref{ice}. Then $R(z,\xi)$ is a solution of the dynamical quantum
Yang-Baxter equation \eqref{qdYBintro} if and only if the corresponding Boltzmann weights \eqref{WeightST} satisfy the star triangle equations
\begin{equation}\label{startriangle}
\begin{split}
\sum_gW&\left(\begin{matrix} f & g\\ a & b\end{matrix}\,\, \vline \,\, z_1-z_2,\xi\right)
W\left(\begin{matrix} g & d\\ b & c\end{matrix}\,\, \vline\,\, z_1-z_3,\xi\right)
W\left(\begin{matrix} f & e\\ g & d\end{matrix}\,\, \vline\,\, z_2-z_3,\xi\right)\\
&\,\,=\sum_gW\left(\begin{matrix} a & g\\ b & c\end{matrix}\,\, \vline\,\, z_2-z_3,\xi\right)
W\left(\begin{matrix} f & e\\ a & g\end{matrix}\,\, \vline\,\, z_1-z_3,\xi\right)
W\left(\begin{matrix} e & d\\ g & c\end{matrix}\,\, \vline\,\, z_1-z_2,\xi\right)
\end{split}
\end{equation}
for all $a,b,c,d,e,f\in\mathbb{C}$.
\end{prop}
\begin{rema}\label{Aremark}
Note that \eqref{startriangle} is a nontrivial identity only if $A_{ab}A_{bc}A_{dc}A_{ed}A_{fe}A_{fa}=1$. Furthermore, for fixed
$a,b,c,d,e,d$ the left and right hand sides of \eqref{startriangle} are well defined since the summand on the left (resp. right) is zero unless $g\in\mathbb{C}$
is in the finite (possibly empty) set of complex numbers $g$ satisfying $A_{fg}A_{gb}A_{gd}=1$ (resp. $A_{ag}A_{eg}A_{gc}=1$). With these two remarks it is clear
that  \eqref{startriangle} coincides with the star triangle equations for face models, see, e.g., \cite[(2.3)]{BP}.
\end{rema}
Unitarity \eqref{unitary} of $R(z,\xi)$ is equivalent to the inversion relation
\[
\sum_e W\left(\begin{matrix} d & e\\ a & b\end{matrix}\,\, \vline\,\, z,\xi\right)
W\left(\begin{matrix} d & c\\ e & b\end{matrix}\,\, \vline\,\, -z,\xi\right)=\delta_{ac}
\]
with $\delta_{ac}$ the Kronecker delta function,
cf. \cite[(2.13)]{BP}.

Boundary star-triangle equations are defined in \cite{BPO,BP,Ku}, leading to integrable face models with boundary. We now proceed to explain its equivalence
to the right dynamical reflection equation \eqref{reflectioneqn}. Let $K(z,\xi): V\rightarrow V$ be a linear operator depending meromorphically on
$(z,\xi)\in\mathbb{C}^2$ and write
\[
K(z,\xi)v_{\mu}=\sum_{\nu\in I}K_{\mu}^{\nu}(z,\xi)v_{\nu}
\]
for $\mu\in I$. For $\mu,\nu\in\mathbb{C}$ we set $K_{\mu}^{\nu}(z,\xi)\equiv 0$ if $(\mu,\nu)\not\in I\times I$.
Define for $a,b,c\in \mathbb{C}$ the boundary Boltzmann weight
\begin{equation}\label{BWeightST}
B\Bigl(b\,\,\, \begin{matrix} c\\a\end{matrix} \,\,\vline\,\, z,\xi\Bigr):=K_{c-b}^{a-b}\Bigl(z,\frac{\xi}{2}-\kappa b\Bigr).
\end{equation}
Note that $B\Bigl(b\,\,\, \begin{matrix} c\\a\end{matrix} \,\,\vline\,\, z,\xi\Bigr)\equiv 0$ unless $A_{ba}A_{bc}=1$.
The following result can be checked by a direct computation.
\begin{prop}\label{corrtoface}
Assume that $R(z,\xi): V\otimes V\rightarrow V\otimes V$ is a dynamical $R$-matrix satisfying the ice rule \eqref{ice}. Define the associated Boltzmann weights
by \eqref{WeightST}.
Then $K(z,\xi)$ is a right dynamical $K$-matrix with respect to $R(z,\xi)$ if and only if the corresponding boundary Boltzmann weights \eqref{BWeightST} satisfy the 
boundary Yang-Baxter equations
\begin{equation}\label{bYBeqn}
\begin{split}
&\sum_{f,g}W\left(\begin{matrix} c & f\\ d & e\end{matrix}\,\,\vline\,\, z_1-z_2,\xi\right)
B\left(f\,\,\, \begin{matrix} g\\e\end{matrix} \,\,\vline\,\, z_1,\xi\right)
W\left(\begin{matrix} c & b\\ f & g\end{matrix}\,\,\vline\,\, z_1+z_2,\xi\right)
B\left(b\,\,\, \begin{matrix} a\\g\end{matrix} \,\,\vline\,\, z_2,\xi\right)=\\
&=\sum_{f,g}B\left(d\,\,\, \begin{matrix} g\\e\end{matrix} \,\,\vline\,\, z_2,\xi\right)
W\left(\begin{matrix} c & f\\ d & g\end{matrix}\,\,\vline\,\, z_1+z_2,\xi\right)
B\left(f\,\,\, \begin{matrix} a\\g\end{matrix} \,\,\vline\,\, z_1,\xi\right)
W\left(\begin{matrix} c & b\\ f & a\end{matrix}\,\,\vline\,\, z_1-z_2,\xi\right)
\end{split}
\end{equation}
for all $a,b,c,d,e\in\mathbb{C}$.
\end{prop}
By a similar reasoning as in Remark \ref{Aremark} one shows that the equations \eqref{bYBeqn} are equivalent to the boundary Yang-Baxter equations \cite[(2.4)]{BP}.
Note that unitarity \eqref{unitaryK} of $K(z,\xi)$ is equivalent to the boundary inversion relation
\[
\sum_dB\left(b\,\,\, \begin{matrix} d\\c\end{matrix} \,\,\vline\,\, z,\xi\right)B\left(b\,\,\, \begin{matrix} a\\d\end{matrix} \,\,\vline\,\, -z,\xi\right)=\delta_{ac},
\]
cf. \cite[(2.14)]{BP}.

Solutions of the star triangle \eqref{startriangle} and boundary Yang-Baxter equations \eqref{bYBeqn} give rise to 
two-dimensional integrable face models with boundary. We refer to \cite{BP,BPO} and references therein for details.
I will finish this subsection by giving the explicit formulas relating $R_{Ba}(z,\xi)$ to the Boltzmann weights of the eight vertex face model
and subsequently pushing the associated $4$-parameter solutions of the right dynamical reflection equation (Corollary \ref{KmatrixTHM}) to a $4$-parameter
family of solutions of the associated boundary Yang-Baxter equations. So in the remainder of this section we take $V=\mathbb{C}^2=\mathbb{C}v_+\oplus \mathbb{C}v_-$
and $hv_{\pm}=\pm v_{\pm}$, so that $I=\{-1,1\}$ and $A_{ab}=1$ iff $|a-b|=1$.

We gauge $R_{Ba}(z,\xi)$ to the unitary dynamical $R$-matrix $R_{8vSOS}(z,\xi)$ given by
\begin{equation*}
R_{8vSOS}(z,\xi):=q^{\frac{z}{2}}\frac{\theta(q^{2\kappa-z};q)}{\theta(q^{2\kappa};q)}\left(
\begin{matrix} 1 & 0 & 0 & 0\\
0 & u(\xi)A_{Ba}(z,\xi) & B_{Ba}(z,\xi) & 0\\
0 & B_{Ba}(z,-\xi) & u(-\xi)A_{Ba}(z,-\xi) & 0\\
0 & 0 & 0 & 1\end{matrix}
\right),
\end{equation*}
with 
\[
u(\xi):=q^{-\kappa}\Bigl(\frac{\theta(q^{-\xi+2\kappa};q)}{\theta(q^{-\xi-2\kappa};q)}\Bigr)^{\frac{1}{2}}.
\]
Denote the corresponding Boltzmann weights (see \eqref{WeightST}) by $W_{8vSOS}\left(\begin{matrix} a & b\\ c & d \end{matrix}\,\,\vline\,\, z,\xi\right)$. Then the nonzero
Boltzmann weights are 
\begin{equation*}
\begin{split}
W_{8vSOS}\left(\begin{matrix} a & a\pm 1\\ a\pm 1 & a\end{matrix} \,\,\vline\,\, z,\xi\right)&=q^{\mp \frac{z}{2}}\frac{\theta(q^{\pm z-\xi+2\kappa a};q)}
{\theta(q^{-\xi+2\kappa a};q)},\\
W_{8vSOS}\left(\begin{matrix} a\pm 1 & a\\ a & a\mp 1\end{matrix}\,\,\vline\,\, z,\xi\right)&=q^{\frac{z}{2}}\frac{\theta(q^{2\kappa-z};q)}{\theta(q^{2\kappa};q)},\\
W_{8vSOS}\left(\begin{matrix} a & a\pm 1\\ a\mp 1 & a\end{matrix}\,\, \vline\,\, z,\xi\right)&=
q^\kappa\Bigl(\frac{\theta(q^{-\xi+2\kappa(a-1)},q^{-\xi+2\kappa(a+1)};q)}{\theta(q^{-\xi+2\kappa a};q)^2}\Bigr)^{\frac{1}{2}}
q^{\frac{z}{2}}\frac{\theta(q^{-z};q)}{\theta(q^{2\kappa};q)},
\end{split}
\end{equation*}
which coincide with the standard weights of the eight-vertex face model. To relate it to the notations in \cite{BP}, note that
it corresponds to the Boltzmann weights \cite[(4.4)]{BP} by identifying the parameters $\lambda,w_0,u$ in \cite[(4.4)]{BP} with our parameters $2\kappa,-\xi,z$
(note though that there is a small typo in \cite[(4.4)]{BP}). The dynamical crossing symmetry \eqref{crossingsymmetry} of $R_{Ba}(z,\xi)$ corresponds to the
crossing symmetry \cite[\S 4.1 \& (2.11)]{BP} of the Boltzmann weights $W_{8vSOS}$.

We write $B_{8vSOS}\Bigl(b\,\,\, \begin{matrix} c\\a\end{matrix} \,\,\vline\,\, z,\xi\Bigr)$ for the boundary Boltzmann weights associated to the dynamical $K$-matrix
$K_{Ba}(z,\xi)$ (see \eqref{BWeightST}). The nonzero values are
\begin{equation}\label{B8vSOS}
\begin{split}
B_{8vSOS}\Bigl(a\,\,\, \begin{matrix} a\pm 1\\a\pm 1\end{matrix} \,\,\vline\,\, z,\xi\Bigr)&=
K_{Ba,\pm}^{\pm}\bigl(z,\frac{\xi}{2}-\kappa a\bigr)=\alpha_{Ba}\bigl(z,\pm\frac{\xi}{2}\mp\kappa a\bigr),\\
B_{8vSOS}\Bigl(a\,\,\, \begin{matrix} a\mp 1\\a\pm 1\end{matrix} \,\,\vline\,\, z,\xi\Bigr)&=K_{Ba,\mp}^{\pm}\bigl(z,\frac{\xi}{2}-\kappa a)=
\beta_{Ba}\bigl(z,\pm\frac{\xi}{2}\mp\kappa a\bigr),
\end{split}
\end{equation}
where $\alpha_{Ba}$ and $\beta_{Ba}$ are explicitly given by \eqref{alphabeta}.
\begin{cor}
$B_{8vSOS}$ is a $4$-parameter family of solutions of the boundary Yang-Baxter equations \eqref{bYBeqn} with respect to the Boltzmann weights
$W_{8vSOS}$.
\end{cor}
\begin{proof}
By Proposition \ref{propgauge} the matrix $K_{Ba}(z,\xi)$ 
is also a right dynamical $K$-matrix with respect to the gauged dynamical $R$-matrix $R_{8vSOS}(z,\xi)$. 
Now use Proposition \ref{corrtoface}.
\end{proof}
Solutions to the boundary Yang-Baxter equations for the eight vertex face model, or equivalently, by vertex-face correspondences, to reflection equations for the eight vertex model itself, have been computed by direct means in \cite{dVG, IK, HSFY} from the vertex perspective and in \cite{FHS,BP,KH} from the face perspective (see \cite{HK} for a discussion of the vertex-face correspondence in this context). See \cite[4.3]{BP} for a direct derivation of the solutions of the boundary Yang-Baxter equations for the eight vertex face model, which are also parametrised (but in a different way) by four degrees of freedoms.

\subsection{Boundary qKZB equations}

A unitary solution of the dynamical quantum Yang-Baxter equation satisfying the ice-rule gives rise to cocycles of the extended affine symmetric group \cite{Fe2}.
This leads to a compatible system of $q$-difference equations, known as the quantum Knizhnik-Zamolodchikov-Bernard (KZB) equations, see \cite{Fe2,FTV}. In this
subsection we establish the analog of this result in the presence of boundaries, resulting in a cocycle of the affine Weyl group of type $C_n$ and a new type of compatible
systems of  $q$-difference equations, which we will call boundary quantum KZB equations.

We return to the setting of Subsection \ref{DynRK}.
If $f: \mathbb{C}\rightarrow V^{\otimes n}$ is a meromorphic $V^{\otimes n}$-valued function then we define
\[
(T_{\alpha h_i}f)(\xi):=\sum_{\underline{\mu}}f_{\underline{\mu}}(\xi+\alpha\mu_i)
\]
if $f(\xi)=\sum_{\underline{\mu}} f_{\underline{\mu}}(\xi)$ is the decomposition of $f$ as the sum of $\bigl(V^{\otimes n}\bigr)_{\underline{\mu}}$-valued meromorphic functions 
$f_{\underline{\mu}}$.
In case of a meromorphic $V^{\otimes n}$-valued function $f(\mathbf{z},\xi)$, the parameter-shift of $T_{\alpha h_i}$ is in the dynamical parameter
$\xi$. When $n=1$ then we write $T_{\alpha h}$ for $T_{\alpha h_1}$.

Let $S(\mathbf{z},\xi)$ be a family of linear operators on $V^{\otimes n}$ depending meromorphically on $(\mathbf{z},\xi)\in\mathbb{C}^n\times\mathbb{C}$. 
Consider the associated matrix difference operator
$T_{-\kappa h_i}ST_{\kappa h_i}$, acting complex linearly on the space $\mathcal{M}(\mathbb{C}^n\times\mathbb{C})\otimes V^{\otimes n}$ of meromorphic
$V^{\otimes n}$-valued functions on $\mathbb{C}^n\times\mathbb{C}$. It is a matrix valued difference operator in the dynamical parameter $\xi$, depending meromorphically
on $\mathbf{z}\in\mathbb{C}^n$. We therefore also denote it by $T_{-\kappa h_i}S(\mathbf{z},\cdot)T_{-\kappa h_i}$.

Let $\underline{K}(z,\xi)$  be a linear operators on $V$ depending meromorphically on $(z,\xi)\in\mathbb{C}^2$. Let 
$P: V\otimes V\rightarrow V\otimes V$
be the permutation operator and consider the  $\textup{End}_{\mathbb{C}}(V^{\otimes n})$-valued difference operator in $w$ defined by
\begin{equation}\label{Mrelnew0}
\begin{split}
M^{s_0}(\mathbf{z}):=&(T_{-\kappa h}\underline{K}(\frac{1}{2}-z_1,\cdot)T_{\kappa h})_1\\
=&T_{-\kappa h_1}\underline{K}_1(\frac{1}{2}-z_1,\cdot)T_{\kappa h_1}.
\end{split}
\end{equation}
\begin{rema}\label{diagonalKrem}
Let $\pi_\mu: V\rightarrow V_\mu$ be the projection along the decomposition $V=\bigoplus_{\nu}V_\nu$. 
For $f(\xi)\in V$ depending meromorphically on $\xi$ we have 
\[
\bigl(\bigl(T_{-\kappa h}\underline{K}(z,\cdot)T_{\kappa h}\bigr)f\bigr)(\xi)=\sum_{\mu,\nu}\underline{K}_\nu^\mu(z,\xi-\kappa\nu)
f_\mu(\xi+\kappa(\mu-\nu))
\]
where $f_\mu(\xi):=\pi_\mu(f(\xi))$ and $\underline{K}_\nu^\mu(z,\xi):=\pi_\nu\underline{K}(z,\xi)|_{V_\mu}$. Hence if the left dynamical $K$-matrix $\underline{K}(z,\xi)$ satisfies
\begin{equation}\label{iceK}
\lbrack \underline{K}(z,\xi),h\rbrack=0
\end{equation}
then $M^{s_0}(\mathbf{z})=\underline{K}_1(\frac{1}{2}-z_1,\cdot -\kappa h_1)$ (which thus simply acts on the space of meromorphic $V^{\otimes n}$-valued functions in 
$(\mathbf{z},\xi)\in\mathbb{C}^n\times\mathbb{C}$ as the linear operator $\underline{K}_1(\frac{1}{2}-z_1,\xi-\kappa h_1)$).
\end{rema}
We have now the following affine version of Proposition \ref{propCocycle0}.
\begin{prop}\label{propCocyclenew}
Let $R(z,\xi)$ be a linear operator on $V\otimes V$
and $K(z,\xi)$, $\underline{K}(z,\xi)$  linear operators on $V$, all depending meromorphically on $(z,\xi)\in\mathbb{C}^2$.
Suppose that $R(z,\xi)$ satisfies the ice-rule \eqref{ice}.
The following two statements are equivalent.
\begin{enumerate}
\item The operators \eqref{Mrelnew0} and \eqref{Mrelnew} are part of a (necessarily unique) set $\{M^v(\mathbf{z})\}_{v\in W}$ of 
$\textup{End}_{\mathbb{C}}(V^{\otimes n})$-valued meromorphic difference operators $M^v(\mathbf{z})$
in $\xi$, depending meromorphically on $\mathbf{z}\in\mathbb{C}^n$, and 
satisfying the cocycle conditions
\[
M^e(\mathbf{z})=\textup{Id},\qquad M^{uv}(\mathbf{z})=
M^u(\mathbf{z})M^v(u^{-1}\mathbf{z})
\]
for all $u,v\in W$;
\item $R(z,\xi)$ is a unitary dynamical $R$-matrix, and $\underline{K}(z,\xi)$ and $K(z,\xi)$ are associated unitary left and right
dynamical $K$-matrices respectively.
\end{enumerate}
\end{prop}
\begin{proof}
For $v\in W_0$ the $M^v(\mathbf{z})$ coincide with the cocycle values from Proposition \ref{propCocycle0}. 
 The extension to the affine setup now follows by a direct computation.
\end{proof}
Proposition \ref{propCocyclenew} should be compared to \cite[Prop. 4.1 \& Thm. 4.2]{Fe2} (see also \cite{FTV}), which deals with
the extended affine symmetric group.

\begin{rema}\label{totrivia}
Note that the cocycle values $M^v(\mathbf{z})$ ($v\in W$) are $\textup{End}_{\mathbb{C}}\bigl(V^{\otimes n}\bigr)$-valued difference operators in $\xi$, depending meromorphically
on $\mathbf{z}\in\mathbb{C}^n$. For $v\in W_0$ they simply act as multiplication operators,
\[(M^v(\mathbf{z})f(\mathbf{z},\cdot))(\xi)=M^v(\mathbf{z},\xi)f(\mathbf{z},\xi),\qquad v\in W_0,
\]
with the linear operator $M^v(\mathbf{z},\xi)$ on $V^{\otimes n}$ 
as defined in Proposition \ref{propCocycle0}. In fact, if $\underline{K}(\mathbf{z},\xi)$ satisfies \eqref{iceK} then all cocycle values $M^v(\mathbf{z})$ ($v\in W$)
act as multiplication operators in view of Remark \ref{diagonalKrem}.
\end{rema}

\begin{defi}
Suppose that $R(z,\xi): V\otimes V\rightarrow V\otimes V$ is a unitary solution of the dynamical quantum Yang-Baxter equation satisfying the ice-rule
and suppose that $\underline{K}(z,\xi)$ and $K(z,\xi)$ are associated unitary left and right dynamical $K$-matrices. Let $\{M^u(\mathbf{z})\}_{u\in W}$ be the $W$-cocycle
of difference operators in $\xi$ as defined in Proposition \ref{propCocyclenew}. 

A $V^{\otimes n}$-valued meromorphic function $f(\mathbf{z},\xi)$ in $(\mathbf{z},\xi)\in\mathbb{C}^n\times\mathbb{C}$ is said to satisfy the boundary quantum Knizhnik-Zamolodchikov-Bernard (KZB) equations if
\begin{equation}\label{qKZB}
M^{\tau(\lambda)}(\mathbf{z})f(\mathbf{z}-\lambda,\,\cdot\,)=f(\mathbf{z},\,\cdot\,)\qquad \forall\,\lambda\in\mathbb{Z}^n.
\end{equation}
\end{defi}
The boundary quantum KZB equations \eqref{qKZB} form a compatible systems of difference equations in both $\mathbf{z}$ and $\xi$. 
Note that the translation action on the dynamical parameter $\xi$ is hidden in the transport operators $M^{\tau(\lambda)}(\mathbf{z})$ ($\lambda\in\mathbb{Z}^n$). 
\begin{rema}
The space of $V^{\otimes n}$-valued meromorphic functions $f(\mathbf{z},\xi)$ in $(\mathbf{z},\xi)\in\mathbb{C}^n\times\mathbb{C}$
satisfying the boundary quantum KZB equations \eqref{qKZB} is invariant for the $W_0$-action
\[
(v\cdot f)(\mathbf{z},\xi):=M^v(\mathbf{z},\xi)f(v^{-1}\mathbf{z},\xi),\qquad v\in W_0.
\]
\end{rema}

\begin{rema}\label{diagonalKrem2}
If the left dynamical $K$-matrix $\underline{K}(z,\xi)$ satisfies \eqref{iceK}, then the transport operators of the 
boundary quantum KZB equations are linear operators $M^{\tau(\lambda)}(\mathbf{z},\xi)$ on $V^{\otimes n}$
depending meromorphically on $(\mathbf{z},\xi)$ (see Remark \ref{totrivia}). Hence, under the additional assumption \eqref{iceK} on $\underline{K}(z,\xi)$,
the boundary quantum KZB equations reduce to the compatible system
\begin{equation*}
M^{\tau(\lambda)}(\mathbf{z},\xi)f(\mathbf{z}-\lambda,\xi)=f(\mathbf{z},\xi)\qquad \forall\,\lambda\in\mathbb{Z}^n
\end{equation*}
of difference equations acting only on the spectral parameters $\mathbf{z}$.
\end{rema}
The boundary quantum KZB equations are equivalent to the set of equations
\begin{equation}\label{bqKZBalt}
M^{\tau(-e_i)}(\mathbf{z})f(\mathbf{z}+e_i,\cdot)=f(\mathbf{z},\cdot)\qquad i=1,\ldots,n.
\end{equation}
The left hand side of \eqref{bqKZBalt}
can be explicitly expressed in terms of the dynamical $R$-and $K$-matrices using the cocycle property of $\{M^v(\mathbf{z})\}_{v\in W}$
(Proposition \ref{propCocyclenew}) and \eqref{taui}. We give the explicit expression of \eqref{bqKZBalt} for $i=1$,
the explicit expressions for arbitrary $i$
then follow from 
\[
M^{\tau(-e_{i+1})}(\mathbf{z})=M^{s_i}(\mathbf{z})M^{\tau(-e_i)}(s_i\mathbf{z})M^{s_i}(s_i\mathbf{z}+e_i)
\]
recursively.
For $i=1$ the left hand side of \eqref{bqKZBalt} is
\begin{equation*}
\begin{split}
\bigl(M^{\tau(-e_1)}&(\mathbf{z})f(\mathbf{z}+e_1,\cdot)\bigr)(\xi)=R_{21}(z_1-z_{2},2\xi)R_{31}(z_1-z_{3},2\xi-2\kappa h_{2})\\
&\times\cdots\times R_{n1}(z_1-z_n,2\xi-2\kappa(h_{2}+\cdots+h_{n-1}))K_1(z_1,\xi-\kappa(h_2+h_3+\cdots +h_n))\\
&\times R_{1n}(z_1+z_n,2\xi-2\kappa(h_2+h_3+\cdots+h_{n-1}))\cdots R_{13}(z_1+z_3,2\xi-2\kappa h_2)\\
&\times R_{12}(z_1+z_2,2\xi)\sum_{\underline{\mu},\underline{\nu}}\underline{K}_{\nu_1}^{\mu_1}(\frac{1}{2}+z_1,\xi-\kappa\nu_1)
f_{\underline{\mu}}(\mathbf{z}+e_1,\xi+\kappa(\mu_1-\nu_1)).
\end{split}
\end{equation*}
In the special case that the left dynamical $K$-matrix $\underline{K}(z,\xi)$ satisfies \eqref{iceK}, the left hand side of \eqref{bqKZBalt} 
for $i=1$ takes on the simpler form
\begin{equation*}
\begin{split}
M^{\tau(-e_1)}&(\mathbf{z},\xi)f(\mathbf{z}+e_1,\xi)=R_{21}(z_1-z_{2},2\xi)R_{31}(z_1-z_{3},2\xi-2\kappa h_{2})\\
&\times\cdots\times R_{n1}(z_1-z_n,2\xi-2\kappa(h_{2}+\cdots+h_{n-1}))K_1(z_1,\xi-\kappa(h_2+h_3+\cdots +h_n))\\
&\times R_{1n}(z_1+z_n,2\xi-2\kappa(h_2+h_3+\cdots+h_{n-1}))\cdots R_{13}(z_1+z_3,2\xi-2\kappa h_2)\\
&\times R_{12}(z_1+z_2,2\xi)\underline{K}_1(\frac{1}{2}+z_1,\xi-\kappa h_1)f(\mathbf{z}+e_1,\xi)
\end{split}
\end{equation*}
due to Remark \ref{diagonalKrem} and Remark \ref{diagonalKrem2}.

We now apply the construction of the boundary quantum KZB equations to Baxter's dynamical $R$-matrix $R_{Ba}(z,\xi)$ and the 
associated $4$-parameter family of unitary dynamical $K$-matrices.
So in the following corollary we take $V=\mathbb{C}^2=\mathbb{C}v_+\oplus\mathbb{C}v_-$ and $hv_{\epsilon}=\epsilon v_\epsilon$ for $\epsilon\in\{\pm\}$.

\begin{cor}\label{9bqKZB}
Let $R_{Ba}(z,\xi)$ be Baxter's dynamical $R$-matrix given explicitly by \eqref{BaSol} (it depends 
on $q$ and the bulk coupling parameter $\kappa$).
Fix four left boundary coupling parameters $\zeta_l,\zeta_l^\prime,\upsilon_l,\upsilon_l^\prime$ and four right boundary coupling parameters
$\zeta,\zeta^\prime,\upsilon,\upsilon^\prime$ and set
\begin{equation*}
\begin{split}
\underline{K}_{Ba}^l(z,\xi)&:=K_{Ba}(z,-\xi;\zeta_l,\zeta_l^\prime,\upsilon_l,\upsilon_l^\prime),\\
K_{Ba}^r(z,\xi)&:=K_{Ba}(z,\xi;\zeta,\zeta^\prime,\upsilon,\upsilon^\prime)
\end{split}
\end{equation*}
with $K_{Ba}(z,\xi;\zeta,\zeta^\prime,\upsilon,\upsilon^\prime)$ given by \eqref{KBa}.
Then $\underline{K}_{Ba}^l(z,\xi)$ and $K_{Ba}^r(z,\xi)$ are unitary left and right 
dynamical $K$-matrices associated to $R_{Ba}(z,\xi)$ respectively. 
Let $\{M_{Ba}^v(\mathbf{z})\}_{v\in W}$ be the associated $W$-cocycle of $\textup{End}_{\mathbb{C}}\bigl(\bigl(\mathbb{C}^2\bigr)^{\otimes n}\bigr)$-valued difference operators 
in $\xi$ depending meromorphically on $\mathbf{z}\in\mathbb{C}^n$
(see Proposition \ref{propCocyclenew}). The corresponding quantum KZB equations
\begin{equation}\label{BaqKZB}
M^{\tau(\lambda)}_{Ba}(\mathbf{z})f(\mathbf{z}-\lambda,\cdot)=f(\mathbf{z},\cdot),\qquad \forall\,\lambda\in\mathbb{Z}^n
\end{equation}
for meromorphic $(\mathbb{C}^2)^{\otimes n}$-valued functions $f(\mathbf{z},\xi)$ in $(\mathbf{z},\xi)\in\mathbb{C}^n\times\mathbb{C}$
is a compatible system of difference equations in $\mathbf{z}$ and $\xi$ with elliptic coefficients that depends, besides on $q$, on nine 
coupling parameters $\kappa,\zeta_l,\zeta_l^\prime,\upsilon_l,\upsilon_l^\prime,
\zeta,\zeta^\prime,\upsilon,\upsilon^\prime$.
\end{cor}
\begin{proof}
The fact that $K^l(z,\xi)$ and $K^r(z,\xi)$ are left and right dynamical $K$-matrices respectively follows from Corollary \ref{KmatrixTHM} and Remark \ref{lrchange}. 
\end{proof}
\begin{rema}
It is natural to expect that a vertex-face transformation (cf. \cite{HK}) transforms the nine parameter family \eqref{BaqKZB} of elliptic boundary quantum KZB equations into
a nine parameter family of elliptic boundary qKZ equations defined in terms of the XYZ spin-$\frac{1}{2}$ $R$-matrix and its associated $K$-matrices. We expect
that a difference Cherednik-Matsuo type of correspondence will relate the solutions of these compatible systems of difference equations to common eigenfunctions of 
the nine parameter family of elliptic Ruijsenaars' systems of type C introduced by van Diejen \cite{vD} and Komori \& Hikami \cite{KH0}.
\end{rema}

\section{The connection problem for quantum affine KZ equations} 

Cherednik \cite[\S 1.3.2]{CQInd} introduced the
Baxterization of a finite dimensional affine Hecke algebra module using the affine intertwiners of the associated 
double affine Hecke algebra. It
leads to a consistent system of first-order difference equations acting on vector-valued multivariate meromorphic functions, 
called the quantum affine Knizhnik-Zamolodchikov (KZ)
equations. These equations are equivariant with respect to a natural action of the underlying Weyl group. 

A natural basis of power series solutions of the quantum affine KZ equations,
defined in terms of their asymptotic behaviour deep in a fixed Weyl chamber, was constructed in 
\cite{vMS,vM,SAnn} in case of minimal principal series. We extend this result to arbitrary principal series modules, and
we also solve the associated connection problem. 
the affine Hecke algebra of type $C$. 

We explain how a special case of this general theory leads to the results on lattice models with boundaries as discussed in the previous section.
\subsection{Notations}
We recall here some of the notations from \cite[\S 1.1 \& \S 2]{S2}, which are well 
suited to treat the twisted and untwisted Cherednik-Macdonald theory at the same time (see \cite{S2, SBook} 
for details). The initial data is given by a five tuple $D=(R_0,\Delta_0,\bullet,\Lambda,\widetilde{\Lambda})$, which is defined as follows.

The first entry $R_0$ is a finite, reduced crystallographic root system in an Euclidean space
$(E,(\cdot,\cdot))$, irreducible within the subspace $V$ of $E$ spanned by the roots. The second entry 
$\Delta_0=(\alpha_1,\ldots,\alpha_n)$ is an ordered basis of the root system $R_0$. The third entry 
$\bullet$ equals $u$ or $t$ ("$u$" stands for untwisted and "$t$" for twisted). The fourth and fifth entries $\Lambda$ and $\widetilde{\Lambda}$ are lattices
that are defined as follows.

Write $R_0=R_0^+\cup R_0^-$ for the decomposition of $R_0$ in positive and negative roots with respect to the basis $\Delta_0$. 
The root system $\widetilde{R}_0$
$\bullet$-dual to $R_0$ is defined as
\begin{equation*}
\widetilde{R}_0:=\{\widetilde{\alpha}:=\mu_\alpha\alpha^\vee\,\, | \,\, \alpha\in R_0\},
\end{equation*}
where $\alpha^\vee=\frac{2\alpha}{|\alpha|^2}$ is the co-root of $\alpha$ and
\begin{equation*}
\mu_\alpha:=
\begin{cases}
1 \qquad &\hbox{ if } \bullet=u,\\
\frac{|\alpha|^2}{2}\qquad &\hbox{ if } \bullet=t.
\end{cases}
\end{equation*}
So in the untwisted case ($\bullet=u$), $\widetilde{R}_0$ is the co-root system $R_0^\vee=\{\alpha^\vee\}_{\alpha\in R_0}$, while in the twisted case ($\bullet=t)$,
$\widetilde{R}_0=R_0$. Let $W_0\subseteq O(E)$ be the Weyl group of $R_0$, generated by the orthogonal reflections $s_\alpha$ in the hyperplanes $\alpha^{-1}(0)$
($\alpha\in R_0$). It is a Coxeter group with Coxeter generators the simple reflections $s_i:=s_{\alpha_i}=s_{\widetilde{\alpha}_i}$ ($i=1,\ldots,n$).

Write $Q$ and $\widetilde{Q}$ for the root lattices of $R_0$ and $\widetilde{R}_0$ respectively. Let $Q^\vee$ and $\widetilde{Q}^\vee$
be the co-root lattice of $R_0$ and $\widetilde{R}_0$ respectively. The fourth entry $\Lambda\subseteq E$ in the five-tuple $D$ of initial data 
is a full lattice in $E$ satisfying the conditions
\[
Q\subseteq\Lambda,\qquad (\Lambda,Q^\vee)\subseteq\mathbb{Z}.
\]
The fifth entry $\widetilde{\Lambda}\subseteq E$ is a full lattice in $E$ satisfying the conditions
\[
\widetilde{Q}\subseteq\widetilde{\Lambda},\qquad (\widetilde{\Lambda},\widetilde{Q}^\vee)\subseteq\mathbb{Z}.
\]

We associate to $D$ the reduced irreducible affine root system 
\[
R^\bullet:=\{\alpha^{(r)}:=\mu_\alpha rc+\alpha\}_{r\in\mathbb{Z},\alpha\in R_0},
\] 
where $\mu_\alpha rc+\alpha$ is the affine linear function  $v\mapsto \mu_\alpha r+(\alpha,v)$ on $E$. Let $s_{\alpha^{(r)}}$ be the orthogonal reflection in the affine hyperplane 
$\bigl(\alpha^{(r)}\bigr)^{-1}(0)\subset E$. Then $s_{\alpha^{(r)}}=\tau(-r\widetilde{\alpha})s_\alpha$ with $\tau(v^\prime): V\rightarrow V$ for $v^\prime\in V$ the translation
operator
$v\mapsto v+v^\prime$. In particular, the affine Weyl group $W^\bullet$ associated to $R^\bullet$, which is defined to be the group generated by the $s_{\alpha^{(r)}}$ ($\alpha\in R_0, r\in\mathbb{Z}$), is isomorphic to $W_0\ltimes \widetilde{Q}$.
We extend $R^\bullet$ to a (possibly nonreduced) irreducible affine
root system $R(D)$ by adding to $R^\bullet$ the affine roots $2\alpha^{(r)}$ for $r\in\mathbb{Z}$ and $\alpha\in R_0$ satisfying
$(\Lambda,\alpha^\vee)\subseteq 2\mathbb{Z}$. 

The basis $\Delta_0$ of $R_0$ extends to a basis $\Delta:=(\alpha_0,\alpha_1,\ldots,\alpha_n)$ of $R^\bullet$ and $R(D)$ with the additional simple affine root $\alpha_0=\mu_\psi c-\psi$ with $\psi\in R_0$ the highest (reps. highest short) root with respect to $\Delta_0$ if $\bullet=u$ (resp. $\bullet=t$). The basis $\Delta$ of $R^\bullet$
gives a decomposition $R^\bullet=R^{\bullet,+}\cup R^{\bullet,-}$ of $R^\bullet$ in positive and negative affine roots. We write $s_0:=s_{\alpha^{(0)}}$. Note that $s_0=\tau(\widetilde{\psi})s_\psi$. The affine Weyl group is a Coxeter group with Coxeter generators $s_0,s_1,\ldots,s_n$. The corresponding braid relations are
of the form
\[
s_is_j\cdots=s_js_i\cdots \qquad (m_{ij} \textup{factors on both sides})
\]
for $0\leq i\not=j\leq n$, where the $m_{ij}\in\{2,3,4,6,\infty\}$ can be read off from the associated Coxeter graph in the usual manner.

The {\it extended affine Weyl group} is defined to be $W:=W_0\ltimes\widetilde{\Lambda}$. It acts on $E$ by reflections and $\widetilde{\Lambda}$-translations.
Its contragredient action on the space of affine linear functions on $E$ preserves $R^\bullet$ and $R(D)$. The length of $w\in W$ is defined to be
\[
l(w):=\#(R^{\bullet,+}\cap w^{-1}R^{\bullet,-}).
\]
Then $\Omega:=\{w\in W \,\, | \,\, l(w)=0\}$ is an abelian subgroup of $W$, isomorphic to $\widetilde{\Lambda}/\widetilde{Q}$. Conjugation by $\Omega$ stabilises
$W^\bullet$ and 
\[
W\simeq \Omega\ltimes W^\bullet.
\]
More precisely $\Omega$, in its action on affine linear functions on $E$, is preserving
the set $\{\alpha_0,\alpha_2,\ldots,\alpha_n\}$ of simple affine roots. For $\varpi\in\Omega$ the induced bijection
on the index set $\{0,\ldots,n\}$ will also be denoted by $\varpi$, so $\varpi(\alpha_j)=\alpha_{\varpi(j)}$ for $0\leq j\leq n$ and $\varpi\in\Omega$.
Hence conjugation by $\Omega$ already stabilises the set $\{s_0,\ldots,s_n\}$ of simple reflections, 
\[
\varpi s_j \varpi^{-1}= s_{\varpi(j)},\qquad \varpi\in\Omega,\,\, j\in\{0,\ldots,n\}.
\]

Write 
\[\widetilde{\Lambda}_{min}^+:=\{\nu\in\widetilde{\Lambda} \,\, | \,\, (\nu,\widetilde{\alpha}^\vee)\in\{0,1\}\quad \forall\, \alpha\in R_0^+\}.
\]
For $\nu\in\widetilde{\Lambda}_{min}^+$ let 
$u(\nu)\in W$ be the element of minimal length in $\tau(\nu)W_0$. Define 
\[
v(\nu):=u(\nu)^{-1}\tau(\nu)\in W_0,
\] 
so that $\tau(\nu)=u(\nu)v(\nu)$ and
$l(\tau(\nu))=l(u(\nu))+l(v(\nu))$. Then it is known that
\[
\Omega=\{u(\nu)\}_{\nu\in\widetilde{\Lambda}_{min}^+}.
\]

There is an involution on the collection of all initial data $D=(R_0,\Delta_0,\bullet,\Lambda,\widetilde{\Lambda})$, mapping $D$ to 
\[
\widetilde{D}:=(\widetilde{R}_0,\widetilde{\Delta}_0,\bullet,\widetilde{\Lambda},\Lambda)
\]
where $\widetilde{\Delta}_0:=(\widetilde{\alpha}_1,\ldots,\widetilde{\alpha}_n)$. The (extended) affine Weyl group, affine root systems etc., will be denoted by adding tildes.
Note that for both the twisted and the untwisted case, the affine simple root associated to $\widetilde{R}^\bullet$ is given by $\widetilde{\alpha}_0=\mu_{\widetilde{\theta}}c-\widetilde{\theta}$ with $\theta\in R_0^+$ the highest short root  (see \cite[\S 2.2]{S2}). In particular, $\widetilde{s}_0=\tau(\theta)s_\theta$.

Finally we introduce multiplicity functions. A multiplicity function is a $W$-invariant function $\kappa: R(D)\rightarrow \mathbb{R}$. We write $\kappa_a\in\mathbb{R}$ for the
value of $\kappa$ at $a\in R(D)$ and $\kappa_j:=\kappa_{\alpha_j}$ for $0\leq j\leq n$. We furthermore write $\kappa_{2\alpha^{(r)}}:=\kappa_{\alpha^{(r)}}$ if $2\alpha^{(r)}\not\in R(D)$. For a finite root $\alpha\in R_0$ we define
the {\it Askey-Wilson parameters} associated to $\kappa$ and $\alpha$ by 
\[
\{a_\alpha,b_\alpha,c_\alpha,d_\alpha\}:=\{q^{\kappa_\alpha+\kappa_{2\alpha}}, -q^{\kappa_\alpha-\kappa_{2\alpha}}, q_\alpha q^{\kappa_{\alpha^{(1)}}+\kappa_{2\alpha^{(1)}}},
-q_\alpha q^{\kappa_{\alpha^{(1)}}-\kappa_{2\alpha^{(1)}}}\},
\]
where $q_\alpha:=q^{\mu_\alpha}$. We furthermore write $\kappa_j:=\kappa_{\alpha_j}$ for $0\leq j\leq n$.

Given a multiplicity function $\kappa$ on $R(D)$, there exists a unique multiplicity function $\widetilde{\kappa}$ on $R(\widetilde{D})$ satisfying
\[
\widetilde{\kappa}_{\widetilde{\alpha}}=\kappa_\alpha,\qquad \widetilde{\kappa}_{\widetilde{\alpha}^{(1)}}=\kappa_{2\alpha},\qquad 
\widetilde{\kappa}_{2\widetilde{\alpha}}=\kappa_{\alpha^{(1)}},\qquad
\widetilde{\kappa}_{2\widetilde{\alpha}^{(1)}}=\kappa_{2\alpha^{(1)}}
\]
for all $\alpha\in R_0$. We write the Askey-Wilson parameters associated to $\widetilde{\kappa}$ and $\widetilde{\alpha}\in\widetilde{R}_0$ by
$\{\widetilde{a}_\alpha,\widetilde{b}_\alpha,\widetilde{c}_\alpha,\widetilde{d}_\alpha\}$ (we abuse notation here by writing $\alpha$ as sublabel instead of $\widetilde{\alpha}$).
In terms of the original multiplicity function $\kappa$ they are explicitly given by
\[
\{\widetilde{a}_\alpha,\widetilde{b}_\alpha,\widetilde{c}_\alpha,\widetilde{d}_\alpha\}=
\{q^{\kappa_\alpha+\kappa_{\alpha^{(1)}}},-q^{\kappa_\alpha-\kappa_{\alpha^{(1)}}},q_\alpha q^{\kappa_{2\alpha}+\kappa_{2\alpha^{(1)}}},-q_\alpha q^{\kappa_{2\alpha}-
\kappa_{2\alpha^{(1)}}}\},
\]
hence they are obtained from the Askey-Wilson parameters $\{a_\alpha,b_\alpha,c_\alpha,d_\alpha\}$ by interchanging the values 
$\kappa_{2\alpha}$ and $\kappa_{\alpha^{(1)}}$. We also write $\widetilde{\kappa}_j:=\widetilde{\kappa}_{\widetilde{\alpha}_j}$ for $0\leq j\leq n$.
\subsection{Principal series representations}\label{psr}
\begin{defi}
The affine Hecke algebra $H^\bullet(\kappa)$ is the unital associative algebra over $\mathbb{C}$ generated by $T_j$ ($0\leq j\leq n$) with defining relations
\[
T_iT_jT_i\cdots=T_jT_iT_j\cdots \qquad (m_{ij} \textup{ {\it factors on both sides}})
\]
for $0\leq i\not=j\leq n$ and $(T_j-q^{-\kappa_j})(T_j+q^{\kappa_j})=0$ for $0\leq j\leq n$.
\end{defi}
We write $H_0(\kappa)$ for the unital subalgebra of $H^\bullet(\kappa)$ generated by $T_i$ ($1\leq i\leq n$). 

The abelian group $\Omega$ of length zero elements in $W$ acts by algebra automorphisms on $H^\bullet(\kappa)$, with $\varpi\in\Omega$ acting by
$T_j\mapsto T_{\varpi(j)}$ for $0\leq j\leq n$.  We write
\[
H(\kappa):=\Omega\ltimes H^\bullet(\kappa)
\]
for the corresponding crossed product algebra. We call it the {\it extended affine Hecke algebra} associated to $D$ and $\kappa$. The additional generators in $H(\kappa)$
parametrized by $\Omega$ are denoted by $T_{\varpi}$ ($\varpi\in\Omega$).
For $w\in W$ we write
\[
T_w:=T_{\varpi}T_{j_1}T_{j_2}\cdots T_{j_r}\in H(\kappa)
\]
if $w=\varpi s_{j_1}s_{j_2}\cdots s_{j_r}$  in $W$ with $\varpi\in\Omega$, $0\leq j\leq n$ and $l(w)=r$. It gives a well defined complex linear basis $\{T_w\}_{w\in W}$ of $H(\kappa)$.

Write
\[
\widetilde{\Lambda}^+:=\{\nu\in\widetilde{\Lambda} \,\, | \,\, (\nu,\widetilde{\alpha}^\vee)\geq 0\qquad \forall\, \alpha\in R_0^+\}
\]
and define for $\nu\in\widetilde{\Lambda}$,
\[
Y^\nu:=T_{\tau(\lambda)}T_{\tau(\mu)}^{-1}\in H(\kappa)
\]
if $\nu=\lambda-\mu$ with $\lambda,\mu\in\widetilde{\Lambda}^+$. The $Y^\nu$ are well defined and satisfy
\[
Y^\nu Y^{\nu^\prime}=Y^{\nu+\nu^\prime},\qquad Y^0=1
\]
for $\nu,\nu^\prime\in\widetilde{\Lambda}$. The extended affine Hecke algebra $H(\kappa)$ is generated by $H_0(\kappa)$ and the commuting elements $Y^\nu$ 
($\nu\in\widetilde{\Lambda}$). The additional relations characterizing $H(\kappa)$ in terms of the generators $T_i$ ($1\leq i\leq n$) and $Y^\nu$ ($\nu\in\widetilde{\Lambda}$)
are the Bernstein-Zelevinsky-Lusztig relations
\begin{equation}\label{BZL}
Y^\nu T_i=T_iY^{s_i\nu}+\Bigl(q^{\widetilde{\kappa}_i}\frac{(1-q^{-\widetilde{\kappa}_i-\widetilde{\kappa}_{2\widetilde{\alpha}_i}}Y^{-\widetilde{\alpha}_i})
(1+q^{-\widetilde{\kappa}_i+\widetilde{\kappa}_{2\widetilde{\alpha}_i}}Y^{-\widetilde{\alpha}_i})}{1-Y^{-2\widetilde{\alpha}_i}}-q^{-\widetilde{\kappa}_i}\Bigr)(Y^{s_i\nu}-Y^\nu)
\end{equation}
for $1\leq i\leq n$ and $\nu\in\widetilde{\Lambda}$ (see \cite[(3.4)]{SBook}).

Let $E_{\mathbb{C}}:=\mathbb{C}\otimes E$ be the complexification of $E$ and extend the scalar product $(\cdot,\cdot)$ to a complex bilinear form on $E_{\mathbb{C}}$.
For $I\subseteq\{1,\ldots,n\}$ write $H_I(\kappa)$ for the subalgebra of $H_I(\kappa)$ generated by $T_i$ ($i\in I$) and $Y^\nu$ ($\nu\in\widetilde{\Lambda}$).
Set
\[
E_{\mathbb{C},I}^\kappa:=\{\gamma\in E_{\mathbb{C}} \,\, | \,\, (\widetilde{\alpha}_i,\gamma)=\widetilde{\kappa}_{\widetilde{\alpha}_i}+\widetilde{\kappa}_{2\widetilde{\alpha}_i}
\quad \forall\, i\in I \}.
\]
\begin{lem}
Let $I\subseteq \{1,\ldots,n\}$ and $\gamma\in E_{\mathbb{C},I}^\kappa$. There exists a unique unital algebra map 
\[
\chi_{I,\gamma}: H_I(\kappa)\rightarrow\mathbb{C}
\]
satisfying
\begin{equation}\label{assignchi}
\begin{split}
\chi_{I,\gamma}(T_i)&=q^{-\kappa_i},\qquad \forall\,i\in I,\\
\chi_{I,\gamma}(Y^\nu)&=q^{-(\nu,\gamma)},\qquad \forall\,\nu\in\widetilde{\Lambda}.
\end{split}
\end{equation}
\end{lem}
\begin{proof}
The assumption $\gamma\in E_{\mathbb{C},I}^\kappa$ ensures that the assignment 
\eqref{assignchi} respects the crossing relations \eqref{BZL} (we use here that $\widetilde{\kappa}_i=\kappa_i$ for $i=1,\ldots,n$).
\end{proof}
\begin{defi}
Let $I\subseteq \{1,\ldots,n\}$ and $\gamma\in E_{\mathbb{C},I}^\kappa$. The induced representation
\[
M_I(\gamma):=\textup{Ind}_{H_I(\kappa)}^{H(\kappa)}\bigl(\mathbb{C}_{\chi_{I,\gamma}}\bigr)
\]
is called the principal series representation associated to $I$, with central character $q^{-\gamma}$. We write $\pi_\gamma^I$ for its representation map.
\end{defi}
Let $\gamma\in E_{\mathbb{C},I}^\kappa$. For $w\in W_0$ set
\[
v_w^I(\gamma):=T_w\otimes_{(H_I(\kappa),\chi_{I,\gamma})}1\in M_I(\gamma).
\]
The elements $v_w^I(\gamma)$ ($w\in W_0$) span $M_I(\gamma)$.

Let $W_{0,I}\subset W_0$ be the standard parabolic subgroup generated by the simple reflections $s_i$ ($i\in I$). Denote by $W_0^I$ the minimal
coset representatives of $W_0/W_{0,I}$. If $w=uv\in W_0$ with $u\in W_0^I$ and $v\in W_{0,I}$ then 
\[
v_w^I(\gamma)=\chi_{I,\gamma}(T_v)v_u^I(\gamma).
\]
The set $\{v_w^I(\gamma)\}_{w\in W_0^I}$ is a linear basis of the principal series module $M_I(\gamma)$. 

In case of minimal principal series ($I=\emptyset$) we omit $I$ from the notations. In particular we write
$\pi_\gamma: H(\kappa)\rightarrow\textup{End}_{\mathbb{C}}\bigl(M(\gamma)\bigr)$ for the corresponding representation map. 
Note that $\pi_\gamma$ is defined for all $\gamma\in E_{\mathbb{C}}$.

Let $\gamma\in E_{\mathbb{C},I}^\kappa$. There exists a unique surjective intertwiner 
$\phi_{I,\gamma}: M(\gamma)\twoheadrightarrow M_I(\gamma)$ of $H(\kappa)$-modules
satisfying for $w=uv\in W_0$ ($u\in W_0^I$, $v\in W_{0,I}$),
\[
\phi_{I,\gamma}(v_w(\gamma))=v_w^I(\gamma)=\chi_{I,\gamma}(T_v)v_u^I(\gamma).
\]
Clearly $\phi_{\emptyset,\gamma}=\textup{Id}_{M(\gamma)}$.

We need generic conditions on $\gamma\in E_{\mathbb{C},I}^\kappa$ to ensure that the principal series module $M_I(\gamma)$ is calibrated 
(i.e., simultaneously diagonalisable for the action of the commuting operators $Y^\nu$ ($\nu\in\widetilde{\Lambda}$)), 
cf. \cite[\S 2.5]{S} and references therein. To describe the results we first need to introduce 
intertwiners for minimal principal series modules.

Let $\gamma\in E_{\mathbb{C}}$ and let $e\in W_0$ be the neutral element. Set $A_e^{unn}(\gamma):=\textup{Id}_{M(\gamma)}$.
For $1\leq i\leq n$ define a linear map $A_{s_i}^{unn}(\gamma): M(s_i\gamma)\rightarrow M(\gamma)$ by
\begin{equation}\label{formulaAunn}
A_{s_i}^{unn}(\gamma)v_{\sigma}(s_i\gamma):=q^{-\widetilde{\kappa}_i}(1-q^{2(\widetilde{\alpha}_i,\gamma)})v_{\sigma s_i}(\gamma)
+(D_{\widetilde{\alpha}_i}(\gamma)-q^{-2\chi(-\sigma\widetilde{\alpha}_i)\widetilde{\kappa}_i}(1-q^{2(\widetilde{\alpha}_i,\gamma)}))v_\sigma(\gamma)
\end{equation}
for $\sigma\in W_0$, where $\chi(\widetilde{\alpha})=1$ if $\widetilde{\alpha}\in\widetilde{R}_0^-$ and $\chi(\widetilde{\alpha})=0$ if 
$\widetilde{\alpha}\in\widetilde{R}_0^+$, and with
\begin{equation*}
\begin{split}
D_{\widetilde{\alpha}}(\gamma):=&\bigl(1-q^{-\widetilde{\kappa}_{\widetilde{\alpha}}-\widetilde{\kappa}_{2\widetilde{\alpha}}+(\widetilde{\alpha},\gamma)}\bigr)
\bigl(1+q^{-\widetilde{\kappa}_{\widetilde{\alpha}}+\widetilde{\kappa}_{2\widetilde{\alpha}}+(\widetilde{\alpha},\gamma)}\bigr)\\
=&\bigl(1-\widetilde{a}_{\alpha}^{-1}q^{(\widetilde{\alpha},\gamma)}\bigr)\bigl(1-\widetilde{b}_{\alpha}^{-1}q^{(\widetilde{\alpha},\gamma)}\bigr)
\end{split}
\end{equation*}
for $\alpha\in R_0$. Then $A_{s_i}^{unn}(\gamma): M(s_i\gamma)\rightarrow M(\gamma)$ is an intertwiner of $H(\kappa)$-modules and 
\[
A_{s_i}^{unn}(\gamma)A_{s_i}^{unn}(s_i\gamma)=D_{s_i}(\gamma)D_{s_i}(s_i\gamma)\textup{Id}_{M(\gamma)}
\]
for all $1\leq i\leq n$.
More generally, if $\sigma:=s_{i_1}s_{i_2}\cdots s_{i_r}$
is a reduced expression for $\sigma\in W_0$ then
\[
A_\sigma^{unn}(\gamma):=A_{s_{i_1}}^{unn}(\gamma)A_{s_{i_2}}^{unn}(s_{i_1}\gamma)\cdots A_{s_{i_r}}^{unn}(s_{i_{r-1}}\cdots s_{i_2}s_{i_1}\gamma)
\]
is independent of the choice of reduced expression and defines an intertwiner
\[
A_{\sigma}^{unn}(\gamma): M(\sigma^{-1}\gamma)\rightarrow M(\gamma)\] 
of $H(\kappa)$-modules. It satisfies 
\[
A_{\sigma}^{unn}(\gamma)A_{\sigma^{-1}}^{unn}(\sigma^{-1}\gamma)=D_\sigma(\gamma)D_{\sigma^{-1}}(\sigma^{-1}\gamma)\textup{Id}_{M(\gamma)}
\]
with
\[
D_\sigma(\gamma):=\prod_{\widetilde{\alpha}\in \widetilde{R}_0^+\cap\sigma\widetilde{R}_0^-}D_{\widetilde{\alpha}}(\gamma).
\]
Note that
\begin{equation}\label{cocyclerestricted}
\begin{split}
A_{\sigma\tau}^{unn}(\gamma)&=A_\sigma^{unn}(\gamma)A_\tau^{unn}(\sigma^{-1}\gamma),\\
D_{\sigma\tau}^{unn}(\gamma)&=D_{\sigma}^{unn}(\gamma)D_{\tau}^{unn}(\sigma^{-1}\gamma)
\end{split}
\end{equation}
for $\sigma,\tau\in W_0$ satisfying $l(\sigma\tau)=l(\sigma)+l(\tau)$.

Set $\widetilde{R}_0^{I,\pm}:=\widetilde{R}_0^{\pm}\cap\textup{span}\{\widetilde{\alpha}_i\}_{i\in I}$. Since
\begin{equation}\label{W0Ialt}
W_0^I=\{\sigma\in W_0 \,\, | \,\, \sigma(\widetilde{R}_0^{I,+})\subseteq\widetilde{R}_0^+ \}
\end{equation}
we actually have 
\begin{equation}\label{ID}
D_{\sigma^{-1}}(\gamma)=\prod_{\widetilde{\alpha}\in(\widetilde{R}_0^+\setminus\widetilde{R}_0^{I,+})\cap \sigma^{-1}\widetilde{R}_0^-}
D_{\widetilde{\alpha}}(\gamma)\qquad \forall\,\sigma\in W_0^I.
\end{equation}

\begin{prop}\label{bb}
Assume that $\gamma\in E_{\mathbb{C},I}^\kappa$ satisfies $q^{2(\widetilde{\alpha},\gamma)}\not=1$ for all 
$\widetilde{\alpha}\in \widetilde{R}_0^+\setminus\widetilde{R}_0^{I,+}$.  Set for
$u\in W_0$,
\[
b_{u}^{unn,I}(\gamma):=\phi_{I,\gamma}\bigl(A_{u}^{unn}(\gamma)v_e(u^{-1}\gamma)\bigr)\in M_I(\gamma).
\]
Then 
\begin{enumerate}
\item For all $u\in W_0$ and all $\nu\in\widetilde{\Lambda}$ we have
\[
\pi_\gamma^I(Y^\nu)b_{u}^{unn,I}(\gamma)=q^{-(\nu,u^{-1}\gamma)}b_{u}^{unn,I}(\gamma);
\]
\item
$\{b_{\sigma^{-1}}^{unn,I}(\gamma)\}_{\sigma\in W_0^I}$ is a linear basis of $M_I(\gamma)$;
\item $b_{\sigma^{-1}}^{unn,I}(\gamma)=0$ for all $\sigma\in W_0\setminus W_0^I$.
\end{enumerate}
\end{prop}
\begin{proof}
(1) is trivial.\\
(2) 
By the definition of the unnormalized intertwiner we have for $\sigma\in W_0^I$
\[
b_{\sigma^{-1}}^{unn,I}(\gamma)=\Bigl(\prod_{\widetilde{\alpha}\in(\widetilde{R}_0^+\setminus\widetilde{R}_0^{I,+})\cap \sigma^{-1}\widetilde{R}_0^-}
q^{-\widetilde{\kappa}_{\widetilde{\alpha}}}(1-q^{2(\widetilde{\alpha},\gamma)})\Bigr)
v_\sigma^I(\gamma)+\sum_{\tau\in W_0^I: \tau<\sigma}c_{\sigma,\tau}v_\tau^I(\gamma)
\]
for certain coefficients $c_{\sigma,\tau}\in\mathbb{C}$, where $\leq$ is the Bruhat order. By the assumption that
$q^{2(\widetilde{\alpha},\gamma)}\not=1$ for all $\widetilde{\alpha}\in \widetilde{R}_0^+\setminus\widetilde{R}_0^{I,+}$ the leading coefficient
is nonzero.\\ 
(3) Fix $\sigma\not\in W_0^I$. Let $i\in I$ such that $\sigma=\sigma^\prime s_i$ and $l(\sigma)=l(\sigma^\prime)+1$ for some $i\in I$.
Then
\[
A_{\sigma^{-1}}^{unn}(\gamma)=A_{s_i}^{unn}(\gamma)A_{\sigma^\prime{}^{-1}}^{unn}(s_i\gamma).
\]
Hence
it suffices to show that $\phi_{I,\gamma}\circ A_{s_i}^{unn}(\gamma)$ is the zero map.

Since $\gamma\in E_{\mathbb{C},I}^\kappa$ we have $D_{\widetilde{\alpha}_i}(\gamma)=0$, hence by \eqref{formulaAunn},
\[
A_{s_i}^{unn}(\gamma)v_\tau(s_i\gamma)=q^{-\widetilde{\kappa}_i}(1-q^{2(\widetilde{\alpha}_i,\gamma)})v_{\tau s_i}(\gamma)
-q^{-2\chi(-\tau\widetilde{\alpha}_i)\widetilde{\kappa}_i}(1-q^{2(\widetilde{\alpha}_i,\gamma)})v_\tau(\gamma),\qquad \tau\in W_0.
\]

Suppose that $\tau\in W_0$ satisfies $\tau\widetilde{\alpha}_i\in\widetilde{R}_0^+$ ($\Leftrightarrow$ $l(\tau s_i)=l(\tau)+1$). Then 
$v_{\tau s_i}^I(\gamma)=q^{-\kappa_i}v_\tau^I(\gamma)=q^{-\widetilde{\kappa}_i}v_\tau^I(\gamma)$ and $\chi(-\tau\widetilde{\alpha}_i)=1$,
hence $\phi_{I,\gamma}(A_{s_i}^{unn}(\gamma)v_\tau(s_i\gamma))=0$.

Suppose that $\tau\in W_0$ satisfies $\tau\widetilde{\alpha}_i\in\widetilde{R}_0^-$ ($\Leftrightarrow$ $l(\tau s_i)=l(\tau)-1$). Then
$v_{\tau s_i}^I(\gamma)=q^{\kappa_i}v_\tau^I(\gamma)=q^{\widetilde{\kappa}_i}v_\tau^I(\gamma)$ and 
$\chi(-\tau\widetilde{\alpha}_i)=0$, hence also in this case $\phi_{I,\gamma}(A_{s_i}^{unn}(\gamma)v_\tau(s_i\gamma))=0$.

Hence indeed $\phi_{I,\gamma}\circ A_{s_i}^{unn}(\gamma)$ for $i\in I$ is the zero map. 
\end{proof}

\begin{defi}
Suppose that $\gamma\in E_{\mathbb{C}}$ satisfies $q^{2(\widetilde{\alpha},\gamma)}\not=1$ and $D_{\widetilde{\alpha}}(\gamma)\not=0$
for all $\widetilde{\alpha}\in \widetilde{R}_0^+$. Let $\sigma\in W_0$.
The normalised intertwiner $A_\sigma(\gamma): M(\sigma^{-1}\gamma)\rightarrow M(\gamma)$ is defined by
\[
A_\sigma(\gamma):=D_\sigma(\gamma)^{-1}A_\sigma^{unn}(\gamma),\qquad \sigma\in W_0.
\]
\end{defi}
Note that normalised intertwiners satisfy $A_e(\gamma)=\textup{Id}_{M(\gamma)}$ and 
\[
A_\sigma(\gamma)A_\tau(\sigma^{-1}\gamma)=A_{\sigma\tau}(\gamma)\qquad \forall\, \sigma,\tau\in W_0.
\]

Let $\gamma\in E_{\mathbb{C},I}^\kappa$ be generic. Note that $D_{\sigma^{-1}}(\gamma)=0$ if $\sigma\in W_0$ is such that 
$\widetilde{\alpha}_i\in\sigma^{-1}\widetilde{R}_0^-$ for some $i\in I$. Since this cannot happen if $\sigma\in W_0^I$
(see \eqref{W0Ialt} and \eqref{ID}), the normalised intertwiner $A_{\sigma^{-1}}(\gamma)$ is well defined for $\sigma\in W_0^I$.
Combined with Proposition \ref{bb} we obtain the following corollary.  
\begin{cor}\label{normcor}
Suppose that $\gamma\in E_{\mathbb{C},I}^\kappa$ satisfies the (generic) conditions that 
$q^{2(\widetilde{\alpha},\gamma)}\not=1$ and $D_{\widetilde{\alpha}}(\gamma)\not=0$
for all $\widetilde{\alpha}\in \widetilde{R}_0^+\setminus\widetilde{R}_0^{I,+}$. Then 
\begin{equation*}
\begin{split}
b_{\sigma^{-1}}^I(\gamma):=&D_{\sigma^{-1}}(\gamma)^{-1}b_{\sigma^{-1}}^{unn,I}(\gamma)\\
=&\phi_{I,\gamma}\bigl(A_{\sigma^{-1}}(\gamma)v_e(\sigma\gamma)\bigr)
\end{split}
\end{equation*}
is well defined for all $\sigma\in W_0^I$ and $\{b_{\sigma^{-1}}^I(\gamma)\}_{\sigma\in W_0^I}$ is a complex linear basis of 
$M_I(\gamma)$ satisfying
\[
\pi_\gamma^I(Y^\nu)b_{\sigma^{-1}}^I(\gamma)=q^{-(\nu,\sigma\gamma)}b_{\sigma^{-1}}^I(\gamma)
\qquad \forall\,\nu\in\widetilde{\Lambda},\,\,\forall\, \sigma\in W_0^I.
\]
\end{cor}

\subsection{Quantum affine KZ equations}
Define for $a\in R^\bullet$ the meromorphic function $c_a(\mathbf{z};\kappa)$ in $\mathbf{z}\in E_{\mathbb{C}}$ by
\[
c_a(\mathbf{z};\kappa):=\frac{(1-q^{-\kappa_a-\kappa_{2a}+a(\mathbf{z})})(1+q^{-\kappa_a+\kappa_{2a}+a(\mathbf{z})})}
{(1-q^{2a(\mathbf{z})})},
\]
where we view $a=\alpha^{(r)}=\mu_{\alpha}rc+\alpha$ as affine linear function on $E_{\mathbb{C}}$ by
\[
a(\mathbf{z}):=\mu_{\alpha}r+(\alpha,\mathbf{z}).
\]
We omit the $\kappa$-dependence from the notation if no confusion is possible.
We write $c_j(\mathbf{z}):=c_{\alpha^{(j)}}(\mathbf{z};\kappa)$ and
$\widetilde{c}_j(\mathbf{z}):=c_{\widetilde{\alpha}^{(j)}}(\mathbf{z};\widetilde{\kappa})$ for $0\leq j\leq n$. The following theorem generalizes Cherednik's \cite{ChBook} 
results on the Baxterization of affine Hecke algebra modules to the current setting (including the twisted, untwisted and nonreduced cases all at once). 
Its proof uses the affine intertwiners of the double affine Hecke algebra associated to the data $(D,\kappa)$
(see \cite{SBook}).
\begin{thm}
Let $\pi: H(\kappa)\rightarrow\textup{End}_{\mathbb{C}}(V)$ be a representation. The following formulas
\begin{equation*}
\begin{split}
\bigl(\nabla^V(s_j)f\bigr)(\mathbf{z})&:=\Bigl(\frac{q^{-\kappa_j}\pi(T_j)+c_j(\mathbf{z})-q^{-2\kappa_j}}{c_j(\mathbf{z})}\Bigr)f(s_j\mathbf{z}),\qquad 0\leq j\leq n,\\
\bigl(\nabla^V(\varpi)f\bigr)(\mathbf{z})&:=\pi(T_{\varpi})f(\varpi^{-1}\mathbf{z}),\qquad\qquad\qquad\qquad\quad\qquad \varpi\in\Omega
\end{split}
\end{equation*}
define a left $W$-action $\nabla^V$ on the space of $V$-valued meromorphic functions $f(\mathbf{z})$ in $\mathbf{z}\in E_{\mathbb{C}}$, where we use
the natural action of $W\simeq W_0\ltimes\widetilde{\Lambda}$ on $\mathbf{z}\in E_{\mathbb{C}}$ by reflections and $\widetilde{\Lambda}$-translations.
\end{thm}
Note that $(\nabla^V(w)f)(\mathbf{z})=C_w^V(\mathbf{z})f(w^{-1}\mathbf{z})$ for $w\in W$, with $C_w^V(\mathbf{z}): V\rightarrow V$ a linear map depending
meromorphically on $\mathbf{z}\in E_{\mathbb{C}}$ and satisfying the cocycle conditions $C_e^V(\mathbf{z})=\textup{Id}_V$ and
\[
C_{uv}^V(\mathbf{z})=C_u^V(\mathbf{z})C_v^V(u^{-1}\mathbf{z}),\qquad \forall\, u,v\in W.
\]
\begin{defi}\label{qaKZV}
For a given representation $\pi: H(\kappa)\rightarrow\textup{End}_{\mathbb{C}}(V)$ we say that a meromorphic $V$-valued function
$f(\mathbf{z})$ in $\mathbf{z}\in E_{\mathbb{C}}$ is a solution of the associated quantum affine KZ equations if
\[
\bigl(\nabla^V(\tau(\lambda))f\bigr)(\mathbf{z})=f(\mathbf{z})\qquad \forall\, \lambda\in\widetilde{\Lambda}.
\]
We write $\textup{Sol}_{KZ}(V)=\textup{Sol}_{KZ}(V;\kappa)$ for the solution space.
\end{defi}
Note that $\textup{Sol}_{KZ}(V)$ is $\nabla^V(W_0)$-stable. Note furthermore that the quantum affine KZ equations form a compatible system of difference equations.
Indeed, the quantum affine KZ equations can be written as
\begin{equation}\label{altqzKZ}
C_{\tau(\lambda)}^V(\mathbf{z})f(\mathbf{z}-\lambda)=f(\mathbf{z})\qquad \forall\,\lambda\in\widetilde{\Lambda}.
\end{equation}
It is apparent from this form of the quantum affine KZ equations that 
the cocycle values $C_{\tau(\lambda)}^V(\mathbf{z})$ play the role of transport operators. The compatibility of the difference equations \eqref{altqzKZ}
in terms of the transport operators reads
\[
C_{\tau(\lambda)}^V(\mathbf{z})C_{\tau(\mu)}^V(\mathbf{z}-\lambda)=C_{\tau(\lambda+\mu)}^V(\mathbf{z})=
C_{\tau(\mu)}^V(\mathbf{z})C_{\tau(\lambda)}^V(\mathbf{z}-\mu)\qquad \forall\, \lambda,\mu\in\widetilde{\Lambda},
\]
which is a direct consequence of the cocycle property of $\{C_u^V(\mathbf{z})\}_{u\in W}$.
Finally, note that an intertwiner $\psi: V\rightarrow V^\prime$ of $H(\kappa)$-modules gives rise to a $W_0$-intertwiner
$\textup{Sol}_{KZ}(V)\rightarrow \textup{Sol}_{KZ}(V^\prime)$. We will denote this map also by $\psi$. 

Explicit expressions for the cocycle $C_w^{M(\gamma)}(\mathbf{z})$ ($w\in W$) 
are given as follows.

\begin{prop}
Let $\gamma\in E_{\mathbb{C}}$. For $\sigma\in W_0$,
\begin{equation*}
\begin{split}
C_{s_0}^{M(\gamma)}(\mathbf{z})v_{\sigma}(\gamma)&=\frac{q^{(\widetilde{\psi},\sigma\gamma)}v_{s_{\psi}\sigma}(\gamma)}
{q^{\kappa_0}c_0(\mathbf{z})}+
\Bigl(\frac{c_0(\mathbf{z})-q^{-2\chi(\sigma^{-1}\widetilde{\psi})\kappa_0}}{c_0(\mathbf{z})}\Bigr)v_{\sigma}(\gamma),\\
C_{s_i}^{M(\gamma)}(\mathbf{z})v_{\sigma}(\gamma)&=\frac{v_{s_i\sigma}(\gamma)}{q^{\kappa_i}c_i(\mathbf{z})}+
\Bigl(\frac{c_i(\mathbf{z})-q^{-2\chi(-\sigma^{-1}\widetilde{\alpha}_i)\kappa_i}}{c_i(\mathbf{z})}\Bigr)v_\sigma(\gamma),\\
C_{\varpi}^{M(\gamma)}(\mathbf{z})v_{\sigma}(\gamma)&=q^{-(\varpi,w_0\sigma\gamma)}v_{v(\varpi)^{-1}\sigma}(\gamma)
\end{split}
\end{equation*}
for $1\leq i\leq n$ and $\varpi\in\Omega$, where $w_0\in W_0$ is the longest Weyl group element.
\end{prop}
\begin{proof}
This follows from the explicit formulas for the action of the affine Hecke algebra generators on the basis $\{v_\sigma(\gamma)\}_{\sigma\in W_0}$
of $M(\gamma)$ (see, e.g., \cite{S2} and references therein). 
\end{proof}
\subsection{Power series solutions}\label{pss}
We first construct a basis consisting of power series solutions of $\textup{Sol}_{KZ}(M(\gamma))$. To determine the 
natural normalisation we need to investigate the dependence on $\gamma$. We identify $M(\gamma)$ as complex vector space with 
$\mathcal{V}:=\bigoplus_{\sigma\in W_0}\mathbb{C}v_\sigma$ via the linear isomorphism $M(\gamma)\overset{\sim}
{\longrightarrow}\mathcal{V}$ satisfying $v_\sigma(\gamma)\mapsto v_\sigma$ 
for all $\sigma\in W_0$. Accordingly, we interpret the cocycle values $C_w^{M(\gamma)}(\mathbf{z}): M(\gamma)\rightarrow M(\gamma)$ ($w\in W$) as linear maps 
\[
C_w(\mathbf{z},\gamma): \mathcal{V}\rightarrow \mathcal{V}
\]
depending meromorphically on $(\mathbf{z},\gamma)\in E_{\mathbb{C}}\times E_{\mathbb{C}}$. In particular, for all $\sigma\in W_0$,
\begin{equation*}
\begin{split}
C_{s_0}(\mathbf{z},\gamma)v_{\sigma}&=\frac{q^{(\widetilde{\psi},\sigma\gamma)}v_{s_{\psi}\sigma}}
{q^{\kappa_0}c_0(\mathbf{z})}+
\Bigl(\frac{c_0(\mathbf{z})-q^{-2\chi(\sigma^{-1}\widetilde{\psi})\kappa_0}}{c_0(\mathbf{z})}\Bigr)v_{\sigma},\\
C_{s_i}(\mathbf{z},\gamma)v_{\sigma}&=\frac{v_{s_i\sigma}}{q^{\kappa_i}c_i(\mathbf{z})}+
\Bigl(\frac{c_i(\mathbf{z})-q^{-2\chi(-\sigma^{-1}\widetilde{\alpha}_i)\kappa_i}}{c_i(\mathbf{z})}\Bigr)v_\sigma,\\
C_{\varpi}(\mathbf{z},\gamma)v_{\sigma}&=q^{-(\varpi,w_0\sigma\gamma)}v_{v(\varpi)^{-1}\sigma}
\end{split}
\end{equation*}
for $1\leq i\leq n$ and $\varpi\in\Omega$.

For fixed $\gamma\in E_{\mathbb{C}}$, the solution space
$\textup{Sol}_{KZ}(M(\gamma))$ thus identifies with the space of $\mathcal{V}$-valued meromorphic functions $f(\mathbf{z},\gamma)$ in $\mathbf{z}\in E_{\mathbb{C}}$ satisfying
\begin{equation}\label{halfbqKZ}
C_{\tau(\lambda)}(\mathbf{z},\gamma)f(\mathbf{z}-\lambda,\gamma)=f(\mathbf{z},\gamma),\qquad \forall\, \lambda\in\widetilde{\Lambda}
\end{equation}
which form half of the bispectral quantum KZ equations (see \cite{vMS,vM,SAnn}); the other half is a compatible dual 
system of quantum affine KZ equations acting on $\gamma$.

Define the plane wave function by
\[
\mathcal{W}(\mathbf{z},\gamma):=q^{(\widetilde{\rho}+w_0\mathbf{z},\rho-\gamma)}
\]
where
\[
\rho:=\frac{1}{2}\sum_{\alpha\in R_0^+}(\kappa_\alpha+\kappa_{\alpha^{(1)}})\widetilde{\alpha}^\vee,\qquad
\widetilde{\rho}:=
\frac{1}{2}\sum_{\alpha\in R_0^+}(\widetilde{\kappa}_{\widetilde{\alpha}}+\widetilde{\kappa}_{\widetilde{\alpha}^{(1)}})\alpha^\vee.
\]
In addition define $\mathcal{S}(\mathbf{z})=\mathcal{S}(\mathbf{z};\kappa)$ by
\begin{equation}\label{S}
\mathcal{S}(\mathbf{z}):=\prod_{\alpha\in R_0^+}\bigl(q_\alpha^2a_\alpha^{-1}q^{-(\alpha,\mathbf{z})},
q_\alpha^2b_\alpha^{-1}q^{-(\alpha,\mathbf{z})}, q_\alpha^2c_\alpha^{-1}q^{-(\alpha,\mathbf{z})}, q_\alpha^2d_\alpha^{-1}q^{-(\alpha,\mathbf{z})};
q_\alpha^2\bigr)_{\infty}
\end{equation}
and $\widetilde{\mathcal{S}}(\gamma)=\widetilde{\mathcal{S}}(\gamma;\widetilde{\kappa})$ by
\[
\widetilde{\mathcal{S}}(\gamma):=
\prod_{\alpha\in R_0^+}\bigl(q_\alpha^2\widetilde{a}_\alpha^{-1}q^{-(\widetilde{\alpha},\gamma)},
q_\alpha^2\widetilde{b}_\alpha^{-1}q^{-(\widetilde{\alpha},\gamma)}, 
q_\alpha^2\widetilde{c}_\alpha^{-1}q^{-(\widetilde{\alpha},\gamma)}, q_\alpha^2\widetilde{d}_\alpha^{-1}q^{-(\widetilde{\alpha},\gamma)};
q_\alpha^2\bigr)_{\infty}.
\]
Note that $\widetilde{\rho}$ and $\widetilde{\mathcal{S}}$ are obtained from $\rho$ and $\mathcal{S}$ by replacing the initial data $(D,\kappa)$ by
$(\widetilde{D},\widetilde{\kappa})$.
The zero locus of the holomorphic function $\mathcal{S}(\mathbf{z})$ will give the singularities of the power series solutions of the quantum affine
KZ equations \eqref{halfbqKZ}.

Write $Q_+:=\mathbb{Z}_{\geq 0}R_0^+$.
The following theorem comprises the results of \cite[\S 3]{S2}. Special cases have been proven before in \cite{vM,vMS}.
\begin{thm}\label{Vref}
There exist unique $\mathcal{V}$-valued holomorphic functions $\Gamma_\mu(\gamma)$ in $\gamma\in E_{\mathbb{C}}$ ($\mu\in Q_+$)
such that 
\begin{enumerate}
\item $\Gamma_0(\gamma)=\Bigl(\prod_{\alpha\in R_0^+}\bigl(q_\alpha^2q^{-2(\widetilde{\alpha},\gamma)};q_\alpha^2\bigr)\Bigr)v_{w_0}$;
\item the $\mathcal{V}$-valued series
\[
\Psi(\mathbf{z},\gamma):=\sum_{\mu\in Q_+}\Gamma_\mu(\gamma)q^{-(\mu,\mathbf{z})}
\]
converges normally for $(\mathbf{z},\gamma)$ in compacta of $E_{\mathbb{C}}\times E_{\mathbb{C}}$;
\item the $\mathcal{V}$-valued meromorphic function $\Phi(\mathbf{z},\gamma)=\Phi(\mathbf{z},\gamma;\kappa)$ in 
$(\mathbf{z},\gamma)\in E_{\mathbb{C}}\times E_{\mathbb{C}}$ defined by
\[
\Phi(z,\gamma):=\frac{\mathcal{W}(z,\gamma)}{\mathcal{S}(\mathbf{z})\widetilde{\mathcal{S}}(\gamma)}\Psi(\mathbf{z},\gamma)
\]
satisfies 
\begin{equation}\label{bqKZPhi}
C_{\tau(\lambda)}(\mathbf{z},\gamma)\Phi(\mathbf{z}-\lambda,\gamma)=\Phi(\mathbf{z},\gamma)\qquad \forall\,\lambda\in\widetilde{\Lambda}.
\end{equation}
\end{enumerate}
If $\widetilde{\Phi}$ denotes the $\mathcal{V}$-valued function $\Phi$ with initial data $(D,\kappa)$ replaced by $(\widetilde{D},\widetilde{\kappa})$,
then
\begin{equation}\label{duality}
\widetilde{\Phi}(\gamma,\mathbf{z})=\Phi(\mathbf{z},\gamma).
\end{equation}
\end{thm}
\begin{rema}
It follows from \eqref{bqKZPhi} and \eqref{duality} that $\Phi(\mathbf{z},\gamma)$ is a solution of dual system of quantum KZ equations
acting as difference equations on $\gamma$.
The quantum affine KZ equations \eqref{halfbqKZ} together with the dual quantum KZ equations are the bispectral quantum KZ
equations from \cite{vM,vMS,S2}.
In \cite{S2} the power series solution 
$\Phi(\mathbf{z},\gamma)$ is first constructed and characterised as power series solution of the bispectral quantum KZ equations. 
The other properties,
such as the duality property \eqref{duality}, the singularities and the normalisation, are subsequently derived.
\end{rema}
\begin{defi}
We call $\gamma\in E_{\mathbb{C},I}^\kappa$ generic if
\begin{enumerate}
\item $q^{2(\widetilde{\alpha},\gamma)}\not=1$ for all $\alpha\in R_0^+\setminus R_0^{I,+}$;
\item $q^{(\widetilde{\alpha},\gamma)}\not\in\{\widetilde{a}_\alpha,\widetilde{b}_\alpha\}$ for all $\alpha\in R_0^+\setminus R_0^{I,+}$;
\item $q^{2(\widetilde{\alpha},\gamma)}\not\in q_\alpha^{2\mathbb{Z}_{>0}}$ for all $\alpha\in R_0\setminus R_0^{I,-}$;
\item $q^{(\widetilde{\alpha},\gamma)}\not\in\{\widetilde{a}_\alpha^{-1}q_\alpha^{2\mathbb{Z}_{>0}},
\widetilde{b}_\alpha^{-1}q_\alpha^{2\mathbb{Z}_{>0}},\widetilde{c}_\alpha^{-1}q_\alpha^{2\mathbb{Z}_{>0}},
\widetilde{d}_\alpha^{-1}q_\alpha^{2\mathbb{Z}_{>0}}\}$ for all $\alpha\in R_0\setminus R_0^{I,-}$.
\end{enumerate}
\end{defi}
Note that these are indeed generic conditions on $\gamma\in E_{\mathbb{C},I}^\kappa$ since
\[
E_{\mathbb{C},I}^\kappa\subseteq\{\gamma\in E_{\mathbb{C}} \,\, | \,\, q^{(\widetilde{\alpha}_i,\gamma)}=\widetilde{a}_i\quad \forall\, i\in I\}.
\]
The condition (1) is to ensure that $\{b_{\sigma^{-1}}^{unn,I}\}_{\sigma\in W_0^I}$ is a well defined basis of $M_I(\gamma)$. Condition
(2) ensures that $D_{\widetilde{\alpha}}(\gamma)\not=0$ for all $\alpha\in R_0^+\setminus R_0^{I,+}$, hence the unnormalised basis
$\{b_{\sigma^{-1}}^{unn,I}(\gamma)\}_{\sigma\in W_0^I}$ of $M_I(\gamma)$ can be turned into the normalised  basis
$\{b_{\sigma^{-1}}^I(\gamma)\}_{\sigma\in W_0^I}$ and the normalised intertwiners $A_{\sigma^{-1}}(\gamma)$ are well defined
for all $\sigma\in W_0^I$ (see Corollary \ref{normcor}). Condition (3) ensures that the leading coefficient
$\Gamma_0(\sigma\gamma)$ in the power series solution of the quantum affine KZ equations is nonzero for all $\sigma\in W_0^I$.
Finally, condition (4) ensures that $\widetilde{\mathcal{S}}(\sigma\gamma)\not=0$ for all $\sigma\in W_0^I$.

For $\gamma\in E_{\mathbb{C},I}^\kappa$ we write $\phi^{\mathcal{V}}_{I,\gamma}: \mathcal{V}\twoheadrightarrow M_I(\gamma)$
for the surjective linear map defined by
\[
\phi^{\mathcal{V}}_{I,\gamma}(v_w):=\chi_{I,\gamma}(T_v)v_u^I(\gamma)
\]
for $w=uv\in W_0$ with $u\in W_0^I$ and $v\in W_{0,I}$. We write $\phi^{\mathcal{V}}_\gamma:=\phi^{\mathcal{V}}_{\emptyset,\gamma}$.
Note that 
\[
\phi^{\mathcal{V}}_{I,\gamma}=\phi_{I,\gamma}\circ\phi^{\mathcal{V}}_\gamma.
\]
Write $F$ for the field of $\widetilde{\Lambda}$-translation invariant meromorphic functions on $E_{\mathbb{C}}$.
\begin{prop}\label{specI}
For generic $\gamma\in E_{\mathbb{C},I}^\kappa$ and $\sigma\in W_0^I$ define
\[
\Phi_{\sigma^{-1}}^I(\mathbf{z},\gamma):=\phi_{I,\gamma}\Bigl(A_{\sigma^{-1}}(\gamma)\phi^{\mathcal{V}}_{\sigma\gamma}\bigl(
\Phi(\mathbf{z},\sigma\gamma)\bigr)\Bigr).
\]
Then $\{\Phi_{\sigma^{-1}}^I(\cdot,\gamma)\}_{\sigma\in W_0^I}$ is a $F$-linear basis of $\textup{Sol}_{KZ}(M_I(\gamma))$. 
Furthermore,
\begin{equation}\label{psI}
\Phi_{\sigma^{-1}}^I(\mathbf{z},\gamma)=\frac{\mathcal{W}(\mathbf{z},\sigma\gamma)}{\mathcal{S}(\mathbf{z})
\widetilde{\mathcal{S}}(\sigma\gamma)}\sum_{\alpha\in Q_+}\Gamma_{\alpha,\sigma^{-1}}^I(\gamma)q^{-(\alpha,\mathbf{z})}
\end{equation}
with the $M_I(\gamma)$-valued power series converging normally for $\mathbf{z}$ in compact of $E_{\mathbb{C}}$ and with 
leading coefficient
\begin{equation}\label{normI}
\Gamma_{0,\sigma^{-1}}^I(\gamma)=\Bigl(\prod_{\alpha\in R_0^+}\bigl(q_\alpha^2q^{-2(\widetilde{\alpha},\sigma\gamma)};q_\alpha^2\bigr)_{\infty}
\Bigr)\pi_\gamma^I(T_{w_0})b_{\sigma^{-1}}^I(\gamma).
\end{equation}
\end{prop}
\begin{proof}
Since $\gamma\in E_{\mathbb{C},I}^\kappa$ is generic, the $\Phi_{\sigma^{-1}}^I(\cdot,\gamma)$ ($\sigma\in W_0^I$)
are well defined $M_I(\gamma)$-valued
meromorphic functions satisfying the quantum affine KZ equations 
\begin{equation}\label{qaKZI}
C_{\tau(\lambda)}^{M_I(\gamma)}(\mathbf{z})\Phi_{\sigma^{-1}}^I(\mathbf{z}-\lambda)=\Phi_{\sigma^{-1}}^I(\mathbf{z}),\qquad \lambda\in
\widetilde{\Lambda}.
\end{equation}
They are $F$-linearly independent
because their leading coefficients $\pi_\gamma^I(T_{w_0})b_{\sigma^{-1}}^I(\gamma)$ ($\sigma\in W_0^I$) form a linear basis of
$M_I(\gamma)$. The solutions $\Phi_{\sigma^{-1}}^I(\cdot,\gamma)$ ($\sigma\in W_0^I$) 
can be used to define a fundamental matrix solution of the quantum affine KZ equations \eqref{qaKZI}, from which
it follows that 
\[
\textup{Dim}_F\bigl(\textup{Sol}_{KZ}(M_I(\gamma))\bigr)=\textup{Dim}_{\mathbb{C}}(M_I(\gamma)).
\]
Hence the $\Phi_{\sigma^{-1}}^I(\cdot,\gamma)$ ($\sigma\in W_0^I$) form a $F$-linear basis of 
$\textup{Sol}_{KZ}(M_I(\gamma))$ (cf. \cite[\S 5.6]{vMS}, where the proof is discussed in detail for the initial data of type $A$
and $I=\emptyset$).
\end{proof}

\subsection{The connection problem and its solution}\label{connectionproblemsection}
Fix generic $\gamma\in E_{\mathbb{C},I}^\kappa$. There exists unique coefficients
\[
m_{\tau_1,\tau_2}^{I,\sigma}(\cdot,\gamma)\in F\qquad\qquad 
(\sigma\in W_0; \tau_1,\tau_2\in W_0^I)
\] 
such that
\begin{equation}\label{connectioncoefficients}
\nabla^{M_I(\gamma)}(\sigma)\Phi_{\tau_2^{-1}}^I(\cdot,\gamma)=\sum_{\tau_1\in W_0^I}
m_{\tau_1,\tau_2}^{I,\sigma}(\cdot,\gamma)\Phi_{\tau_1^{-1}}^I(\cdot,\gamma).
\end{equation}
\begin{defi}
We call the $\#W_0^I\times\#W_0^I$-matrices
\[
M^{I,\sigma}(\cdot,\gamma)=\bigl(m_{\tau_1,\tau_2}^{I,\sigma}(\cdot,\gamma)\bigr)_{\tau_1,\tau_2\in W_0^I},\qquad \sigma\in W_0
\]
with coefficients in $F$ the connection matrices of the quantum afffine KZ equations
associated to the initial data $(D,\kappa)$ and the principal series representation $M_I(\gamma)$.
\end{defi}
The following cocycle property is immediate:
\begin{equation*}
\begin{split}
M^{I,\sigma\sigma^\prime}(\mathbf{z},\gamma)&=M^\sigma(\mathbf{z},\gamma)
M^{\sigma^\prime}(\sigma^{-1}\mathbf{z},\gamma),\qquad \sigma,\sigma^\prime\in W_0,\\
M^{I,e}(\mathbf{z},\gamma)&=\textup{Id}.
\end{split}
\end{equation*}
To explicitly compute the connection cocycle $\{M^{I,\sigma}(\mathbf{z},\gamma)\}_{\sigma\in W_0}$ it thus suffices
to compute $M^{I,s_i}(\mathbf{z},\gamma)$ ($1\leq i\leq n$).

To state the result we first need to introduce some more notations. For $i\in \{1,\ldots,n\}$ let $i^*\in\{1,\ldots,n\}$ be the index
such that 
\[
-w_0(\alpha_i)=\alpha_{i^*}.
\]
For $\sigma,\sigma^\prime\in W_0$ let $\delta_{\sigma,\sigma^\prime}$  be equal to one if $\sigma=\sigma^\prime$ and equal to
zero otherwise.
For a finite root $\alpha\in R_0$ write
\[
\mathfrak{e}_\alpha(x,y):=
q^{-\frac{1}{2\mu_\alpha}(\kappa_\alpha+\kappa_{2\alpha}-x)(\kappa_\alpha+\kappa_{\alpha^{(1)}}-y)}
\frac{\theta\bigl(\widetilde{a}_\alpha q^y,\widetilde{b}_\alpha q^y, \widetilde{c}_\alpha q^y, d_\alpha q^{y-x}/\widetilde{a}_\alpha;
q_\alpha^2\bigr)}{\theta\bigl(q^{2y},d_\alpha q^{-x};q_\alpha^2\bigr)}.
\]
We write $\widetilde{\mathfrak{e}}_\alpha(x,y)$ ($\alpha\in R_0$) for its dual version,
\[
\widetilde{\mathfrak{e}}_\alpha(x,y):=
q^{-\frac{1}{2\mu_\alpha}(\kappa_\alpha+\kappa_{\alpha^{(1)}}-x)(\kappa_\alpha+\kappa_{2\alpha}-y)}
\frac{\theta\bigl(a_\alpha q^y,b_\alpha q^y,c_\alpha q^y,\widetilde{d}_\alpha q^{y-x}/a_\alpha;
q_\alpha^2\bigr)}{\theta\bigl(q^{2y},\widetilde{d}_\alpha q^{-x};q_\alpha^2\bigr)}.
\]
\begin{thm}\label{cqaKZ}
Fix generic $\gamma\in E_{\mathbb{C},I}^\kappa$ satisfying the additional conditions 
\begin{equation}\label{extrageneric}
q^{2(\widetilde{\beta},\gamma)}\not\in q_{\beta}^{2\mathbb{Z}}\qquad \forall\, 
\beta\in R_0. 
\end{equation}
Let $i\in\{1,\ldots,n\}$ and $\tau_2\in W_0^I$.
\begin{enumerate}
\item If $s_{i^*}\tau_2\not\in W_0^I$ then 
\[ m_{\tau_1,\tau_2}^{I,s_i}(\cdot,\gamma)=\delta_{\tau_1,\tau_2}\qquad \forall\,\tau_1\in W_0^I.
\]
\item If $s_{i*}\tau_2\in W_0^I$ then 
\[
m_{\tau_1,\tau_2}^{I,s_i}(\cdot,\gamma)\equiv 0\quad \textup{ if }\, \tau_1\not\in\{\tau_2,s_{i^*}\tau_2\}
\]
and
\begin{equation}\label{explicitC}
\begin{split}
m_{\tau_2,\tau_2}^{I,s_i}(\mathbf{z},\gamma)&=\frac{\mathfrak{e}_{\alpha_i}\bigl((\alpha_i,\mathbf{z}),
(\widetilde{\alpha}_{i^*},\tau_2\gamma)\bigr)-\widetilde{\mathfrak{e}}_{\alpha_i}\bigl((\widetilde{\alpha}_{i^*},\tau_2\gamma),
(\alpha_i,\mathbf{z})\bigr)}{\widetilde{\mathfrak{e}}_{\alpha_i}\bigl(
(\widetilde{\alpha}_{i^*},\tau_2\gamma),-(\alpha_i,\mathbf{z})\bigr)},\\
m_{s_{i^*}\tau_2,\tau_2}^{I,s_i}(\mathbf{z},\gamma)&=\frac{\mathfrak{e}_{\alpha_i}\bigl((\alpha_i,\mathbf{z}),
-(\widetilde{\alpha}_{i^*},\tau_2\gamma)\bigr)}{\widetilde{\mathfrak{e}}_{\alpha_i}\bigl(
(\widetilde{\alpha}_{i^*},\tau_2\gamma),-(\alpha_i,\mathbf{z})\bigr)}.
\end{split}
\end{equation}
\end{enumerate}
\end{thm}
\begin{rema}
An alternative formulation of Theorem \ref{cqaKZ} is as follows. Suppose that $\gamma\in E_{\mathbb{C},I}^\kappa$ is generic
and that \eqref{extrageneric} is valid. Let $i\in\{1,\ldots,n\}$ and $\tau_2\in W_0^I$. If $s_{i^*}\tau_2\not\in W_0^I$ then
\[
\nabla^{M_I(\gamma)}(s_i)\Phi_{\tau_2^{-1}}^I(\cdot,\gamma)=\Phi_{\tau_2^{-1}}^I(\cdot,\gamma),
\]
if $s_{i^*}\tau_2\in W_0^I$ then
\[
\nabla^{M_I(\gamma)}(s_i)\Phi_{\tau_2^{-1}}^I(\cdot,\gamma)=
m_{\tau_2,\tau_2}^{I,s_i}(\cdot,\gamma)\Phi_{\tau_2^{-1}}^{I}(\cdot,\gamma)+
m_{s_{i^*}\tau_2,\tau_2}^{I,s_i}(\cdot,\gamma)\Phi_{\tau_2^{-1}s_{i^*}}^I(\cdot,\gamma)
\]
with the coefficients $m_{\tau_2,\tau_2}^{I,s_i}(\cdot,\gamma), m_{s_{i^*}\tau_2,\tau_2}^{I,s_i}(\cdot,\gamma)\in F$
explicitly given by \eqref{explicitC}.
\end{rema}
\begin{proof}
In \cite[Thm. 1.6]{S2} the connection problem for the bispectral problem of the 
Ruij\-se\-naars-Macdonald-Koornwinder-Cherednik difference operators associated to $(D,\kappa)$ is solved. Through the difference
Cherednik-Matsuo correspondence it is equivalent to the solution of the connection problem for the bispectral quantum
KZ equations associated to $(D,\kappa)$ (cf. the proof of \cite[Thm. 1.5]{S2}). In our notations it gives the following.
View the normalised intertwiners $A_\tau(\xi)$ ($\tau\in W_0$) as linear operators on $\mathcal{V}$ depending meromorphically
on $\xi\in E_{\mathbb{C}}$ by identifying $\mathcal{V}\simeq M(\xi)$ as vector spaces through $\phi_\xi^{\mathcal{V}}$.
Then for $1\leq i\leq n$ and $\tau_2\in W_0$,
\begin{equation}\label{start}
\begin{split}
C_{s_i}(\mathbf{z},\xi)(A_{\tau_2^{-1}}(\xi)\Phi(s_i\mathbf{z},\tau_2\xi)\bigr)&=
n_{\tau_2,\tau_2}^{s_i}(\mathbf{z},\xi)\bigl(A_{\tau_2^{-1}}(\xi)\Phi(\mathbf{z},\tau_2\xi)\bigr)\\
&+n_{s_{i^*}\tau_2,\tau_2}^{s_i}(\mathbf{z},\xi)\bigl(A_{\tau_2^{-1}s_{i^*}}(\xi)\Phi(\mathbf{z},s_{i^*}\tau_2\xi)\bigr)
\end{split}
\end{equation}
as $\mathcal{V}$-valued meromorphic functions in $(\mathbf{z},\xi)\in E_{\mathbb{C}}\times E_{\mathbb{C}}$,
with $n_{\tau_1,\tau_2}^{s_i}$ (which is denoted by $m_{\tau_1^{-1},\tau_2^{-1}}^{s_i}$ in \cite{S2}) given explicitly by
\begin{equation*}
\begin{split}
n_{\tau_2,\tau_2}^{s_i}(\mathbf{z},\xi)&=\frac{\mathfrak{e}_{\alpha_i}\bigl((\alpha_i,\mathbf{z}),
(\widetilde{\alpha}_{i^*},\tau_2\xi)\bigr)-\widetilde{\mathfrak{e}}_{\alpha_i}\bigl((\widetilde{\alpha}_{i^*},\tau_2\xi),
(\alpha_i,\mathbf{z})\bigr)}{\widetilde{\mathfrak{e}}_{\alpha_i}\bigl(
(\widetilde{\alpha}_{i^*},\tau_2\xi),-(\alpha_i,\mathbf{z})\bigr)},\\
n_{s_{i^*}\tau_2,\tau_2}^{s_i}(\mathbf{z},\xi)&=\frac{\mathfrak{e}_{\alpha_i}\bigl((\alpha_i,\mathbf{z}),
-(\widetilde{\alpha}_{i^*},\tau_2\xi)\bigr)}{\widetilde{\mathfrak{e}}_{\alpha_i}\bigl(
(\widetilde{\alpha}_{i^*},\tau_2\xi),-(\alpha_i,\mathbf{z})\bigr)}.
\end{split}
\end{equation*}
By the additional generic conditions \eqref{extrageneric} on $\gamma\in E_{\mathbb{C},I}^\kappa$, we can specialise $\xi$ in the coefficients 
$n_{\tau_2,\tau_2}^{s_i}(\cdot,\xi)$ and $n_{s_{i^*}\tau_2,\tau_2}^{s_i}(\cdot,\xi)$ to $\gamma$ to get coefficients
$n_{\tau_2,\tau_2}^{s_i}(\cdot,\gamma), n_{s_{i^*}\tau_2,\tau_2}^{s_i}(\cdot,\gamma)\in F$. 
We consider now two cases. Let $\tau_2\in W_0^I$.\\
{\bf Case 1:} $s_{i^*}\tau_2\in W_0^I$.\\
Then in \eqref{start} we can specialise $\xi$ to $\gamma$ and apply $\phi_{I,\gamma}^{\mathcal{V}}$ to obtain
\[
C_{s_i}^{M_I(\gamma)}(\mathbf{z})\Phi_{\tau_2^{-1}}^I(s_i\mathbf{z},\gamma)=
n_{\tau_2,\tau_2}^{s_i}(\mathbf{z},\gamma)\Phi_{\tau_2^{-1}}^I(\mathbf{z},\gamma)+
n_{s_{i^*}\tau_2,\tau_2}^{s_i}(\mathbf{z},\gamma)\Phi_{\tau_2^{-1}s_{i^*}}^I(\mathbf{z},\gamma)
\]
as identity between $M_I(\gamma)$-valued meromorphic functions in $\mathbf{z}\in E_{\mathbb{C}}$,
cf. Proposition \ref{specI}. This proves (2).\\
{\bf Case 2:} $s_{i^*}\tau_2\not\in W_0^I$.\\
Then we rewrite \eqref{start} as
\begin{equation}\label{startnew}
\begin{split}
C_{s_i}(\mathbf{z},\xi)(A_{\tau_2^{-1}}(\xi)\Phi&(s_i\mathbf{z},\tau_2\xi)\bigr)=
n_{\tau_2,\tau_2}^{s_i}(\mathbf{z},\xi)\bigl(A_{\tau_2^{-1}}(\xi)\Phi(\mathbf{z},\tau_2\xi)\bigr)\\
&+G(\mathbf{z},\xi)
\bigl(\widetilde{\mathcal{S}}(s_{i^*}\tau_2\xi)A_{\tau_2^{-1}s_{i^*}}^{unn}(\xi)\Phi(\mathbf{z},s_{i^*}\tau_2\xi)\bigr)
\end{split}
\end{equation}
with
\begin{equation}\label{G}
G(\mathbf{z},\xi):=\frac{n_{s_{i^*}\tau_2,\tau_2}^{s_i}(\mathbf{z},\xi)}{\widetilde{\mathcal{S}}(s_{i^*}\tau_2\xi)D_{\tau_2^{-1}s_{i^*}}(\xi)}.
\end{equation}
We first show that $G(\mathbf{z},\xi)$ is regular at $\xi=\gamma$. 
Since $\tau_2\in W_0^I$ and $s_{i^*}\tau_2\not\in W_0^I$ we have $\alpha_{i^*}=\tau_2(\beta)$ for some $\beta\in R_0^{I,+}$.
So $s_\beta\in W_{0,I}$, hence
\[
l(s_{i^*}\tau_2)=l(\tau_2s_\beta)=l(\tau_2)+l(s_\beta).
\]
Consequently $\beta=\alpha_j$ for some $j\in I$ and $l(s_{i^*}\tau_2)=l(\tau_2s_j)=l(\tau_2)+1$. 
Hence we have
\begin{equation*}
\begin{split}
D_{\tau_2^{-1}s_{i^*}}(\xi)&=D_{\tau_2^{-1}}(\xi)D_{s_{i^*}}(\tau_2\xi)\\
&=D_{\tau_2^{-1}}(\xi)D_{\alpha_j}(\xi)\\
&=D_{\tau_2^{-1}}(\xi)(1-\widetilde{a}_j^{-1}q^{(\widetilde{\alpha}_j,\xi)})(1-\widetilde{b}_j^{-1}q^{(\widetilde{\alpha}_j,\xi)})
\end{split}
\end{equation*}
by \eqref{cocyclerestricted}, and $D_{\tau_2^{-1}}(\gamma)\not=0$ in view of the generic conditions on $\gamma\in E_{\mathbb{C},I}^\kappa$. 
Note though that $(1-\widetilde{a}_j^{-1}q^{(\widetilde{\alpha}_j,\xi)})$ will be zero when specialising $\xi$ to $\gamma\in E_{\mathbb{C},I}^\kappa$.
We will show though in a moment that this factor also occurs as factor in the connection coefficient $n_{s_{i^*}\tau_2,\tau_2}^{s_i}(\mathbf{z},\xi)$, hence they
will cancel each other out in \eqref{G}. 

Observe that
\begin{equation*}
\begin{split}
\widetilde{\mathcal{S}}(s_{i^*}\tau_2\xi)=&\widetilde{\mathcal{S}}^{rem}(\xi)\bigl(q_j^2\widetilde{a}_j^{-1}q^{(\widetilde{\alpha}_j,\xi)},
q_j^2\widetilde{b}_j^{-1}q^{(\widetilde{\alpha}_j,\xi)}, q_j^2\widetilde{c}_j^{-1}q^{(\widetilde{\alpha}_j,\xi)},
q_j^2\widetilde{d}_j^{-1}q^{(\widetilde{\alpha}_j,\xi)};q_j^2\bigr)_{\infty},\\
\widetilde{\mathcal{S}}^{rem}(\xi):=&\prod_{\alpha\in \tau_2^{-1}(R_0^+)\setminus\{\alpha_j\}}
\bigl(q_\alpha^2\widetilde{a}_{\alpha}^{-1}q^{-(\widetilde{\alpha},\xi)}, q_\alpha^2\widetilde{b}_{\alpha}^{-1}q^{-(\widetilde{\alpha},\xi)},
q_\alpha^2\widetilde{c}_{\alpha}^{-1}q^{-(\widetilde{\alpha},\xi)}, q_\alpha^2\widetilde{d}_{\alpha}^{-1}q^{-(\widetilde{\alpha},\xi)};q_\alpha^2\bigr)_{\infty}
\end{split}
\end{equation*}
and $\widetilde{\mathcal{S}}^{rem}(\gamma)\not=0$ by the generic conditions on $\gamma\in E_{\mathbb{C},I}^\kappa$. 
The connection coefficient $n_{s_{i^*}\tau_2,\tau_2}^{s_i}(\mathbf{z},\xi)$ decouples,
\[
n_{s_{i^*}\tau_2,\tau_2}^{s_i}(\mathbf{z},\xi)=\frac{\widetilde{p}_{\widetilde{\alpha}_j}(-(\widetilde{\alpha}_j,\xi))}{p_{\alpha_i}(-(\alpha_i,\mathbf{z}))}
\]
with 
\[
p_\alpha(x):=\frac{\theta\bigl(a_\alpha q^x,b_\alpha q^x,c_\alpha q^x, d_\alpha q^x;q_\alpha^2\bigr)}
{\theta\bigl(q^{2x};q_\alpha^2\bigr)}q^{\frac{1}{\mu_\alpha}(\kappa_\alpha+\kappa_{\alpha^{(1)}})x}
\]
and with $\widetilde{p}_{\widetilde{\alpha}}(x)$ its dual version
\[
\widetilde{p}_{\widetilde{\alpha}}(x):=\frac{\theta\bigl(\widetilde{a}_\alpha q^x,\widetilde{b}_\alpha q^x,\widetilde{c}_\alpha q^x, \widetilde{d}_\alpha q^x;q_\alpha^2\bigr)}
{\theta\bigl(q^{2x};q_\alpha^2\bigr)}q^{\frac{1}{\mu_\alpha}(\widetilde{\kappa}_{\widetilde{\alpha}}+
\widetilde{\kappa}_{\widetilde{\alpha}^{(1)}})x}
\]
(the decoupling formula in Remark \ref{decoupleremark}{\bf (ii)} is a special case). 
Combining these observations
we conclude that the coefficient $G(\mathbf{z},\xi)$ (see \eqref{G}) can be rewritten as
\begin{equation*}
\begin{split}
G(\mathbf{z},\xi)&=\frac{\widetilde{a}_j\widetilde{b}_j}
{p_{\alpha_i}(-(\alpha_i,\mathbf{z}))D_{\tau_2^{-1}}(\xi)\widetilde{\mathcal{S}}^{rem}(\xi)
\theta\bigl(q^{-2(\widetilde{\alpha}_j,\xi)};q_j^2\bigr)}
q^{-2(\widetilde{\alpha}_j,\xi)-\frac{1}{\mu_j}(\widetilde{\kappa}_{\widetilde{\alpha}_j}+\widetilde{\kappa}_{\widetilde{\alpha}_j^{(1)}})
(\widetilde{\alpha}_j,\xi)}\\
&\times \bigl(q_j^2\widetilde{a}_jq^{-(\widetilde{\alpha}_j,\xi)},
q_j^2\widetilde{b}_jq^{-(\widetilde{\alpha}_j,\xi)},
\widetilde{c}_jq^{-(\widetilde{\alpha}_j,\xi)},
\widetilde{d}_jq^{-(\widetilde{\alpha}_j,\xi)};q_j^2\bigr)_{\infty}.
\end{split}
\end{equation*}
Hence $G(\mathbf{z},\xi)$ is regular at $\xi=\gamma$.

Returning to the second term in the right hand side of \eqref{startnew}, we next observe that
\begin{equation}\label{secondpart}
\begin{split}
\widetilde{\mathcal{S}}(s_{i^*}\tau_2\xi)A_{\tau_2^{-1}s_{i^*}}^{unn}(\xi)&\Phi(\mathbf{z},s_{i^*}\tau_2\xi)=\\
&=\frac{\mathcal{W}(\mathbf{z},s_{i^*}\tau_2\xi)}{\mathcal{S}(\mathbf{z})}
A_{s_j}^{unn}(\xi)\Bigl(A_{\tau_2^{-1}}^{unn}(s_j\xi)\sum_{\mu\in Q_+}\Gamma_\mu(s_{i^*}\tau_2\xi)q^{-(\mu,\mathbf{z})}\Bigr)
\end{split}
\end{equation}
since 
\[
A_{\tau_2^{-1}s_{i^*}}^{unn}(\xi)=A_{s_j\tau_2^{-1}}^{unn}(\xi)=A_{s_j}^{unn}(\xi)A_{\tau_2^{-1}}^{unn}(s_j\xi)
\]
by \eqref{cocyclerestricted}. Now $\phi_{I,\gamma}^{\mathcal{V}}\circ A_{s_j}^{unn}(\gamma): \mathcal{V}\rightarrow M_I(\gamma)$
is the zero map (see the proof of Proposition \ref{bb}), hence 
we conclude that
\[
\phi_{I,\gamma}^{\mathcal{V}}\Bigl(\widetilde{\mathcal{S}}(s_{i^*}\tau_2\gamma)A_{\tau_2^{-1}s_{i^*}}^{unn}(\gamma)\Phi(\mathbf{z},s_{i^*}\tau_2\gamma)
\Bigr)\equiv 0
\]
as $M_I(\gamma)$-valued holomorphic function in $\mathbf{z}\in E_{\mathbb{C}}$. 

Thus specialising $\xi$ to $\gamma$ in \eqref{startnew} and applying
the map $\phi_{I,\gamma}^{\mathcal{V}}$ we obtain
\[
C_{s_i}^{M_I(\gamma)}(\mathbf{z})\Phi_{\tau_2^{-1}}^I(s_i\mathbf{z},\gamma)=n_{\tau_2,\tau_2}^{s_i}(\mathbf{z},\gamma)
\Phi_{\tau_2^{-1}}^I(\mathbf{z},\gamma).
\]
It thus remains to show that 
\begin{equation}\label{startnewnew}
n_{\tau_2,\tau_2}^{s_i}(\mathbf{z},\gamma)\equiv 1
\end{equation}
as meromorphic function in $\mathbf{z}\in E_{\mathbb{C}}$. To prove \eqref{startnewnew} first note that 
\[
n_{\tau_2,\tau_2}^{s_i}(\mathbf{z},\gamma)=
\frac{\mathfrak{e}_{\alpha_i}((\alpha_i,\mathbf{z}),(\widetilde{\alpha}_j,\gamma))-
\widetilde{\mathfrak{e}}_{\alpha_i}((\widetilde{\alpha}_j,\gamma),(\alpha_i,\mathbf{z}))}
{\widetilde{\mathfrak{e}}_{\alpha_i}((\widetilde{\alpha}_{j},\gamma),-(\alpha_i,\mathbf{z}))}
\]
since $\alpha_j=\tau_2^{-1}(\alpha_{i^*})$. But $j\in I$, hence $\gamma\in E_{\mathbb{C},I}^\kappa$ and
$\alpha_{i}\in W_0(\alpha_j)$
implies that 
\[
(\widetilde{\alpha}_j,\gamma)=\widetilde{\kappa}_{\widetilde{\alpha}_j}+
\widetilde{\kappa}_{2\widetilde{\alpha}_j}=
\widetilde{\kappa}_{\widetilde{\alpha}_i}+
\widetilde{\kappa}_{2\widetilde{\alpha}_i}.
\]
Consequently
\begin{equation}\label{stepn}
n_{\tau_2,\tau_2}^{s_i}(\mathbf{z},\gamma)=
\frac{\mathfrak{e}_{\alpha_i}((\alpha_i,\mathbf{z}),\widetilde{\kappa}_{\widetilde{\alpha}_i}+
\widetilde{\kappa}_{2\widetilde{\alpha}_i})-
\widetilde{\mathfrak{e}}_{\alpha_i}(\widetilde{\kappa}_{\widetilde{\alpha}_i}+
\widetilde{\kappa}_{2\widetilde{\alpha}_i},(\alpha_i,\mathbf{z}))}
{\widetilde{\mathfrak{e}}_{\alpha_i}(\widetilde{\kappa}_{\widetilde{\alpha}_i}+
\widetilde{\kappa}_{2\widetilde{\alpha}_i},-(\alpha_i,\mathbf{z}))}.
\end{equation}
By a direct computation the numerator in \eqref{stepn} can be rewritten as
\begin{equation}\label{quart}
\begin{split}
\mathfrak{e}_{\alpha_i}&((\alpha_i,\mathbf{z}),\widetilde{\kappa}_{\widetilde{\alpha}_i}+
\widetilde{\kappa}_{2\widetilde{\alpha}_i})-
\widetilde{\mathfrak{e}}_{\alpha_i}(\widetilde{\kappa}_{\widetilde{\alpha}_i}+
\widetilde{\kappa}_{2\widetilde{\alpha}_i},(\alpha_i,\mathbf{z}))=\\
&=\frac{\theta\bigl(\widetilde{d}_i/\widetilde{a}_i,q^{2(\alpha_i,\mathbf{z})},\widetilde{a}_i\widetilde{b}_i,
\widetilde{a}_i\widetilde{c}_i;q_i^2\bigr)-
\theta\bigl(a_iq^{(\alpha_i,\mathbf{z})},b_iq^{(\alpha_i,\mathbf{z})}, c_iq^{(\alpha_i,\mathbf{z})},
\widetilde{d}_iq^{(\alpha_i,\mathbf{z})}/a_i\widetilde{a}_i;q_i^2\bigr)}
{\theta\bigl(q^{2(\alpha_i,\mathbf{z})},\widetilde{d}_i/\widetilde{a}_i;q_i^2\bigr)}.
\end{split}
\end{equation}
Recall the well known theta function identity
\begin{equation}\label{thetaid}
\theta(x\nu,x/\nu,\lambda\mu,\mu/\lambda;q)-\theta(x\lambda,x/\lambda,\mu\nu,\mu/\nu;q)=
-\frac{\mu}{\lambda}\theta(x\mu, x/\mu, \lambda\nu, \lambda/\nu;q).
\end{equation}
For an exhaustive discussion of \eqref{thetaid} and its origin, see \cite{Kotheta}. Using
\eqref{thetaid} with $(q,x,\nu,\lambda,\mu)$ specialised to
\[
(q_i^2,c_i\sqrt{a_ib_i},\sqrt{a_i^{-1}b_i},\sqrt{a_ib_i}q^{-(\alpha_i,\mathbf{z})},
\sqrt{a_ib_i}q^{(\alpha_i,\mathbf{z})})
\]
we can rewrite \eqref{quart} as
\begin{equation}\label{quart2}
\begin{split}
\mathfrak{e}_{\alpha_i}&((\alpha_i,\mathbf{z}),\widetilde{\kappa}_{\widetilde{\alpha}_i}+
\widetilde{\kappa}_{2\widetilde{\alpha}_i})-
\widetilde{\mathfrak{e}}_{\alpha_i}(\widetilde{\kappa}_{\widetilde{\alpha}_i}+
\widetilde{\kappa}_{2\widetilde{\alpha}_i},(\alpha_i,\mathbf{z}))=\\
&=-\frac{\theta\bigl(a_iq^{-(\alpha_i,\mathbf{z})}, b_iq^{-(\alpha_i,\mathbf{z})}, c_iq^{-(\alpha_i,\mathbf{z})}
\widetilde{d}_iq^{-(\alpha_i,\mathbf{z})}/a_i\widetilde{a}_i;q_i^2\bigr)}
{\theta\bigl(q^{2(\alpha_i,\mathbf{z})},\widetilde{d}_i/\widetilde{a}_i;q_i^2\bigr)}
q^{2(\alpha_i,\mathbf{z})}.
\end{split}
\end{equation}
A direct computation shows that the right hand side of \eqref{quart2} equals 
$\widetilde{\mathfrak{e}}_{\alpha_i}(\widetilde{\kappa}_{\widetilde{\alpha}_i}+
\widetilde{\kappa}_{2\widetilde{\alpha}_i},-(\alpha_i,\mathbf{z}))$. Hence
\[
n_{\tau_2,\tau_2}^{s_i}(\mathbf{z},\gamma)\equiv 1
\]
as meromorphic function in $\mathbf{z}\in E_{\mathbb{C}}$, as desired.
\end{proof}

\subsection{The link to the spin-$\frac{1}{2}$ XXZ boundary qKZ equations}\label{link}
We specialise the results of the previous subsection to the initial data $(D,\kappa)$ corresponding to the Koornwinder (or $C^\vee C_n$) case
of the Cherednik-Macdonald theory. The five tuple $D=(R_0,\Delta_0,\bullet,\Lambda,\widetilde{\Lambda})$ is then given as follows (we choose our
notations to ensure that it matches with the notations in Subsection \ref{spinsec}).
Take $E=\mathbb{R}^n$ ($n\geq 2$) with the standard orthonormal basis $\{e_i\}_{i=1}^n$. The root system $R_0$ is taken to be
the root system of type $B_n$ realised by
\[
R_0=\{\pm(e_i\pm e_j)\}_{1\leq i<j\leq n}\cup \{\pm e_i\}_{i=1}^n.
\]
As ordered basis we take
\[
\Delta_0=(\alpha_1,\ldots,\alpha_{n-1},\alpha_n)=(e_1-e_2,\ldots,e_{n-1}-e_n,e_n).
\]
We are considering the twisted theory
$\bullet=t$, hence $\mu_\alpha=|\alpha|^2/2$ and $\widetilde{\alpha}=\alpha$ for all $\alpha\in R_0$. The lattices are
taken to be as small as possible,
\[
\Lambda=\mathbb{Z}^n=\widetilde{\Lambda}
\]
(the root lattice of $R_0$). The affine simple root is $\alpha=\frac{c}{2}-e_1$. 
The Weyl group $W_0=\langle s_1,\ldots,s_{n-1},s_n\rangle$ is isomorphic to the hyperoctahedral group $S_n\ltimes (\pm 1)^n$.
Furthermore we have  $s_0=\tau(e_1)s_{e_1}$ and $W=\langle s_0,\ldots,s_n\rangle\simeq W_0\ltimes\mathbb{Z}^n$. 

The affine root system $R(D)$ associated to $D$ has five $W$-orbits
\[
W\alpha_0, W(2\alpha_0), W\alpha_1, W\alpha_n, W(2\alpha_n).
\]
A multiplicity function 
$\kappa: R(D)\rightarrow\mathbb{R}$ thus is determined by five values, which we denote by
\begin{equation}\label{linkparameters}
(\zeta,\upsilon,\zeta^\prime,\upsilon^\prime,\kappa):=
(\kappa_{\alpha_n}, \kappa_{2\alpha_n}, \kappa_{\alpha_0}, \kappa_{2\alpha_0},\kappa_{\alpha_1}).
\end{equation}
The dual parameters are then obtained by interchanging $\upsilon$
and $\zeta^\prime$. The Askey-Wilson parameters $\{a_\alpha,b_\alpha,c_\alpha,d_\alpha\}$
are $\{a,b,c,d\}$ (see \eqref{abcd}) if $\alpha\in R_0$ is short. If $\alpha\in R_0$ is long then
$\{a_\alpha,b_\alpha,c_\alpha,d_\alpha\}=\{q^{2\kappa},-1,q^{1+2\kappa},-q\}$.

We take $I:=\{1,\ldots,n-1\}$, so that
\[
W_{0,I}=\langle s_1,\ldots,s_{n-1}\rangle\simeq S_n.
\]
Note that
\begin{equation*}
E_{\mathbb{C},I}^\kappa
=\{(\xi+(n-1)\kappa,\xi+(n-3)\kappa,\ldots,\xi+(1-n)\kappa)\,\, | \,\, \xi\in\mathbb{C} \}.
\end{equation*}
Recall the spin representation $\bigl(\pi^{sp}_{(\xi)},(\mathbb{C}^2)^{\otimes n}\bigr)$ (see Proposition \ref{spinprop}) and 
the elements $w_{\underline{\epsilon}}\in W_0$ ($\underline{\epsilon}\in\{\pm 1\}^n$) defined by \eqref{wepsilon},
which are the minimal coset representatives $W_0^I$ of $W_0/W_{0,I}$ (see Lemma \ref{wepsilonlemma}). By
\cite[Prop. 3.5]{SV} we have 
\begin{equation}\label{isomorfismexplicit}
\bigl(\pi_\gamma^I,M_I(\gamma)\bigr)\simeq \bigl(\pi^{sp}_{(\xi)},(\mathbb{C}^2)^{\otimes n}\bigr)  
\end{equation}
as $H(\kappa)$-modules, with $\gamma\in E_{\mathbb{C},I}^\kappa$ given by
\begin{equation}\label{gammaspecial}
\gamma:=(\xi+(n-1)\kappa,\xi+(n-3)\kappa,\ldots,\xi+(1-n)\kappa).
\end{equation}
The isomorphism $M_I(\gamma)\overset{\sim}{\longrightarrow} \bigl(\mathbb{C}^2\bigr)^{\otimes n}$ of $H(\kappa)$-modules
maps the basis element $v_{w_{\underline{\epsilon}}}(\gamma)\in M_I(\gamma)$ to $v_{\underline{\epsilon}}\in
\bigl(\mathbb{C}^2\bigr)^{\otimes n}$ for all $\underline{\epsilon}\in\{\pm 1\}^n$ (see Lemma \ref{wepsilonlemma}).

In the remainder of this subsection we assume that $\gamma\in E_{\mathbb{C},I}^\kappa$ (see \eqref{gammaspecial}) 
is generic and satisfies \eqref{extrageneric} (this gives generic conditions
on $\xi$ and $\kappa$).
The isomorphism $M_I(\gamma)\overset{\sim}{\longrightarrow} (\mathbb{C}^2)^{\otimes n}$ allows to transfer 
the definitions and results of this section to the setting of the XXZ spin chain (see Section \ref{section2}) 
in the following way. 

Let $\underline{\epsilon},\underline{\epsilon^\prime}\in\{\pm 1\}^n$ and $\sigma\in W_0$, then
\begin{enumerate}
\item The vector $b_{w_{\underline{\epsilon}}^{-1}}^I(\gamma)\in M_I(\gamma)$ (Corollary \ref{normcor} and Proposition \ref{bb}) is mapped to 
$b_{\underline{\epsilon}}\in\bigl(\mathbb{C}^2\bigr)^{\otimes n}$ (Lemma \ref{bepsilon});
\item The quantum affine KZ equations associated to $M_I(\gamma)$ (Definition \ref{qaKZV}) 
become the boundary quantum KZ equations \eqref{REFLqKZXXZ} associated to the XXZ spin-$\frac{1}{2}$ chain;
\item The power series solution $\Phi_{w_{\underline{\epsilon}}^{-1}}^I(\cdot,\gamma)$ (Proposition \ref{specI}) of the quantum affine
KZ equations 
become the power series
solution $\Phi_{\underline{\epsilon}}(\cdot)$ (Theorem \ref{asymptoticspin}) of the boundary qKZ equation \eqref{REFLqKZXXZ}.
In particular, Theorem \ref{asymptoticspin} is a special case of Proposition \ref{specI}. 
\item The connection coefficient $m_{w_{\underline{\epsilon}},w_{\underline{\epsilon}^\prime}}^{I,\sigma}(\cdot,\gamma)$
(Subsection \ref{connectionproblemsection}) equals 
$M_{cm;\underline{\epsilon},\underline{\epsilon}^\prime}^\sigma(\cdot,\xi)$ (Subsection \ref{connXXZsec}).
\end{enumerate}
We are now ready to prove Theorem \ref{Mwc} as a consequence of the explicit expressions of the connection matrix
of the quantum affine KZ equation for principal series modules (see Theorem \ref{cqaKZ}).\\

\noindent
{\it Proof of Theorem \ref{Mwc}.}
Recall that $M_{cm}^\sigma(\mathbf{z},\xi)$ is the linear operator on $\bigl(\mathbb{C}^2\bigr)^{\otimes n}$ defined by
\begin{equation*}
\begin{split}
M_{cm}^\sigma(\mathbf{z},\xi)v_{\underline{\epsilon}^\prime}:=&\sum_{\underline{\epsilon}\in\{\pm 1\}^n}
M_{cm;\underline{\epsilon},\underline{\epsilon}^\prime}^\sigma(\mathbf{z},\xi)v_{\underline{\epsilon}}\\
=&\sum_{\underline{\epsilon}\in\{\pm 1\}^n}m_{w_{\underline{\epsilon}}, w_{\underline{\epsilon}^\prime}}^{I,\sigma}(\mathbf{z},\gamma)
v_{\underline{\epsilon}}.
\end{split}
\end{equation*}
Fix $i\in \{1,\ldots,n\}$ and note that $i^*=i$. By Theorem \ref{cqaKZ} we have
\begin{equation}\label{no1}
M_{cm}^{s_i}(\mathbf{z},\xi)v_{\underline{\epsilon}^\prime}=v_{\underline{\epsilon}^\prime}\quad \hbox{ if }\,\, s_iw_{\underline{\epsilon}^\prime}\not\in
W_0^I,
\end{equation}
and
\begin{equation}\label{yes1}
M_{cm}^{s_i}(\mathbf{z},\xi)v_{\underline{\epsilon}^\prime}=
m_{w_{\underline{\epsilon}^\prime}, w_{\underline{\epsilon}^\prime}}^{I,s_i}(\mathbf{z},\gamma)v_{\underline{\epsilon}^\prime}+
m_{w_{\underline{\epsilon}^{\prime\prime}},w_{\underline{\epsilon}^\prime}}^{I,s_i}(\mathbf{z},\gamma)v_{\underline{\epsilon}^{\prime\prime}}
\quad \hbox{ if }\,\, s_iw_{\underline{\epsilon}^\prime}\in W_0^I,
\end{equation}
where $\underline{\epsilon}^{\prime\prime}\in\{\pm 1\}^n$ is such that $s_iw_{\underline{\epsilon}^\prime}=
w_{\underline{\epsilon}^{\prime\prime}}$. Let $W_0\simeq S_n\ltimes (\pm 1)^n$ act on $\{\pm 1\}^n$ by permutations and sign changes.
We need the following simple lemma.
\begin{lem}\label{leminbetween}
Let $\underline{\epsilon}\in\{\pm 1\}^n$ and $1\leq j<n$.\\
{\bf (i)} $s_jw_{\underline{\epsilon}}\in W_0^I$ $\Leftrightarrow$ $(\epsilon_j,\epsilon_{j+1})\in
\{(+1,-1), (-1,+1)\}$, in which case $s_jw_{\underline{\epsilon}}=w_{s_j\underline{\epsilon}}$.\\
{\bf (ii)} $s_nw_{\underline{\epsilon}}=w_{s_n\underline{\epsilon}}\in W_0^I$.
\end{lem}
\begin{proof}
{\bf (ii)} is trivial in view of the explicit expression \eqref{wepsilon} of $w_{\underline{\epsilon}}\in W_0^I$. The explicit expression
\eqref{wepsilon} of $w_{\underline{\epsilon}}$ immediately
implies that $s_jw_{\underline{\epsilon}}=
w_{s_j\underline{\epsilon}}$ if $(\epsilon_j,\epsilon_{j+1})\in \{(+1,-1), (-1,+1)\}$. 
Suppose that $(\epsilon_j,\epsilon_{j+1})=(+1,+1)$. Then 
\[
s_jw_{\underline{\epsilon}}=w_{\underline{\epsilon}}s_k
\]
for some $1\leq k\leq j<n$ since
\begin{equation*}
s_j(s_rs_{r+1}\cdots s_{n-1}s_n)=
\begin{cases}
(s_rs_{r+1}\cdots s_{n-1}s_n)s_j\quad &\hbox{ if }\,\, j+1<r\leq n,\\
(s_rs_{r+1}\cdots s_{n-1}s_n)s_{j-1}\quad &\hbox{ if }\,\, 1\leq r<j.\\
\end{cases}
\end{equation*}
Hence $s_jw_{\underline{\epsilon}}\not\in W_0^I$.

Suppose that $(\epsilon_j,\epsilon_{j+1})=(-1,-1)$. Then we need in addition the equality
\[
s_j\bigl((s_{j+1}s_{j+2}\cdots s_n)(s_js_{j+1}\cdots s_n)\bigr)=
\bigl((s_{j+1}s_{j+2}\cdots s_n)(s_js_{j+1}\cdots s_n)\bigr)s_{n-1}
\]
in $W_0$ to conclude that $s_jw_{\underline{\epsilon}}=w_{\underline{\epsilon}}s_k$ for some $1\leq k<n$, hence
$s_jw_{\underline{\epsilon}}\not\in W_0^I$.
\end{proof}
Continuing with the proof of Theorem \ref{Mwc}, we conclude from \eqref{no1}, \eqref{yes1}, Lemma \ref{leminbetween} 
and the explicit expressions for the connection coefficients 
$m_{w_{\underline{\epsilon}}, w_{\underline{\epsilon}^\prime}}^{I,s_i}(\cdot,\gamma)$
(see Theorem \ref{cqaKZ}) that for all $\underline{\epsilon}^\prime=(\epsilon_1^\prime,\ldots,\epsilon_n^\prime)\in\{\pm 1\}^n$,
\[
M_{cm}^{s_i}(\mathbf{z},\xi)v_{\underline{\epsilon}^\prime}=R^{(i)}((\alpha_i,\mathbf{z}), (\alpha_i,w_{\underline{\epsilon}^{\prime (i)}}\gamma))_{i,i+1}
v_{\underline{\epsilon}^\prime}
\]
if $1\leq i<n$
with $\underline{\epsilon}^{\prime (i)}:=(\epsilon_1^\prime,\ldots,\epsilon_{i-1}^\prime,1,-1,\epsilon_{i+1}^\prime,\ldots,\epsilon_n^\prime)$
and 
\begin{equation*}
R^{(i)}(x,y):=
\left(\begin{matrix} 1 & 0 & 0 & 0\\
0 & \frac{\mathfrak{e}_{\alpha_i}(x,y)-\widetilde{\mathfrak{e}}_{\alpha_i}(y,x)}{\widetilde{\mathfrak{e}}_{\alpha_i}(y,-x)}
& \frac{\mathfrak{e}_{\alpha_i}(x,y)}{\widetilde{\mathfrak{e}}_{\alpha_i}(-y,-x)} & 0\\
0 & \frac{\mathfrak{e}_{\alpha_i}(x,-y)}{\widetilde{\mathfrak{e}}_{\alpha_i}(y,-x)} 
& \frac{\mathfrak{e}_{\alpha_i}(x,-y)-\widetilde{\mathfrak{e}}_{\alpha_i}(-y,x)}{\widetilde{\mathfrak{e}}_{\alpha_i}(-y,-x)} & 0\\
0 & 0 & 0 & 1
\end{matrix}\right),
\end{equation*}
and
\[
M_{cm}^{s_n}(\mathbf{z},\xi)v_{\underline{\epsilon}^\prime}=K((\alpha_n,\mathbf{z}),(\alpha_n,w_{\underline{\epsilon}^{\prime (n)}}\gamma))_n
v_{\underline{\epsilon}^\prime}
\]
with $\underline{\epsilon}^{\prime (n)}:=(\epsilon_1^\prime,\ldots,\epsilon_{n-1}^\prime,+1)$ and
\begin{equation*}
K(x,y):=
\left(\begin{matrix}
\frac{\mathfrak{e}_{\alpha_n}(x,y)-\widetilde{\mathfrak{e}}_{\alpha_n}(y,x)}{\widetilde{\mathfrak{e}}_{\alpha_n}(y,-x)} & 
\frac{\mathfrak{e}_{\alpha_n}(x,y)}{\widetilde{\mathfrak{e}}_{\alpha_n}(-y,-x)}\\
\frac{\mathfrak{e}_{\alpha_n}(x,-y)}{\widetilde{\mathfrak{e}}_{\alpha_n}(y,-x)} & 
\frac{\mathfrak{e}_{\alpha_n}(x,-y)-\widetilde{\mathfrak{e}}_{\alpha_n}(-y,x)}{\widetilde{\mathfrak{e}}_{\alpha_n}(-y,-x)}
\end{matrix}
\right).
\end{equation*}

If $1\leq i<n$ then $\bigl(\mathbb{Z},\alpha_i^\vee\bigr)=\mathbb{Z}$, so \cite[(1.9) \& Prop. 1.7]{S2} gives
\begin{equation*}
\begin{split}
\frac{\mathfrak{e}_{\alpha_i}(x,y)-\widetilde{\mathfrak{e}}_{\alpha_i}(y,x)}{\widetilde{\mathfrak{e}}_{\alpha_i}(y,-x)}&=
\frac{\theta\bigl(q^{2\kappa},q^{y-x};q\bigr)}{\theta\bigl(q^y,q^{2\kappa-x};q\bigr)}q^{(2\kappa-y)x}=B_{cm}(x,-y),\\
\frac{\mathfrak{e}_{\alpha_i}(x,-y)}{\widetilde{\mathfrak{e}}_{\alpha_i}(y,-x)}&=\frac{\theta\bigl(q^{2\kappa-y},q^{-x};q\bigr)}
{\theta\bigl(q^{2\kappa-x},q^{-y};q\bigr)}q^{2\kappa(x-y)}=A_{cm}(x,y),
\end{split}
\end{equation*}
 hence
\[
M_{cm}^{s_i}(\mathbf{z},\xi)v_{\underline{\epsilon}^\prime}=
P_{i,i+1}R_{cm}(z_i-z_{i+1},(\alpha_i,w_{\underline{\epsilon}^{\prime (i)}}\gamma))_{i,i+1}
v_{\underline{\epsilon}^\prime}
\]
(see Theorem \ref{Mwc} for the definitions of 
$A_{cm}(x,y), B_{cm}(x,y)$ and $R_{cm}(x,y)$). For $i=n$ a direct computation using \eqref{linkparameters} gives
\[
{\mathfrak{e}}_{\alpha_n}(x,y)=\mathcal{C}(x,y)
\]
(see \eqref{calC} for the definition of $\mathcal{C}(x,y)$), hence
\[
M_{cm}^{s_n}(\mathbf{z},\xi)v_{\underline{\epsilon}^\prime}=
K_{cm}(z_n,(\alpha_n,w_{\underline{\epsilon}^{\prime (n)}}\gamma))_nv_{\underline{\epsilon}^\prime}
\]
(see Theorem \ref{Mwc} for the definition of $K_{cm}(x,y)$). So the following lemma completes the proof of 
Theorem \ref{Mwc}.
\begin{lem}
Let $1\leq i<n$, $\underline{\epsilon}\in\{\pm 1\}^n$ and $\gamma\in E_{\mathbb{C},I}^\kappa$ (see \eqref{gammaspecial}).
\begin{enumerate}
\item[{\bf (1)}] If $(\epsilon_i,\epsilon_{i+1})\in\{(+1,-1),(-1,+1)\}$ then 
\[
(\alpha_i,w_{\underline{\epsilon}^{(i)}}\gamma)=2\xi-2\kappa(\epsilon_1+\epsilon_2+\cdots+\epsilon_{i-1}).
\]
\item[{\bf (2)}] $(\alpha_n,w_{\underline{\epsilon}^{(n)}}\gamma)=\xi-\kappa(\epsilon_1+\epsilon_2+\cdots+\epsilon_{n-1})$.
\end{enumerate}
\end{lem}
\begin{proof}
Write $V_{\underline{\epsilon}}:=\{j\in \{1,\ldots,n\} \,\, | \,\, \epsilon_j=-1\}
=\{j_1,\ldots,j_r\}$ with $j_1<j_2<\cdots<j_r$.\\
{\bf (1)} Suppose that $(\epsilon_i,\epsilon_{i+1})=(+1,-1)$ and let $1\leq t\leq r$ be the index such that $j_t=i+1$.
Using the expression
\[
w_{\underline{\epsilon}}=(s_{j_r}s_{j_r+1}\cdots s_n)\cdots (s_{j_2}s_{j_2+1}\cdots s_n)(s_{j_1}s_{j_1+1}\cdots s_n)
\]
(see \eqref{wepsilon}) we obtain
\[
w_{\underline{\epsilon}^{(i)}}^{-1}(\alpha_i)=w_{\underline{\epsilon}}^{-1}(\alpha_i)=
e_{i-t+1}+e_{n-t+1}.
\]
Then 
\begin{equation*}
\begin{split}
(\alpha_i,w_{\underline{\epsilon}^{(i)}}\gamma)=\gamma_{i-t+1}+\gamma_{n-t+1}&=
2\xi+(4t-2i-2)\kappa\\
&=2\xi+2\kappa(t-1)-2\kappa(i-t)\\
&=2\xi-2\kappa(\epsilon_1+\cdots+\epsilon_{i-1}).
\end{split}
\end{equation*}

If $(\epsilon_i,\epsilon_{i+1})=(-1,+1)$ then we apply the result of the previous paragraph to $\underline{\epsilon}^\prime:=
s_i\underline{\epsilon}$ to conclude that
\begin{equation*}
\begin{split}
(\alpha_i,w_{\underline{\epsilon}^{(i)}}\gamma)&=
(\alpha_i,w_{\underline{\epsilon}^{\prime (i)}}\gamma)\\
&=2\xi-2\kappa(\epsilon_1^\prime+\cdots+\epsilon_{i-1}^\prime)\\
&=2\xi-2\kappa(\epsilon_1+\cdots+\epsilon_{i-1}).
\end{split}
\end{equation*}
{\bf (2)} Suppose that $\epsilon_n=+1$, so that $j_r<n$. Then
\[
w_{\underline{\epsilon}^{(n)}}^{-1}(\alpha_n)=w_{\underline{\epsilon}}^{-1}(\alpha_n)=e_{n-r}
\]
hence
\begin{equation*}
\begin{split}
(\alpha_n,w_{\underline{\epsilon}^{(n)}}\gamma)&=\gamma_{n-r}\\
&=\xi+\kappa(1-n+2r)=\xi+\kappa r-\kappa(n-1-r)\\
&=\xi-\kappa(\epsilon_1+\epsilon_2+\cdots+\epsilon_{n-1}).
\end{split}
\end{equation*}
If $\epsilon_n=-1$ then we apply this result to $\underline{\epsilon}^\prime:=s_n\underline{\epsilon}$ to reach the same conclusion,
\begin{equation*}
\begin{split}
(\alpha_n,w_{\underline{\epsilon}^{(n)}}\gamma)&=(\alpha_n,w_{\underline{\epsilon}^{\prime (n)}}\gamma)\\
&=\xi-\kappa(\epsilon_1^\prime+\epsilon_2^\prime+\cdots+\epsilon_{n-1}^\prime)\\
&=\xi-\kappa(\epsilon_1+\epsilon_2+\cdots+\epsilon_{n-1}).
\end{split}
\end{equation*}
This completes the proof of the lemma.
\end{proof}


\end{document}